\documentclass{article}
\usepackage[T1]{fontenc}
\usepackage{amsthm, fullpage}
\usepackage{amsmath}
\usepackage{graphicx}
\usepackage{mathrsfs}
\usepackage{amssymb}
\usepackage{amsfonts}
\usepackage{eucal}
\usepackage[all]{xy}
\usepackage[frenchb, english]{babel}
\usepackage{pslatex}
\usepackage{hyperref}
\providecommand{\scr}{\mathcal}

\newtheorem{prop}{Proposition}[subsection]

\newtheorem{lem}[prop]{Lemma}
\newtheorem*{lemm*}{Lemma}

\theoremstyle{definition}

\newtheorem{empt}[prop]{}
\newtheorem{dfn}[prop]{Definition}
\newtheorem{rem}[prop]{Remark} 
\newtheorem{ntn}[prop]{Notation} 
\newtheorem{exs}[prop]{Examples} 
\newtheorem{ex}[prop]{Example} 
\newtheorem*{rem*}{Remark}

\theoremstyle{thm}
\newtheorem{thm}[prop]{Theorem}
\newtheorem*{thm*}{Theorem}
\newtheorem*{lem*}{Lemma}

\newtheorem*{cor*}{Corollary}
\newtheorem*{prop*}{Proposition}

\theoremstyle{dfn}
\newtheorem*{dfn*}{Definition}

\numberwithin{equation}{prop}

\newcommand{\riso}{ \overset{\sim}{\longrightarrow}\, }

\newcommand{\Spf}{\mathrm{Spf}\,}

\renewcommand{\sp}{\mathrm{sp}}

\newcommand{\FF}{{\mathcal{F}}}
\newcommand{\B}{{\mathcal{B}}}

\newcommand{\E}{{\mathcal{E}}}
\newcommand{\G}{{\mathcal{G}}}

\newcommand{\D}{{\mathcal{D}}}

\newcommand{\PP}{{\mathcal{P}}}
\newcommand{\QQ}{{\mathcal{Q}}}
\renewcommand{\O}{{\mathcal{O}}}

\newcommand{\V}{\mathcal{V}}
\newcommand{\W}{\mathcal{W}}

\newcommand{\Y}{\mathcal{Y}}
\newcommand{\ZZ}{\mathcal{Z}}
\newcommand{\X}{\mathfrak{X}}

\newcommand{\U}{\mathfrak{U}}

\renewcommand{\P}{\mathbb{P}}

\newcommand{\DD}{\mathbb{D}}
\renewcommand{\L}{\mathbb{L}}
\newcommand{\R}{\mathbb{R}}
\newcommand{\Q}{\mathbb{Q}}
\newcommand{\Z}{\mathbb{Z}}
\newcommand{\N}{\mathbb{N}}

\newcommand{\hdag}{  \phantom{}{^{\dag} }    }

\begin{document}

\title{Unipotent monodromy and arithmetic $\D$-modules}
\author{Daniel Caro} 

\date{}

\maketitle

\begin{abstract}
In the framework of Berthelot's theory of arithmetic $\D$-modules, 
we introduce the notion of arithmetic $\D$-modules having potentially unipotent monodromy. 
For example, from Kedlaya's semistable reduction theorem, 
the overconvergent isocrystals with Frobenius structure have potentially unipotent monodromy. 
We construct some coefficients stable under Grothendieck's six operations, 
containing overconvergent isocrystals with Frobenius structure
and whose objects have potentially unipotent monodromy. 

On the other hand, 
we introduce the notion of arithmetic $\D$-modules having quasi-unipotent monodromy. 
These objects are overholonomic, 
contain the isocrystals having potentially unipotent monodromy 
and are stable under Grothendieck's six operations and under base change.

\end{abstract}

\tableofcontents

\bigskip 

Let $\V$ be a complete discrete valued ring of mixed characteristic $(0,p)$, 
$K$ its field of fractions,  
$k$ its residue field which is supposed to be perfect. 
Let $X$ be a smooth quasi-projective $k$-variety, $Z:=X - Y$ be a simple normal crossing divisor of $X$, let 
$Z = \cup _{i=1} ^{r} Z _i$ be the decomposition of $Z$ into irreducible components
and $\prod _{i=1} ^{r} \Sigma _i$ be a subset of $(\Z _p / \Z) ^{r}$. 
Let $E $ be an overconvergent isocrystal on $(Y,X/K)$.
Atsushi Shiho defined  (see the end of the definition \cite[3.9]{Shiho-logextension}) the notion of 
overconvergent isocrystals having $\prod _{i=1} ^{r} \Sigma _i$-unipotent monodromy. 
When for any $i$ the sets $\Sigma _i$ are equal to a set $\Sigma$, we will say for short 
``overconvergent isocrystals having $ \Sigma$-unipotent monodromy''. 
When $ \Sigma _i= \{ 0\}$, we retrieve Kedlaya's unipotent monodromy (see 
\cite{kedlaya-semistableI}). 
Without non Liouvilleness conditions, these isocrystals have no finite cohomology and in particular
they are not (over)holonomic. 

From now, suppose $\Sigma$ is a subgroup of $ \Z _p / \Z$ with $p$-adically non Liouville numbers, then 
it follows from \cite[2.3.13]{caro-Tsuzuki} that 
overconvergent isocrystals 
 with $ \Sigma$-unipotent monodromy 
are overholonomic. 
We recall that with Frobenius structures,  
we already know the stability under Grothendieck's six operations of the overholonomicity (see \cite{caro-Tsuzuki})
but, without Frobenius structures, the stability under tensor products is still an open question. 
In this paper, in the framework of Berthelot's arithmetic $\D$-modules,
we introduce the notion of 
arithmetic $\D$-modules having potentially $\Sigma$-unipotent monodromy (see \ref{dfn-potSigma}). 
For example an overconvergent isocrystal has potentially $\Sigma$-unipotent monodromy
if by definition it gets
$ \Sigma$-unipotent monodromy after some generically etale alteration. 
By descent from this alteration, we check that they are overholonomic. 
Moreover, a reformulation of Kedlaya's semistable reduction theorem 
(see \cite{kedlaya-semistableI,kedlaya-semistableII,kedlaya-semistableIII,kedlaya-semistableIV})
is that overconvergent isocrystals with some Frobenius structure have  potentially unipotent monodromy.
We also introduce the notion of 
arithmetic $\D$-modules having quasi-$\Sigma$-unipotent monodromy (see \ref{dfnqu}). 
These coefficients are overholonomic, 
contain isocrystals having potentially $ \Sigma$-unipotent monodromy 
and are stable under Grothendieck's six operations and base change. 
Finally, we construct some coefficients stable under Grothendieck's six operations and base change, 
containing overconvergent isocrystals with Frobenius structure
and whose objects have potentially unipotent monodromy.

\subsection*{Notation and convention}

In the rest of the paper, we fix $\Sigma$ a subgroup of $ \Z _p / \Z$ with $p$-adically non Liouville numbers. 
We choose $\tau \colon \Z _p / \Z \to \Z _p$ a section of the canonical extension 
$\Z _p \to \Z _p / \Z$ such that $\tau (0) =0$.

We also fix $\V$ a complete discrete valued ring of mixed characteristic $(0,p)$. 
We denote by  $K$ the field of fractions of $\V$,  
$k$ its residue field which is supposed to be perfect. 
A formal scheme over $\V$ means a formal scheme for the $p$-adic topology.
By convention, our formal schemes are always separated. 
The special fiber of a formal scheme over $\V$ will be denoted by the corresponding capital roman letter.

\section{Stability under cohomological operations of data of coefficients}
\subsection{Data of coefficients}

\begin{dfn}
\label{DVR}
We denote by $\mathrm{DVR}  (\V)$ the full subcategory of the category of 
$\V$-algebras 
whose objects are 
complete discrete valued rings of mixed characteristic $(0,p)$ with perfect residue field. 
\end{dfn}

\begin{empt}
\label{t-structure-coh}
Let $\W$ be an object of $\mathrm{DVR}  (\V)$, and 
$\X$ be a smooth formal $\W$-scheme.
If there is no possible confusion (some confusion might arise if for example we do know that $\V \to \W$ is finite and etale), 
for any integer $m \in \N$, 
we denote 
$\smash{\widehat{\D}} _{\X/\Spf (\W)} ^{(m)}$
(resp. $\smash{\D} ^\dag _{\X/\Spf (\W), \Q}$)
simply by 
$\smash{\widehat{\D}} _{\X} ^{(m)}$
(resp. $\smash{\D} ^\dag _{\X\Q}$).
Berthelot checked the following equivalence of categories (see \cite[4.2.4]{Beintro2}):
\begin{equation}
\label{limeqcat}
\underrightarrow{\lim}
\colon 
\smash{\underrightarrow{LD}} ^{\mathrm{b}} _{\Q,\mathrm{coh}} ( \smash{\widehat{\D}} _{\X} ^{(\bullet)})
\cong 
D ^{\mathrm{b}}  _{\mathrm{coh}} (\smash{\D} ^\dag _{\X\Q}).
\end{equation}

The category 
$D ^{\mathrm{b}}  _{\mathrm{coh}} (\smash{\D} ^\dag _{\X\Q})$ 
is endowed with its usual t-structure.
Via \ref{limeqcat}, we get a t-structure on 
$\smash{\underrightarrow{LD}} ^{\mathrm{b}} _{\Q,\mathrm{coh}} ( \smash{\widehat{\D}} _{\X} ^{(\bullet)})$
whose heart is 
$\smash{\underrightarrow{LM}}  _{\Q,\mathrm{coh}} ( \smash{\widehat{\D}} _{\X} ^{(\bullet)})$
(see Notation \cite[2.2.4]{caro-stab-sys-ind-surcoh}).
In fact, from 
\cite[1.2.7]{caro-stab-sys-ind-surcoh}
and 
\cite[2.5.1]{caro-stab-sys-ind-surcoh},
we have canonical explicit cohomological functors
$\mathcal{H} ^n 
\colon 
\smash{\underrightarrow{LD}} ^{\mathrm{b}} _{\Q,\mathrm{coh}} ( \smash{\widehat{\D}} _{\X} ^{(\bullet)})
\to 
\smash{\underrightarrow{LM}} _{\Q,\mathrm{coh}} ( \smash{\widehat{\D}} _{\X} ^{(\bullet)})$.
The equivalence of categories 
\ref{limeqcat} commute with the 
cohomogical functors $\mathcal{H} ^n $
(where the cohomogical functors $\mathcal{H} ^n $ on 
$D ^{\mathrm{b}}  _{\mathrm{coh}} (\smash{\D} ^\dag _{\X\Q})$
are the obvious ones),
 i.e. 
$\underrightarrow{\lim}
\mathcal{H} ^n (\E ^{(\bullet)})$
is  canonically isomorphic
to 
$\mathcal{H} ^n (\underrightarrow{\lim} \,\E ^{(\bullet)})$.

Last but not least, 
via Theorem \cite[2.5.7]{caro-stab-sys-ind-surcoh}
and 
Lemma \cite[2.5.2]{caro-stab-sys-ind-surcoh}, 
we remind the equivalence of categories 
$\smash{\underrightarrow{LD}} ^{\mathrm{b}} _{\Q,\mathrm{coh}} ( \smash{\widehat{\D}} _{\X} ^{(\bullet)})
\cong 
D ^{\mathrm{b}} _{\mathrm{coh}}
(
\smash{\underrightarrow{LM}} _{\Q} ( \smash{\widehat{\D}} _{\X} ^{(\bullet)})
)
$
which is also compatible with t-structures
(the t-structure on 
$D ^{\mathrm{b}} _{\mathrm{coh}}
(\smash{\underrightarrow{LM}} _{\Q} ( \smash{\widehat{\D}} _{\X} ^{(\bullet)})$
is the canonical one as the derived category of an abelian category).
\end{empt}

\begin{dfn}
\label{dfn-datacoef}
A {\it data of coefficients $\mathfrak{C}$ over $\V$}
will be the data for any object
$\W$ of $\mathrm{DVR}  (\V)$, 
for any  
smooth formal scheme $\X$ over $\W$ 
of a full subcategory of 
$\smash{\underrightarrow{LD}} ^{\mathrm{b}} _{\Q,\mathrm{coh}} ( \smash{\widehat{\D}} _{\X} ^{(\bullet)})$,
which will be denoted by $\mathfrak{C} (\X)$.
If there is no ambiguity with $\V$, 
we simply say a {\it data of coefficients}.
\end{dfn}

\begin{exs}
\label{ex-Dcst}

\begin{enumerate}
\item We define the data of coefficients $\mathfrak{B} _\emptyset$ as follows: 
for any object $\W$ of $\mathrm{DVR}  (\V)$, 
for any  smooth formal scheme $\X$ over $\W$,  
the category $\mathfrak{B} _\emptyset (\X)$ is the full subcategory of 
$\smash{\underrightarrow{LD}} ^{\mathrm{b}} _{\Q,\mathrm{coh}} ( \smash{\widehat{\D}} _{\X} ^{(\bullet)})$
whose unique object is 
$\O _{\X} ^{(\bullet)}$ 
(where $\O _{\X} ^{(\bullet)}$ is the constant object $\O _{\X} ^{(m)} = \O _{\X}$).

\item We will need the larger data of coefficients $\mathfrak{B} _\mathrm{div}$ 
defined as follows: 
for any object $\W$ of $\mathrm{DVR}  (\V)$, 
for any  smooth formal scheme $\X$ over $\W$,  
the category $\mathfrak{B} _\mathrm{div}(\X)$ is the full subcategory of 
$\smash{\underrightarrow{LD}} ^{\mathrm{b}} _{\Q,\mathrm{coh}} ( \smash{\widehat{\D}} _{\X} ^{(\bullet)})$
whose objects are of the form  
$\widehat{\B} ^{(\bullet)} _{\X} (T)$, 
where $T$ is any divisor of the special fiber of $\X$ (the sheaf 
$\widehat{\B} ^{(\bullet)} _{\X} (T)$ is defined in \cite[4.2.4]{Be1}). 
From Corollary \cite[3.5.3]{caro-stab-sys-ind-surcoh}, 
we have
$\widehat{\B} ^{(\bullet)} _{\X} (T)\in 
\smash{\underrightarrow{LD}} ^{\mathrm{b}} _{\Q,\mathrm{coh}} ( \smash{\widehat{\D}} _{\X} ^{(\bullet)})$. 

\item We define $\mathfrak{B} _\mathrm{cst}$ 
as follows: 
for any object $\W$ of $\mathrm{DVR}  (\V)$, 
for any  smooth formal scheme $\X$ over $\W$,  
the category $\mathfrak{B} _\mathrm{cst}(\X)$ is the full subcategory of 
$\smash{\underrightarrow{LD}} ^{\mathrm{b}} _{\Q,\mathrm{coh}} ( \smash{\widehat{\D}} _{\X} ^{(\bullet)})$
whose objects are of the form  
$\R \underline{\Gamma} ^\dag _{Y} \O _\X ^{(\bullet)} $, 
where $Y$ is a subvariety of the special fiber of $\X$
and the functor 
$\R \underline{\Gamma} ^\dag _{Y}$ is defined in 
\cite[3.2.1]{caro-2006-surcoh-surcv}
(to see that these objects are coherent,
we proceed as in the proof of \ref{S(D,C)stability3-pre}
this is a consequence of Corollary \cite[3.5.3]{caro-stab-sys-ind-surcoh}).

\end{enumerate}

\end{exs}

\begin{dfn}
\label{realizablefscheme}
Let $\W$ be an object of $\mathrm{DVR}  (\V)$  (see Definition \ref{DVR}). 
Let $f \colon \PP' \to \PP$ be a morphism of smooth formal $\W$-schemes. 
We say that $f$ is {\it realizable} if 
there exist an immersion 
$u \colon \PP' \hookrightarrow \PP''$ of smooth formal $\W$-schemes, 
a proper morphism 
$\pi \colon \PP'' \to \PP$ of smooth formal $\W$-schemes
such that
$f = \pi \circ u$.
When $\PP= \Spf \W$, we say that 
$\PP'$ is a realizable smooth formal $\W$-scheme.
We remark that any morphism of realizable smooth formal $\W$-schemes is a realizable morphism.

Because of the relative duality isomorphism of the form
\ref{rel-dual-isom},
which is not known in a more general case, 
we will need to focus on pushforwards by realizable morphisms. 
\end{dfn}

\begin{dfn}
\label{dfn-stable-data}
In order to be precise, let us fix some terminology.
Let $\mathfrak{C}$ and $\mathfrak{D}$ be two data of coefficients over $\V$. 
\begin{enumerate}
\item We will say that the data of coefficients of $\mathfrak{C}$ is
stable under pushforwards (resp. {\it realizable} pushforwards)
if for any object $\W$ of $\mathrm{DVR}  (\V)$, 
for any 
morphism (resp.  {\it realizable} morphism) $g \colon \X ' \to \X$ of  smooth formal schemes over $\W$, 
for any objet $\E ^{\prime (\bullet)}$ of $\mathfrak{C} (\X')$ with proper support over $X$ via $g$ 
(i.e., if $Z'$ is the support of $\E ^{\prime (\bullet)}$ then the composition $Z ' \hookrightarrow  X' \overset{g}{\to} X$ is proper),
the complex $g _{+} (\E ^{\prime (\bullet)})$ is an object of  $\mathfrak{C} (\X)$.

\item We will say that the data of coefficients of $\mathfrak{C}$ is stable under extraordinary pullbacks 
(resp. under {\it smooth} extraordinary pullbacks)
if for any object $\W$ of $\mathrm{DVR}  (\V)$, 
for any morphism (resp. {\it smooth} morphism)
$f \colon \Y \to \X$ of  smooth formal schemes over $\W$, 
for any objet $\E ^{(\bullet)}$ of $\mathfrak{C} (\X)$, 
we have $f ^{!} (\E ^{(\bullet)})\in \mathfrak{C} (\Y)$.

\item We still say that the data of coefficients of $\mathfrak{C}$ satisfies 
the first property (resp. the second property) of Berthelot-Kashiwara theorem
or satisfies $BK ^!$ (resp. $BK _+$) for short if the following property is satisfied:
for any object $\W$ of $\mathrm{DVR}  (\V)$, 
for any closed immersion  $u \colon \ZZ \hookrightarrow \X$ of  smooth formal schemes over $\W$, 
for any objet $\E ^{(\bullet)}$ of $\mathfrak{C} (\X)$ with support in $\ZZ$, 
we have $u ^{!} (\E ^{(\bullet)})\in \mathfrak{C} (\ZZ)$
(resp. 
for any objet $\G ^{(\bullet)}$ of $\mathfrak{C} (\ZZ)$, 
we have $u _{+} (\G ^{(\bullet)})\in \mathfrak{C} (\X)$).
Remark that $BK ^!$ and $BK _+$ hold if and only if the data of coefficients $\mathfrak{C}$ satisfies 
(an analogue of) Berthelot-Kashiwara theorem, which justifies the terminology. 

\item We will say that the data of coefficients $\mathfrak{C}$ is stable under base change if
for any morphism $\W \to \W' $ of $\mathrm{DVR}  (\V)$, 
for any  smooth formal scheme $\X$ over $\W$, 
for any objet $\E ^{(\bullet)}$ of $\mathfrak{C} (\X)$, 
we have $ \W'  \smash{\overset{\L}{\otimes}}   ^{\dag}
_{\W}  \E^{(\bullet) }\in \mathfrak{C} (\X \times _{\Spf \W} \Spf \W ')$
(see Notation \ref{comm-chg-base}).

\item We will say that the data of coefficients $\mathfrak{C}$ is stable under tensor products (resp. duals) if
for any object $\W$ of $\mathrm{DVR}  (\V)$, 
for any  smooth formal scheme $\X$ over $\W$, 
for any objects $\E ^{(\bullet)}$ and $\FF ^{(\bullet)}$ of $\mathfrak{C} (\X)$
we have 
$\FF ^{(\bullet)}
\smash{\overset{\L}{\otimes}}   ^{\dag}
_{\O  _{\X}} \E^{(\bullet) }
\in \mathfrak{C} (\X)$
(resp. 
$\DD  _{\X}(\E^{(\bullet) }) \in  \mathfrak{C} (\X)$) .

\item We will say that the data of coefficients $\mathfrak{C}$ is stable under local cohomological functors
(resp. under localizations outside a divisor), if 
for any object $\W$ of $\mathrm{DVR}  (\V)$, 
for any  smooth formal scheme $\X$ over $\W$, 
for any object $\E ^{(\bullet)}$ of $\mathfrak{C} (\X)$, 
for any subvariety $Y$ (resp. for any divisor $T$) of the special fiber of $\X$, 
we have
$\R \underline{\Gamma} ^\dag _{Y} \E ^{(\bullet)} \in \mathfrak{C} (\X)$ 
(resp. $(\hdag T) (\E ^{(\bullet)} )\in \mathfrak{C} (\X)$),
where we use the notation of \cite[3.2.1]{caro-2006-surcoh-surcv}).

\item We will say that the data of coefficients $\mathfrak{C}$ is stable under shifts if,
for any object $\W$ of $\mathrm{DVR}  (\V)$, 
for any  smooth formal scheme $\X$ over $\W$,
for any object  $\E ^{(\bullet)}$ of $\mathfrak{C} (\X)$, for any integer $n$, 
$\E ^{(\bullet)} [n]$ is an object of $\mathfrak{C} (\X)$.

\item We will say that the data of coefficients $\mathfrak{C}$ is stable by devissages if
 $\mathfrak{C}$ is stable by shifts and if
for any object $\W$ of $\mathrm{DVR}  (\V)$, 
for any  smooth formal scheme $\X$ over $\W$, 
for any exact triangle 
$\E ^{(\bullet)} _1
\to 
\E ^{(\bullet)} _2
\to 
\E ^{(\bullet)} _3
\to 
\E ^{(\bullet)} _1 [1]$
of $\smash{\underrightarrow{LD}} ^{\mathrm{b}} _{\Q,\mathrm{coh}} ( \smash{\widehat{\D}} _{\X} ^{(\bullet)})$, 
if two objects are in $\mathfrak{C} (\X)$, then so is the third one. 

\item We will say that the data of coefficients $\mathfrak{C}$ is stable under direct factors if,
for any object $\W$ of $\mathrm{DVR}  (\V)$, 
for any  smooth formal scheme $\X$ over $\W$ we have the following property: 
any direct factor in $\smash{\underrightarrow{LD}} ^{\mathrm{b}} _{\Q,\mathrm{coh}} ( \smash{\widehat{\D}} _{\X} ^{(\bullet)})$
of an object of $\mathfrak{C} (\X)$ is an object of
$\mathfrak{C} (\X)$.

\item We say that $\mathfrak{C}$ contains $\mathfrak{D}$ 
(or $\mathfrak{D}$ is contained in $\mathfrak{C}$) 
if for any object $\W$ of $\mathrm{DVR}  (\V)$, 
for any  smooth formal scheme $\X$ over $\W$
the category $\mathfrak{D} (\X)$ is a subcategory of $\mathfrak{C} (\X)$.

\item We say that the data of coefficients $\mathfrak{C}$ is local 
if for any object $\W$ of $\mathrm{DVR}  (\V)$, 
for any  smooth formal scheme $\X$ over $\W$, 
for any open covering $(\X _i) _{i\in I}$ of $\X$, 
for any object $\E ^{(\bullet)}$ of
$\smash{\underrightarrow{LD}} ^{\mathrm{b}} _{\Q,\mathrm{qc}} ( \smash{\widehat{\D}} _{\X} ^{(\bullet)})$, 
we have 
$\E ^{(\bullet)}\in \mathrm{Ob} \mathfrak{C} (\X)$ if and only if 
$\E ^{(\bullet)}| \X _i \in \mathrm{Ob} \mathfrak{C} (\X _i)$ for any $i \in I$. 
For instance, it follows from 
Theorem \cite[2.5.7]{caro-stab-sys-ind-surcoh} 
(see the localness in Definition \cite[2.4.1]{caro-stab-sys-ind-surcoh})
that 
the data of coefficients 
$\smash{\underrightarrow{LD}} ^{\mathrm{b}} _{\Q,\mathrm{coh}}$ is local. 
 
\end{enumerate}

\end{dfn}

We finish the subsection with some notation.

\begin{empt}
[Duality]
\label{ntn-dual}
Let $\mathfrak{C}$ be a data of coefficients. We define its dual data of coefficients 
$\mathfrak{C}  ^{\vee}$ as follows: 
for any object $\W$ of $\mathrm{DVR}  (\V)$, 
for any   smooth formal scheme $\X$ over $\W$, 
the category $\mathfrak{C}  ^{\vee} (\X)$ is the subcategory of 
$\smash{\underrightarrow{LD}} ^{\mathrm{b}} _{\Q,\mathrm{coh}} ( \smash{\widehat{\D}} _{\X} ^{(\bullet)})$
of objects $\E^{(\bullet) }$
such that $\DD _{\X} (\E^{(\bullet) }) \in \mathfrak{C} (\X)$.
\end{empt}

\subsection{Data of coefficients with potentially Frobenius structure over $\V$}

\begin{ntn}
\label{ntn-hypothFrob}
We assume that the absolute Frobenius homomorphism
$k \riso k$ sending $x $ to $x ^{p}$ lifts to an automorphism $\sigma \colon \V \riso \V$. 
We denote by $\mathrm{DVR}  (\V, \sigma)$ the full subcategory of $\mathrm{DVR}  (\V)$
whose objects $\alpha \colon \V \to \W$ are such that the absolute Frobenius homomorphism of the residue field of $\W$ 
has a lifting of the form $\sigma _\W\colon  \W \riso \W$ commuting with $\sigma$, i.e. 
such that 
$\sigma _\W \circ \alpha= \alpha \circ \sigma $.
\end{ntn}

Similarly to \ref{dfn-datacoef}, we introduce the following definition. 
\begin{dfn}
\label{dfn-datacoefFrob}
We keep the hypothesis of \ref{ntn-hypothFrob}.
A {\it data of coefficients with potentially Frobenius structure $\mathfrak{C}$ over $\V$}
will be the data, for any object
$\W$ of $\mathrm{DVR}  (\V,\sigma)$, 
for any  smooth formal scheme $\X$ over $\W$, 
of a full subcategory of 
$\smash{\underrightarrow{LD}} ^{\mathrm{b}} _{\Q,\mathrm{coh}} ( \smash{\widehat{\D}} _{\X} ^{(\bullet)})$,
which will be denoted by $\mathfrak{C} (\X)$.
If there is no ambiguity with $\V$, we simply say
a {\it data of coefficients with potentially Frobenius structure $\mathfrak{C}$}.

As in \ref{dfn-stable-data}, we define the notion of {\it local} data of coefficients with potentially Frobenius structure over $\V$, 
of its stability under shifts, devissages, direct factors, extraordinary pullbacks, 
pushforwards, base change (of course, we restrict here to morphism in $\mathrm{DVR}  (\V,\sigma)$),
tensor products, dual functors, local cohomological functors, localisation outside a divisor etc.  

\end{dfn}

\begin{rem}
We notice that by definition, a data of coefficients  over $\V$
induces by restriction a data of coefficients with potentially Frobenius structure $\mathfrak{C}$ over $\V$.

\end{rem}

\subsection{Overcoherence, overholonomicity (after any base change) and complements}

\begin{dfn}
\label{dfnS(D,C)}
Let $\mathfrak{C}$ and $\mathfrak{D}$ be two data of coefficients.

\begin{enumerate}
\item We denote by 
$S _0 (\mathfrak{D}, \mathfrak{C})$
the data of coefficients defined as follows: 
for any object $\W$ of $\mathrm{DVR}  (\V)$, 
for any  smooth formal scheme $\X$ over $\W$,  
the category 
$S _0 (\mathfrak{D}, \mathfrak{C}) (\X)$
is the full subcategory of 
$\smash{\underrightarrow{LD}} ^{\mathrm{b}} _{\Q,\mathrm{coh}} ( \smash{\widehat{\D}} _{\X} ^{(\bullet)})$
of objects  $\E ^{(\bullet)}$
satisfying the following properties :
\begin{enumerate}
\item [($\star$)] for any smooth morphism $f\colon \Y \to \X$ of  smooth formal $\W$-schemes, 
for any object 
$\FF ^{(\bullet)}
\in
\mathfrak{D} (\Y)$,
we have 
$\FF ^{(\bullet)}
\smash{\overset{\L}{\otimes}}   ^{\dag}
_{\O  _{\Y}} f ^{!} (\E^{(\bullet) })
\in
\mathfrak{C} (\Y)$.
\end{enumerate}

\item We denote by 
$S(\mathfrak{D}, \mathfrak{C})$ 
the data of coefficients defined as follows: 
for any object $\W$ of $\mathrm{DVR}  (\V)$, 
for any  smooth formal scheme $\X$ over $\W$,  
the category 
$S  (\mathfrak{D}, \mathfrak{C}) (\X)$ 
is the full subcategory of 
$\smash{\underrightarrow{LD}} ^{\mathrm{b}} _{\Q,\mathrm{coh}} ( \smash{\widehat{\D}} _{\X} ^{(\bullet)})$
of objects  $\E ^{(\bullet)}$
satisfying the following property: 
\begin{enumerate}
\item [($\star \star$)] for any morphism $\W \to \W '$ of $\mathrm{DVR}  (\V)$, 
we have 
$ \W'  \smash{\overset{\L}{\otimes}}   ^{\dag}_{\W}  \E^{(\bullet) }\in 
S _0 (\mathfrak{D}, \mathfrak{C}) (\X \times _{\Spf \W} \Spf \W ')$.
\end{enumerate}

\end{enumerate}

\end{dfn}

\begin{exs}
\label{ex-cst-surcoh}
\begin{enumerate}

\item
We get by definition the equality
$ \smash{\underrightarrow{LD}} ^{\mathrm{b}} _{\Q,\mathrm{ovcoh}}= 
S _0 (\mathfrak{B} _\mathrm{div}, \smash{\underrightarrow{LD}} ^{\mathrm{b}} _{\Q,\mathrm{coh}})$
(this notion of overcoherence is defined in
\cite[5.4]{caro-stab-sys-ind-surcoh}).
Moreover, $S (\mathfrak{B} _\mathrm{div}, \smash{\underrightarrow{LD}} ^{\mathrm{b}} _{\Q,\mathrm{coh}})$
corresponds to the notion of overcoherence after any base change as defined in 
\cite{surcoh-hol}.

\item \label{hstab} 
We put 
$\mathfrak{H} _0 :=S (\mathfrak{B} _\mathrm{div}, \smash{\underrightarrow{LD}} ^{\mathrm{b}} _{\Q,\mathrm{coh}})$ 
and 
by induction on $i \in \N$, 
we put $\mathfrak{H} _{i+1} :=
\mathfrak{H} _{i} \cap S(\mathfrak{B} _\mathrm{div}, \mathfrak{H} _{i} ^{\vee})$
(see Notation \ref{ntn-dual}).
The coefficients of 
$\mathfrak{H} _{i}$ are called 
{\it $i$-overholonomic after any base change}.
We get the data of coefficients 
$\smash{\underrightarrow{LD}} ^{\mathrm{b}} _{\Q,\mathrm{h}}:=  \cap _{i\in \N} \mathfrak{H} _{i}$
whose objects are called {\it overholonomic after any base change}.

\item \label{ovholstab}
Replacing $S $ by $S _0$ in the definition of $\smash{\underrightarrow{LD}} ^{\mathrm{b}} _{\Q,\mathrm{h}}$, 
we get a data of coefficients that we will denote by 
$\smash{\underrightarrow{LD}} ^{\mathrm{b}} _{\Q,\mathrm{ovhol}}$.

\end{enumerate}

\end{exs}

\begin{rem}
\label{rem-overhol}
\begin{enumerate}
\item Let $\mathfrak{C}$ be a data of coefficients.
The data of coefficients $\mathfrak{C}$ is stable under smooth extraordinary inverse image, localizations outside a divisor
(resp. under smooth extraordinary inverse image, localizations outside a divisor, and base change) if and only if 
$S _0 (\mathfrak{B} _\mathrm{div}, \mathfrak{C})=\mathfrak{C}$
(resp. $S  (\mathfrak{B} _\mathrm{div}, \mathfrak{C})=\mathfrak{C}$).

\item By construction, 
we remark that 
$\smash{\underrightarrow{LD}} ^{\mathrm{b}} _{\Q,\mathrm{ovhol}}$
is the biggest data of coefficients 
which contains 
$\mathfrak{B} _\mathrm{div}$, 
is stable by devissage, 
dual functors
and the operation
$S _{0} (\mathfrak{B} _\mathrm{div}, -)$. 
Moreover, 
$\smash{\underrightarrow{LD}} ^{\mathrm{b}} _{\Q,\mathrm{h}}$
is the biggest data of coefficients 
which contains 
$\mathfrak{B} _\mathrm{div}$, 
is stable by devissage, 
dual functors
and the operation
$S  (\mathfrak{B} _\mathrm{div}, -)$. 

 \item Let $\W$ be an object  of $\mathrm{DVR}  (\V)$, 
$\X$ be a   smooth formal $\W$-scheme.
We denote by 
$D ^{\mathrm{b}}  _{\mathrm{ovcoh}} (\smash{\D} ^\dag _{\X\Q})$
(resp. $D ^{\mathrm{b}}  _{\mathrm{ovhol}} (\smash{\D} ^\dag _{\X\Q})$, 
resp. $D ^{\mathrm{b}}  _{\mathrm{h}} (\smash{\D} ^\dag _{\X\Q})$) the category of 
overcoherent complexes
(resp. overholonomic complexes,
resp. overholonomic complexes after any base change)
of $\smash{\D} ^\dag _{\X\Q}$-modules 
as defined in \cite[3.2.1]{surcoh-hol}
(resp. \cite[3]{caro_surholonome},
resp. \cite[3.2.1]{surcoh-hol}).
We recall that, from the proposition \cite[5.4.3]{caro-stab-sys-ind-surcoh}, we have 
as in \ref{limeqcat} the following  equivalence of categories
\ref{limeqcat}
\begin{equation}
\label{limeqcatovcoh}
\underrightarrow{\lim}
\colon 
\smash{\underrightarrow{LD}} ^{\mathrm{b}} _{\Q,\mathrm{ovcoh}} ( \smash{\widehat{\D}} _{\X} ^{(\bullet)})
\cong 
D ^{\mathrm{b}}  _{\mathrm{ovcoh}} (\smash{\D} ^\dag _{\X\Q}).
\end{equation}
Hence, as in \ref{t-structure-coh}, we get from \ref{limeqcatovcoh} a canonical t-structure on 
$\smash{\underrightarrow{LD}} ^{\mathrm{b}} _{\Q,\mathrm{ovcoh}} ( \smash{\widehat{\D}} _{\X} ^{(\bullet)})$.
Moreover, 
we check that the functor 
$\underrightarrow{\lim}$
of 
\ref{limeqcatovcoh}
induces an equivalence between 
$\smash{\underrightarrow{LD}} ^{\mathrm{b}} _{\Q,\mathrm{ovhol}} ( \smash{\widehat{\D}} _{\X} ^{(\bullet)})$
(resp. $\smash{\underrightarrow{LD}} ^{\mathrm{b}} _{\Q,\mathrm{h}} ( \smash{\widehat{\D}} _{\X} ^{(\bullet)})$)
and 
$D ^{\mathrm{b}}  _{\mathrm{ovhol}} (\smash{\D} ^\dag _{\X\Q})$
(resp. $D ^{\mathrm{b}}  _{\mathrm{h}} (\smash{\D} ^\dag _{\X\Q})$).
Finally, recall from \cite[3.4.2]{surcoh-hol} that 
$D ^{\mathrm{b}}  _{\mathrm{ovhol}} (\smash{\D} ^\dag _{\X\Q})$ 
(resp. $D ^{\mathrm{b}}  _{\mathrm{h}} (\smash{\D} ^\dag _{\X\Q})$) has a canonical t-structure induced by 
that of $D ^{\mathrm{b}}  _{\mathrm{coh}} (\smash{\D} ^\dag _{\X\Q})$, i.e.
a coherent complex is overholonomic (after any base change) if and only so are if its cohomological spaces.
Hence, we get a canonical t-structure on 
$\smash{\underrightarrow{LD}} ^{\mathrm{b}} _{\Q,\mathrm{ovhol}} ( \smash{\widehat{\D}} _{\X} ^{(\bullet)})$
and $\smash{\underrightarrow{LD}} ^{\mathrm{b}} _{\Q,\mathrm{h}} ( \smash{\widehat{\D}} _{\X} ^{(\bullet)})$.

\end{enumerate}
\end{rem}

\begin{empt}
\label{ovcoh-invim}
We already know that the data of coefficients
$ \smash{\underrightarrow{LD}} ^{\mathrm{b}} _{\Q,\mathrm{ovcoh}}$
is stable under direct factors and 
extraordinary pullbacks (see \cite[5.4.5]{caro-stab-sys-ind-surcoh}, or remark that this 
is a consequence of Lemmas \ref{lem-stabextpullback},
\ref{S(D,C)stability3-pre} and \ref{S(D,C)stability}.4--5). 
Concerning the stability under pushforwards, 
this is the purpose of Proposition \ref{stab-propersupp}
(in the litterature, we only knew the stability under pushforwards by a proper morphism :
see \cite[5.4.8]{caro-stab-sys-ind-surcoh}).
First, we need to recall some properties and notations concerning the devissability in overcoherent isocrystals.
\end{empt}

\begin{empt}
[Isocrystals and notation]
\label{Isodagdagdiv}
Let $\W$ be an object of $\mathrm{DVR}  (\V)$  (see Definition \ref{DVR}). 
Let $\PP$ be a smooth formal scheme over $\W$, 
$X$ be a closed subscheme of $P$ and $T$ be a divisor of $P$ such that $Y:=X \setminus T$ is smooth (over the residue field of $\W$).
We denote by 
$\textrm{Isoc} ^{\dag \dag}(\PP, T, X/\W)$ the category of partially overcoherent isocrystals on 
$(\PP, T, X/\W)$
(see \cite[1.4.2]{caro-stab-prod-tens}, but we replaced in the notation $\W$ by its field of fraction).
This is a full subcategory of that of 
overcoherent $\smash{\D} ^\dag _{\PP} (\hdag T) _{\Q}$-modules with support in $X$.
On the other hand, 
we denote by 
$\mathrm{Isoc} ^{\dag}(\PP, T, X/\W)$, Berthelot's category of overconvergent isocrystals on $(\PP, T, X/\W)$
(i.e. $\mathrm{Isoc} ^{\dag}(\PP, T, X/\W):=\mathrm{Isoc} ^{\dag}(Y, X/\W)$, more precisely their objects are the realization over $\PP$ of
overconvergent isocrystals on  $(Y, X/\W)$).
From \cite[5.4.6.1]{caro-pleine-fidelite}, we have the equivalence of categories of the form
\begin{equation}
\label{eqcat-isocpre}
\sp _{+}
\colon
\mathrm{Isoc} ^{\dag}(\PP, T, X/\W)
\cong 
\mathrm{Isoc} ^{\dag \dag}(\PP, T, X/\W),
\end{equation}
which explains the terminology of the right hand side.
We also denote by
$D^{\mathrm{b}}_{\mathrm{isoc}}(\PP, T, X/\W)$
the full subcategory of
$D ^{\mathrm{b}}  _{\mathrm{ovcoh}} (\smash{\D} ^\dag _{\PP} (\hdag T) _{\Q})$
whose cohomological spaces are objects of 
$\mathrm{Isoc} ^{\dag \dag}(\PP, T, X/\W)$.

Following \cite[1.4.3]{caro-stab-prod-tens}, we denote by 
$\mathrm{Isoc} ^{(\bullet)} (\PP, T, X/\W)$
the full subcategory of 
$\smash{\underrightarrow{LM}}  _{\Q,\mathrm{coh}} ( \smash{\widehat{\D}} _{\PP} ^{(\bullet)} (T))$
(see Notation \cite[2.2.4]{caro-stab-sys-ind-surcoh})
of objects
$\E ^{(\bullet)}$ such that
$\underrightarrow{\lim}~
\E ^{ (\bullet)}
\in 
\mathrm{Isoc} ^{\dag \dag}(\PP, T, X/\W)$.
Following \cite[4.1.4]{caro-stab-prod-tens}, we denote by 
$\smash{\underrightarrow{LD}} ^{\mathrm{b}} _{\Q,\mathrm{isoc},X} ( \smash{\widehat{\D}} _{\PP} ^{(\bullet)} (T))$
the full subcategory of 
$\smash{\underrightarrow{LD}} ^{\mathrm{b}} _{\Q,\mathrm{coh}} ( \smash{\widehat{\D}} _{\PP} ^{(\bullet)} (T))$
of objects
$\E ^{(\bullet)}$ such that
$\mathcal{H} ^j (\E ^{ (\bullet)}) \in  \mathrm{Isoc} ^{(\bullet)} (\PP, T, X/\W)$
for any integer $j\in \Z$.
Following \cite[4.1.5]{caro-stab-prod-tens},
we have the equivalence of categories
 $\underrightarrow{\lim}
 \colon 
 \smash{\underrightarrow{LD}}  ^{\mathrm{b}}  _{\Q,\mathrm{isoc},X} ( \smash{\widehat{\D}} _{\PP} ^{(\bullet)} (T))
 \cong
D ^{\mathrm{b}} _{\mathrm{isoc}}(\PP, T, X/\W)$.
Following \cite[4.1.4]{caro-stab-prod-tens}, the category
$ \smash{\underrightarrow{LD}} ^{\mathrm{b}} _{\Q,\mathrm{isoc},X} ( \smash{\widehat{\D}} _{\PP} ^{(\bullet)} (T))$ 
does not depend on the choice of the closed scheme $X$ of $P$ and
on the divisor $T$ of $P$ such that 
$Y = X \setminus T$. Hence, 
the category 
$ \smash{\underrightarrow{LD}} ^{\mathrm{b}} _{\Q,\mathrm{isoc},X} ( \smash{\widehat{\D}} _{\PP} ^{(\bullet)} (T))$
will simply be denoted  by 
$ \smash{\underrightarrow{LD}} ^{\mathrm{b}} _{\Q,\mathrm{isoc},Y} ( \smash{\widehat{\D}} _{\PP} ^{(\bullet)})$.

\end{empt}

\begin{empt}
[Devissage in overcoherent isocrystals]
\label{ovcoharedev}
Let $\W$ be an object of $\mathrm{DVR}  (\V)$.
Let $\PP$ be a smooth formal $\W$-scheme. 
Let $\E ^{ (\bullet)} \in \smash{\underrightarrow{LD}} ^{\mathrm{b}} _{\Q,\mathrm{ovcoh}} ( \smash{\widehat{\D}} _{\PP} ^{(\bullet)})$.
Let $X$ be the support of $\E ^{(\bullet)}$. 
There exists a smooth $d$-stratification 
$(Y _{i}) _{i=1, \dots , r}$ of $X$ in $P$ (see Definition \cite[4.1.2.2]{caro-stab-prod-tens}) 
such that we have
$ \R \underline{\Gamma} ^\dag _{Y _i}  (\E ^{ (\bullet)}) 
\in 
\underrightarrow{LD} ^{\mathrm{b}} _{\Q,  \textrm{isoc}, Y _i} ( \smash{\widehat{\D}} _{\PP} ^{(\bullet)})$,
for any for any $i =1,\dots, r$ (see Theorem \cite[6.2.3]{caro-pleine-fidelite} and Definition \cite[6.2.2]{caro-pleine-fidelite}).
\end{empt}

\begin{prop}
\label{stab-propersupp}
Let $\W$ be an object of $\mathrm{DVR}  (\V)$.
Let $g \colon \X '\to \X$ be a morphism of smooth formal $\W$-schemes.
For any 
$\E ^{\prime (\bullet)} \in \smash{\underrightarrow{LD}} ^{\mathrm{b}} _{\Q,\mathrm{ovcoh}} ( \smash{\widehat{\D}} _{\X '} ^{(\bullet)})$
with proper support over $X$, 
 the object 
$g _{+} (\E ^{\prime (\bullet)} ) $
belongs to 
$\smash{\underrightarrow{LD}} ^{\mathrm{b}} _{\Q,\mathrm{ovcoh}} ( \smash{\widehat{\D}} _{\X} ^{(\bullet)})$.
\end{prop}

\begin{proof}
This is an analogue of Theorem \cite[2.3.2]{caro-image-directe}
 (see the remark \ref{rem-LD-fromD} below
which explain why this is not a straightforward consequence of Theorem \cite[2.3.2]{caro-image-directe})
and its proof can be adapted. For the comfort of the reader, a complete detailed proof is given as follows: 
Let $Z'$ be the support of $\E ^{\prime (\bullet)}$. 
Following \ref{ovcoharedev}, 
since $\E ^{\prime (\bullet)}$ is overcoherent, 
there exists a smooth $d$-stratification 
$(U '_{i}) _{i=1, \dots , r}$ of $Z'$ in $X'$
such that 
$ \R \underline{\Gamma} ^\dag _{U '_i}  (\E ^{\prime (\bullet)}) 
\in 
\underrightarrow{LD} ^{\mathrm{b}} _{\Q,  \textrm{isoc},U '_i} ( \smash{\widehat{\D}} _{\X '} ^{(\bullet)})$,
for any for any $i =1,\dots, r$.
Since 
$\smash{\underrightarrow{LD}} ^{\mathrm{b}} _{\Q,\mathrm{ovcoh}} ( \smash{\widehat{\D}} _{\X'} ^{(\bullet)})$
is a triangle subcategory of 
$\smash{\underrightarrow{LD}} ^{\mathrm{b}} _{\Q,\mathrm{qc}} ( \smash{\widehat{\D}} _{\X'} ^{(\bullet)})$,
we reduce by devissage to check that 
$ g _{+} \R \underline{\Gamma} ^\dag _{U '_i}  (\E ^{\prime (\bullet)})
\in \smash{\underrightarrow{LD}} ^{\mathrm{b}} _{\Q,\mathrm{ovcoh}} ( \smash{\widehat{\D}} _{\X} ^{(\bullet)})$. 
We can suppose $U' _i$ integral. 
Again by devissage, 
we reduce to check 
$ g _{+} \left (\mathcal{H} ^j (\R \underline{\Gamma} ^\dag _{U '_i}  (\E ^{\prime (\bullet)})) \right )
\in \smash{\underrightarrow{LD}} ^{\mathrm{b}} _{\Q,\mathrm{ovcoh}} ( \smash{\widehat{\D}} _{\X} ^{(\bullet)})$
for any integer $j\in \Z$. 
We have 
$\E ^{\prime (\bullet)} _{i,j}:=\mathcal{H} ^j (\R \underline{\Gamma} ^\dag _{U '_i}  (\E ^{\prime (\bullet)}))
\in  \mathrm{Isoc} ^{(\bullet)} (\X', T '_i, Z '_i /\W)$, where $Z ' _i$ is the closure of $U '_i$ in $ X' _i$ and 
$T ' _i$ is some divisor of $X ' _i$.

Following \cite[5.3.1.1]{caro-pleine-fidelite}, 
there exists a commutative diagram of the form 
  \begin{equation}
  \label{diag-casliss631dev}
  \xymatrix {
  {Z'' _i} \ar[r] ^-{u''} \ar[d] _{a '} & {\P ^N _{X'}} \ar[r] \ar[d] &
  {\widehat{\P} ^N _{\X'}} \ar[r] ^{\widehat{\P} ^N _g} \ar[d] ^-{q'} & {\widehat{\P} ^N _{\X}} \ar[d]^-{q} \\
  {Z _i '} \ar[r] ^-{u'} & {X'} \ar[r] & {\X'} \ar[r] ^-g & {\X,}
   }
\end{equation}
where  $Z'' _i$ is smooth over the residue field of $\W$, $q$ et $q'$ are the canonical projections, 
$u''$ is a closed immersion, $T '' _i := a  ^{\prime -1} (T_i '\cap Z' _i)$ is a strict normal crossing divisor of  $Z'' _i$,
  $a '$ is proper, surjective, generically finite and etale. 
  Put $\E ^{\prime \prime (\bullet)} _{i,j} := a ^{\prime !} (\E ^{\prime (\bullet)} _{i,j})
:= 
\R \underline{\Gamma} ^\dag _{Z '' _i} 
q ^{\prime !} (\E ^{\prime (\bullet)} _{i,j}) 
  \in  
  \mathrm{Isoc} ^{(\bullet)} (\widehat{\P} ^N _{\X'}, T ''_i, Z''_i /\W) 
  \cap 
  \smash{\underrightarrow{LM}} _{\Q,\mathrm{ovcoh}} ( \smash{\widehat{\D}} _{\widehat{\P} ^N _{\X '}} ^{(\bullet)})$.
By copying word by word the proof of  \cite[5.3.1.1]{caro-pleine-fidelite}, 
we check that 
$\E ^{\prime (\bullet)} _{i,j}$ is a direct factor of  $q ' _+ ( \E ^{\prime \prime (\bullet)} _{i,j})$.
By construction (see the beginning of the proof of \cite[5.3.1.1]{caro-pleine-fidelite}), 
the morphism
$Z'' _i \to \widehat{\P} ^N _{\X}$ is an immersion 
(indeed, this is the composition of the graph of 
$Z'' _i \to \X$ with the immersion 
$Z'' _i \times \X
\hookrightarrow  
\widehat{\P} ^N _{\X}$
induced by an immersion of the form
$Z'' _i 
\hookrightarrow  
\widehat{\P} ^N _{\W}$).
Since 
$Z'' _i$ is proper over $X$ then 
$Z'' _i \to \widehat{\P} ^N _{\X}$
is more precisely a closed immersion. 
Since $\E ^{\prime \prime (\bullet)} _{i,j}
\in 
  \smash{\underrightarrow{LM}} _{\Q,\mathrm{ovcoh}} ( \smash{\widehat{\D}} _{\widehat{\P} ^N _{\X '}} ^{(\bullet)})$
  has in support in $Z'' _i$ which is smooth, 
  since overcoherence is a local notion, 
we check similarly to \cite[5.1.4]{caro-pleine-fidelite}
that 
$(\widehat{\P} ^N _g) _+
(\E ^{\prime \prime (\bullet)} _{i,j})
\in 
  \smash{\underrightarrow{LM}} _{\Q,\mathrm{ovcoh}} ( \smash{\widehat{\D}} _{\widehat{\P} ^N _{\X}} ^{(\bullet)})$.
Since $q$ is proper, then 
$q _+ $ preserves the overcoherence 
and then 
$g _+ q '_+ ( \E ^{\prime \prime (\bullet)} _{i,j})
\riso 
q _+ \circ  (\widehat{\P} ^N _g) _+
(\E ^{\prime \prime (\bullet)} _{i,j})
\in
  \smash{\underrightarrow{LD}} ^\mathrm{b} _{\Q,\mathrm{ovcoh}} ( \smash{\widehat{\D}} _{\X} ^{(\bullet)})$.
Since 
$\E ^{\prime (\bullet)} _{i,j}$ is a direct factor of  $q ' _+ ( \E ^{\prime \prime (\bullet)} _{i,j})$,
then 
$g _+ (\E ^{\prime (\bullet)} _{i,j})$ is a direct factor of the overcoherent complex $g _+  q ' _+ ( \E ^{\prime \prime (\bullet)} _{i,j})$.
Hence, we are done.

\end{proof}

\begin{rem}
\label{rem-LD-fromD}
With the notation of \ref{stab-propersupp}, 
for any complex
$\E ^{ (\bullet)} \in \smash{\underrightarrow{LD}} ^{\mathrm{b}} _{\Q,\mathrm{coh}} ( \smash{\widehat{\D}} _{\X } ^{(\bullet)})$
such that 
$\underrightarrow{\lim} (  \E ^{ (\bullet)} ) \in D ^{\mathrm{b}}  _{\mathrm{ovcoh}} (\smash{\D} ^\dag _{\X\Q})$
we have 
$\E ^{ (\bullet)} \in \smash{\underrightarrow{LD}} ^{\mathrm{b}} _{\Q,\mathrm{ovcoh}} ( \smash{\widehat{\D}} _{\X } ^{(\bullet)})$
(with notation \ref{limeqcat}, this is Proposition \cite[5.4.3]{caro-stab-sys-ind-surcoh}).
Hence, when $g$ is proper,
since $g _{+} (\E ^{\prime (\bullet)} ) $
belongs to 
$\smash{\underrightarrow{LD}} ^{\mathrm{b}} _{\Q,\mathrm{coh}} ( \smash{\widehat{\D}} _{\X} ^{(\bullet)})$, 
it follows that
Proposition \ref{stab-propersupp} is a straightforward consequence of 
Theorem \cite[2.3.2]{caro-image-directe}. 

But, when $g$ is not proper, this is not a clear consequence.
Indeed, since $g$ is not proper, 
we only know without effort that 
$g _{+} (\E ^{\prime (\bullet)} ) $
belongs to 
$\smash{\underrightarrow{LD}} ^{\mathrm{b}} _{\Q,\mathrm{qc}} ( \smash{\widehat{\D}} _{\X} ^{(\bullet)})$
(the check of the coherence seems as hard as the check of the overcoherence). 
Without coherence hypothesis, we still 
 have the functor
$\underrightarrow{\lim}
\colon 
\smash{\underrightarrow{LD}} ^{\mathrm{b}} _{\Q,\mathrm{qc}} ( \smash{\widehat{\D}} _{\X} ^{(\bullet)})
\to 
D ^{\mathrm{b}}   (\smash{\D} ^\dag _{\X\Q})$.
But, for any complex
$\E ^{ (\bullet)} \in \smash{\underrightarrow{LD}} ^{\mathrm{b}} _{\Q,\mathrm{qc}} ( \smash{\widehat{\D}} _{\X } ^{(\bullet)})$
such that 
$\underrightarrow{\lim} (  \E ^{ (\bullet)} ) \in D ^{\mathrm{b}}  _{\mathrm{coh}} (\smash{\D} ^\dag _{\X\Q})$
(resp. 
$\underrightarrow{\lim} (  \E ^{ (\bullet)} ) \in D ^{\mathrm{b}}  _{\mathrm{ovcoh}} (\smash{\D} ^\dag _{\X\Q})$),
it seems false that this implies that 
$\E ^{ (\bullet)} \in \smash{\underrightarrow{LD}} ^{\mathrm{b}} _{\Q,\mathrm{coh}} ( \smash{\widehat{\D}} _{\X } ^{(\bullet)})$
(resp. $\E ^{ (\bullet)} \in \smash{\underrightarrow{LD}} ^{\mathrm{b}} _{\Q,\mathrm{ovcoh}} ( \smash{\widehat{\D}} _{\X } ^{(\bullet)})$).
\end{rem}

We will need later the following base change isomorphism.

\begin{prop}
Let $\W$ be an object of $\mathrm{DVR}  (\V)$  (see Definition \ref{DVR}). 
Let $f \colon \Y \to \X$, $g \colon \X '\to \X$ be two morphisms of smooth formal $\W$-schemes.
We suppose $f$ smooth.
Let  
$f '\colon \Y \times _{\X} \X '\to \X'$  and 
$g '\colon \Y \times _{\X} \X '\to \Y$ be the structural projections. 
For any 
$\E ^{\prime (\bullet)} \in \smash{\underrightarrow{LD}} ^{\mathrm{b}} _{\Q,\mathrm{ovcoh}} ( \smash{\widehat{\D}} _{\X '} ^{(\bullet)})$
with proper support over $X$, 
we have the base change isomorphism in 
$\smash{\underrightarrow{LD}} ^{\mathrm{b}} _{\Q,\mathrm{ovcoh}} ( \smash{\widehat{\D}} _{\Y} ^{(\bullet)})$ of the form
\begin{equation}
\label{basechange}
 f ^{!} g _{+} (\E ^{\prime (\bullet)} ) 
\riso 
g '_{+}  f ^{\prime !} (\E ^{\prime (\bullet)} ) .
\end{equation}
\end{prop}

\begin{proof}
This is analogue to the proof \cite[5.4.6]{caro-stab-sys-ind-surcoh}:
let $\E ^{\prime (\bullet)} \in \smash{\underrightarrow{LD}} ^{\mathrm{b}} _{\Q,\mathrm{ovcoh}} ( \smash{\widehat{\D}} _{\X '} ^{(\bullet)})$
with proper support over $X$.
First, we remark that by using \ref{stab-propersupp},
both objects of \ref{basechange}
belongs to 
$\smash{\underrightarrow{LD}} ^{\mathrm{b}} _{\Q,\mathrm{ovcoh}} ( \smash{\widehat{\D}} _{\Y} ^{(\bullet)})$.
The morphism $f$ is the composition 
of its graph $\gamma \colon \Y \hookrightarrow \X \times \Y$ with the projection 
$\pi \colon \X \times \Y \to \X$.
Let
$g'' \colon \X '\times \Y \to \X \times \Y $, 
and $\pi '\colon \X' \times \Y '\to \X'$
be the canonical projections. 
Let 
$\gamma '\colon \X '\times _\X \Y \hookrightarrow \X '\times \Y$
be the closed immersion induced by base change
via $g''$ of $\gamma$.
In the second part of the proof of 
\cite[5.4.6]{caro-stab-sys-ind-surcoh}, 
we have proved the isomorphism 
$\pi ^! g _+ (\E ^{\prime (\bullet)}) 
\riso 
g _+ ''\pi ^{\prime !} (\E ^{\prime (\bullet)})$.
This yields the second isomorphism
$\gamma _+ 
f ^! g _+ (\E ^{\prime (\bullet)}) 
\riso 
\gamma _+ \gamma ^!
\pi ^! g _+ (\E ^{\prime (\bullet)}) 
\riso 
\gamma _+ \gamma ^!
g _+ ''\pi ^{\prime !} (\E ^{\prime (\bullet)})$.
Using
Theorem \cite[5.2.8.2]{caro-stab-sys-ind-surcoh}
and Corollary 
\cite[5.3.8]{caro-stab-sys-ind-surcoh}, 
we get the first isomorphism
$\gamma _+ \gamma ^!
g _+ ''\pi ^{\prime !} (\E ^{\prime (\bullet)})
\riso 
g '' _+ \gamma ' _+ \gamma ^{\prime !} \pi ^{\prime !} (\E ^{\prime (\bullet)})
\riso
\gamma  _+  g ' _+ f ^{\prime !} (\E ^{\prime (\bullet)})$.
Hence, by composition, we get the isomorphism
$\gamma _+ 
f ^! g _+ (\E ^{\prime (\bullet)}) 
\riso
\gamma  _+  g ' _+ f ^{\prime !} (\E ^{\prime (\bullet)})$.
Since 
$f ^! g _+ (\E ^{\prime (\bullet)}) $ 
and 
$g ' _+ f ^{\prime !} (\E ^{\prime (\bullet)})$ are coherent (this is a consequence of 
Theorem \ref{stab-propersupp}), 
then we can use Berthelot-Kashiwara theorem in the form \cite[5.3.7.1]{caro-stab-sys-ind-surcoh}.
In other words, 
by applying 
$\gamma ^!$ to the isomorphism
$\gamma _+ 
f ^! g _+ (\E ^{\prime (\bullet)}) 
\riso
\gamma  _+  g ' _+ f ^{\prime !} (\E ^{\prime (\bullet)})$
we get the isomorphism \ref{basechange}.

\end{proof}

\begin{ntn}
\label{ntn-t-structureovcoh}
Let $\W$ be an object of $\mathrm{DVR}  (\V)$. 
Let $\PP$ be a smooth formal $\W$-scheme,
$Y$ be a subscheme of $P$.
We denote by 
$\smash{\underrightarrow{LD}} ^{\mathrm{b}} _{\Q,\mathrm{ovcoh}} (Y , \PP /\W)$
the full subcategory of 
$\smash{\underrightarrow{LD}} ^{\mathrm{b}} _{\Q,\mathrm{ovcoh}} ( \smash{\widehat{\D}} _{\PP } ^{(\bullet)})$
of complexes
$\E ^{(\bullet)}$
such that there exists an isomorphism  of the form
$\R \underline{\Gamma} ^{\dag} _{Y} \E ^{ (\bullet)}
\riso 
\E ^{(\bullet)}$. 

Similarly to \cite[1.2.1-5]{Abe-Caro-weights} , 
there is a canonical t-structure on 
$\smash{\underrightarrow{LD}} ^{\mathrm{b}} _{\Q,\mathrm{ovcoh}} (Y , \PP /\W)$
defined as follows: 
choose $\U$ an open set of $\PP$ such that 
$Y$ is closed in $\U$.
We denote by $\smash{\underrightarrow{LD}} ^{\leq n} _{\Q,\mathrm{ovcoh}}  (Y,\PP/\W)$
and 
$\smash{\underrightarrow{LD}} ^{\geq n} _{\Q,\mathrm{ovcoh}}  (Y,\PP/\W)$
is the full subcategory of 
$\smash{\underrightarrow{LD}} ^{\mathrm{b}} _{\Q,\mathrm{ovcoh}} (Y , \PP /\W)$
of complexes 
$\E$ such that
$\E |\U 
\in 
\smash{\underrightarrow{LD}} ^{\leq n} _{\Q,\mathrm{ovcoh}} (\smash{\D} ^\dag _{\U,\Q})$
(resp. $\E |\U 
\in 
\smash{\underrightarrow{LD}} ^{\leq n} _{\Q,\mathrm{ovcoh}} (\smash{\D} ^\dag _{\U,\Q})$),
where the t-structure on 
$\smash{\underrightarrow{LD}} ^{\leq n} _{\Q,\mathrm{ovcoh}} (\smash{\D} ^\dag _{\U,\Q})$
is the canonical one (see \ref{rem-overhol}.3).
The heart of this t-structure, the category of overcoherent modules on $(Y,\PP/\W)$, 
will be denoted by
$\smash{\underrightarrow{LM}}  _{\Q,\mathrm{ovcoh}}   (Y,\PP/\W)$.
Finally, we denote by 
$\mathcal{H} ^i _{\mathrm{t}}$
the $i$th space of cohomology with respect to this canonical t-structure. 

\end{ntn}

\begin{thm}
[Independence]
\label{ind-ovcoh}
Let $\W$ be an object of $\mathrm{DVR}  (\V)$. 
Let $f \colon \PP '\to \PP$ be a realizable morphism of smooth formal $\W$-schemes.
Let $X'$ be a closed subscheme of $P'$, 
$X$ be a closed subscheme of $P$, 
such that $f (X') \subset X$ and 
the induced morphism $X' \to X$ of schemes is proper. 
Let $Y$ be an open subscheme of $X'$ such that 
the composition 
$Y \to X' \to X$ is an open immersion. 

\begin{enumerate}
\item For any 
$\E ^{(\bullet)} \in \smash{\underrightarrow{LM}}  _{\Q,\mathrm{ovcoh}} (Y , \PP /\W)$,
for any 
$\E ^{\prime(\bullet)} \in \smash{\underrightarrow{LM}}  _{\Q,\mathrm{ovcoh}} (Y , \PP '/\W)$, 
for any $n \in \Z \setminus \{ 0\}$, we have
$$\mathcal{H} _t ^n \R \underline{\Gamma} ^{\dag} _{Y} f ^! ( \E ^{(\bullet)}) =0,
\hspace{1 cm}
\mathcal{H} ^n _t f _+ ( \E ^{\prime (\bullet)}) =0.$$

\item 
For any $\E ^{(\bullet)} \in \smash{\underrightarrow{LD}} ^{\mathrm{b}} _{\Q,\mathrm{ovcoh}} (Y , \PP /\W)$, 
for any $\E ^{\prime (\bullet)} \in \smash{\underrightarrow{LD}} ^{\mathrm{b}} _{\Q,\mathrm{ovcoh}} (Y , \PP '/\W)$, 
we have {\it canonical} isomorphisms of the form
$\R \underline{\Gamma} ^{\dag} _{Y} f ^!  f _+ (\E ^{\prime (\bullet)})
\riso \E ^{\prime (\bullet)}$ and 
$f _+  \R \underline{\Gamma} ^{\dag} _{Y} f ^!  
(\E ^{(\bullet)} )
\riso 
\E ^{(\bullet)} $.
In particular, the functors
$\R \underline{\Gamma} ^{\dag} _{Y} f ^!$
and 
$f _+$ induce quasi-inverse equivalences of categories between 
$\smash{\underrightarrow{LD}} ^{\mathrm{b}} _{\Q,\mathrm{ovcoh}} (Y , \PP /\W)$
and 
$\smash{\underrightarrow{LD}} ^{\mathrm{b}} _{\Q,\mathrm{ovcoh}} (Y , \PP '/\W)$.
\end{enumerate}

\end{thm}

\begin{proof}
With the first part of the Remark \ref{rem-LD-fromD}, 
the first statement is a consequence of 
\cite[4.2.3.2]{caro-image-directe}.
Let us check the second one.

First, we reduce to the case where $f$ is proper. 
Let $u \colon \PP' \hookrightarrow \PP''$ be an immersion of smooth formal $\W$-schemes, 
$\pi \colon \PP'' \to \PP$ 
be a proper morphism of smooth formal $\W$-schemes
such that
$f = \pi \circ u$.
Let $v \colon \PP' \hookrightarrow \U'' $ 
be a closed immersion 
and $j\colon \U'' \hookrightarrow \PP''$ be an open immersion
such that $u = j \circ v$.
Since $X' \to P$ is proper, 
then 
$X'\to U''$ 
and
$X'\to P''$
are closed immersions (because they are proper immersions). 
Since the objects of 
$\smash{\underrightarrow{LD}} ^{\mathrm{b}} _{\Q,\mathrm{ovcoh}} (Y , \PP ''/\W)$
and 
of 
$\smash{\underrightarrow{LD}} ^{\mathrm{b}} _{\Q,\mathrm{ovcoh}} (Y , \U ''/\W)$
have their support in $X'$, 
since the functors $j _*$ and $j ^*$ are quasi-inverse equivalence of categories 
between complexes over $\U''$ (resp. $\PP''$) with support in $X'$, 
since $j _+ = j _*$ and $j ^! = j ^*$ preserve overcoherence (use \ref{stab-propersupp} for $j _+$), then 
$j _+$ and $j ^!$ induces equivalence of categories between 
$\smash{\underrightarrow{LD}} ^{\mathrm{b}} _{\Q,\mathrm{ovcoh}} (Y , \PP ''/\W)$
and 
$\smash{\underrightarrow{LD}} ^{\mathrm{b}} _{\Q,\mathrm{ovcoh}} (Y , \U ''/\W)$
(remark that 
$\R \underline{\Gamma} ^{\dag} _{Y} j ^! = j ^!$
over $\smash{\underrightarrow{LD}} ^{\mathrm{b}} _{\Q,\mathrm{ovcoh}} (Y , \PP ''/\W)$).
Hence, we reduce to the case where $f$ is proper. 

Finally, when $f$ is proper, 
with the first part of the Remark \ref{rem-LD-fromD}, 
 the second statement is 
a consequence of \cite[4.2.3.4]{caro-image-directe}.
\end{proof}

\begin{thm}
[Relative duality isomorphism]
\label{rel-dual-isom}
Let $\W$ be an object of $\mathrm{DVR}  (\V)$. 
Let $g \colon \PP '\to \PP$ be a realizable morphism of smooth formal $\W$-schemes.
For any 
$\E ^{\prime (\bullet)} 
\in 
\smash{\underrightarrow{LD}} ^{\mathrm{b}} _{\Q,\mathrm{ovcoh}} ( \smash{\widehat{\D}} _{\PP '} ^{(\bullet)})$
with proper support over $P$, 
we have the isomorphism of $\smash{\underrightarrow{LD}} ^{\mathrm{b}} _{\Q,\mathrm{coh}} ( \smash{\widehat{\D}} _{\PP} ^{(\bullet)})$
of the form 
$$g _{+} \circ \DD 
(\E ^{\prime (\bullet)} ) 
\riso 
\DD \circ g _{+} 
(\E ^{\prime (\bullet)} ) .
$$
\end{thm}

\begin{proof}
Let $X'$ be a (closed) subscheme of $P'$ which is proper over $P$ via $g$. 
Let $\E ^{\prime (\bullet)} \in \smash{\underrightarrow{LD}} ^{\mathrm{b}} _{\Q,\mathrm{ovcoh}} (X', \PP'/\W)$.
Let $u \colon \PP' \hookrightarrow \PP''$ be an immersion of smooth formal $\W$-schemes, 
$\pi \colon \PP'' \to \PP$ 
be a proper morphism of smooth formal $\W$-schemes
such that
$f = \pi \circ u$.
Let $v \colon \PP' \hookrightarrow \U'' $ 
be a closed immersion 
and $j\colon \U'' \hookrightarrow \PP''$ be an open immersion
such that $u = j \circ v$.

From the relative duality isomorphism in the proper case (see \cite{Vir04}), 
$\DD  v _+ (\E ^{\prime (\bullet)}) \riso 
v _+   \DD (\E ^{\prime (\bullet)})$.
Set $\FF ^{\prime (\bullet)} := v _+ (\E ^{\prime (\bullet)})$.
Since
$j _+ \FF ^{\prime (\bullet)} 
\in 
\smash{\underrightarrow{LD}} ^{\mathrm{b}} _{\Q,\mathrm{ovcoh}} (X', \PP '' /\W) $
has his support in $X'$, then 
$ \DD  j _+ \FF ^{\prime (\bullet)} 
\in 
 \smash{\underrightarrow{LD}} ^{\mathrm{b}} _{\Q,\mathrm{coh}} ( \smash{\widehat{\D}} _{\PP ''} ^{(\bullet)})$
 and has its support in $X'$. 
Hence, 
$ \DD  j _+ \FF ^{\prime (\bullet)} 
\riso 
j _+ j ^!  \DD  j _+ \FF ^{\prime (\bullet)}$. 
Moreover, this is obvious that
$j ^!  \DD  j _+ \FF ^{\prime (\bullet)}
\riso 
\DD  j ^!   j _+ \FF ^{\prime (\bullet)}
\riso 
\DD \FF ^{\prime (\bullet)}
$.
Hence, 
$ \DD  j _+ \FF ^{\prime (\bullet)} 
\riso 
j _+  \DD  \FF ^{\prime (\bullet)}$. 
By composition we get 
$\DD  u _+ (\E ^{\prime (\bullet)}) \riso 
u _+   \DD (\E ^{\prime (\bullet)})$.
Since $\pi$ is proper, 
from the relative duality isomorphism in the proper case (see \cite{Vir04}), 
we obtain the first isomorphism
$\DD  \pi _+  u _+ (\E ^{\prime (\bullet)}) 
\riso 
\pi _+ \DD  u _+ (\E ^{\prime (\bullet)}) 
\riso
\pi _+  u _+ \DD (\E ^{\prime (\bullet)}) $. 
Hence, we are done.

\end{proof}

\subsection{Constructions of stable data of coefficients}

\begin{empt}
Let $\W$ be an object of $\mathrm{DVR}  (\V)$  (see Definition \ref{DVR}). 
Let $f \colon \Y \to \X$ be a morphism of smooth formal $\W$-schemes. 
Following \cite[2.1.4]{caro_surcoherent}, for any 
$\E ^{(\bullet)} \in \smash{\underrightarrow{LD}} ^{\mathrm{b}} _{\Q,\mathrm{qc}} ( \smash{\widehat{\D}} _{\X} ^{(\bullet)})$, 
$\FF ^{(\bullet)} \in \smash{\underrightarrow{LD}} ^{\mathrm{b}} _{\Q,\mathrm{qc}} ( \smash{\widehat{\D}} _{\Y} ^{(\bullet)})$
we have the isomorphism
\begin{equation}
\label{2.1.4surcoh}
f _{+} \left (\FF ^{(\bullet)}
\smash{\overset{\L}{\otimes}}   ^{\dag}
_{\O  _{\Y}} 
f ^{!} (\E^{(\bullet) }) \right )
\riso 
f _{+} (\FF ^{(\bullet)}) 
\smash{\overset{\L}{\otimes}}   ^{\dag}
_{\O  _{\X}} \E^{(\bullet) }  [d _{Y/X}],
\end{equation}
where $d _{Y/X}:=\dim Y -\dim X$.

Let $U$ be a subscheme of $X$. 
Since this is not explicitly written in the literature, let us clarify the following isomorphism.
Using \cite[2.2.6.1, 2.2.8, 2.2.14]{caro_surcoherent}
(or for a wider version with more details, use 
\cite[4.3.6.1,4.4]{caro-stab-sys-ind-surcoh})
for any $\E ^{(\bullet)} \in \smash{\underrightarrow{LD}} ^{\mathrm{b}} _{\Q,\mathrm{qc}} ( \smash{\widehat{\D}} _{\X} ^{(\bullet)})$, 
$\E ^{\prime(\bullet)} \in \smash{\underrightarrow{LD}} ^{\mathrm{b}} _{\Q,\mathrm{qc}} ( \smash{\widehat{\D}} _{\X} ^{(\bullet)})$, 
we have the isomorphisms  
\begin{equation}
\label{localfunct-tensprod}
\R \underline{\Gamma} ^\dag _{U}
\left (
\E ^{\prime (\bullet)}
\smash{\overset{\L}{\otimes}}   ^{\dag}
_{\O  _{\X}} 
\E ^{(\bullet)}
\right )
\riso
\R \underline{\Gamma} ^\dag _{U}
(\E ^{\prime (\bullet)} )
\smash{\overset{\L}{\otimes}}   ^{\dag}
_{\O  _{\X}} 
\E ^{(\bullet)}
\riso 
\E ^{\prime (\bullet)}
\smash{\overset{\L}{\otimes}}   ^{\dag}
_{\O  _{\X}} 
\R \underline{\Gamma} ^\dag _{U} (\E ^{(\bullet)}).
\end{equation}

\end{empt}

\begin{empt}
[Base change and their commutation with cohomological operations]
\label{comm-chg-base}

Let $\alpha \colon \W \to \W' $ be a morphism of $\mathrm{DVR}  (\V)$,
let $\X$ be a smooth formal scheme over $\W$, 
$\E ^{(\bullet)} \in \smash{\underrightarrow{LD}} ^{\mathrm{b}} _{\Q,\mathrm{qc}} ( \smash{\widehat{\D}} _{\X} ^{(\bullet)})$
$\X' := \X \times _{\Spf (\W)} \Spf \W'$ and $\pi  \colon \X' \to \X$ 
be the projection.
The base change of $\E ^{(\bullet)} $ by $\alpha$ is 
the object $\pi  ^{!} (\E ^{(\bullet)} )= 
\pi  ^{! (\bullet)} (\E ^{(\bullet)})$ of 
$\smash{\underrightarrow{LD}} ^{\mathrm{b}} _{\Q,\mathrm{qc}} ( \smash{\widehat{\D}} _{\X '} ^{(\bullet)})$ (see \cite[2.2.2]{Beintro2}).
Similarly to \cite[2.2.2]{Beintro2}, 
it will simply be denoted by 
$ \W'  \smash{\overset{\L}{\otimes}}   ^{\dag}
_{\W}  \E^{(\bullet) }$.

From \cite[2.4.2]{Beintro2}, push forwards commute with base change. 
The commutation of base change with extraordinary pullbacks, local cohomological functors,
duals functors (for coherent complexes), and tensor products is 
straightforward. 

\end{empt}

We will need later the following Lemmas. 
\begin{lem}
\label{lem-stabextpullback}
Let $\mathfrak{C}$ be a data of coefficients stable under local cohomological functors.
Then the data of coefficients 
$\mathfrak{C}$
is stable under smooth extraordinary pullbacks and  satisfies 
$BK ^!$ 
if and only if $\mathfrak{C}$ is stable under extraordinary pullbacks
(see Definitions \ref{dfn-stable-data}). 
\end{lem}

\begin{proof}
Since the converse is obvious, 
let us check that if $\mathfrak{C}$
is stable under smooth extraordinary pullbacks and  satisfies 
$BK ^!$ 
then $\mathfrak{C}$ is stable under extraordinary pullbacks.
Let $\W$ be an object of $\mathrm{DVR}  (\V)$, 
$f \colon \Y \to \X$ be a morphism of  smooth formal schemes over $\W$, 
and $\E ^{(\bullet)}$ be an object of $\mathfrak{C} (\X)$. 
Since $f $ is the composition of its graph $\Y \hookrightarrow \Y\times \X$ 
followed by the projection 
$\Y \times \X \to \X$ which is smooth, 
using the stability under smooth extraordinary pullbacks, we reduce to the case where 
$f$ is a closed immersion. 
From the stability under local cohomological functors, 
$\R \underline{\Gamma} ^\dag _{Y} \E ^{(\bullet)} \in \mathfrak{C} (\X)$. 
Since 
 $\mathfrak{C}$ satisfies $BK ^!$,
 then 
$f ^! \R \underline{\Gamma} ^\dag _{Y} \E ^{(\bullet)}
\in 
\mathfrak{C} (\Y)$.
We conclude using 
 the isomorphism 
$f ^! \R \underline{\Gamma} ^\dag _{Y} \E ^{(\bullet)} \riso 
f ^! ( \E ^{(\bullet)})$ 
(use \cite[5.2.8]{caro-stab-sys-ind-surcoh}). 
\end{proof}

\begin{lem}
\label{S(D,C)stability3}
Let $\mathfrak{D}$ be a data of coefficients over $\V$.
If $\mathfrak{D}$ contains $\mathfrak{B} _\mathrm{div}$ 
(see the second example  of \ref{ex-Dcst}),
and if $\mathfrak{D}$ is stable under 
tensor products,  
then 
$\mathfrak{D}$ is stable 
under localizations outside a divisor.
\end{lem}

\begin{proof}
This is a consequence of the isomorphisms 
\ref{localfunct-tensprod} (we use the case where 
$\E ^{\prime (\bullet)}= \O _{\X} ^{(\bullet)}$).
\end{proof}

\begin{lem}
\label{S(D,C)stability3-pre}
Let $\mathfrak{C}$ be a data of coefficients stable under devissage.
Then the data of coefficients $\mathfrak{C}$ is stable under local cohomological functors if and only if it 
 stable under localizations outside a divisor. 

\end{lem}

\begin{proof}
This is checked by using exact triangles of localisation (see \cite[2.2.6]{caro_surcoherent} or \cite[4.4.3]{caro-stab-sys-ind-surcoh}
for a more precise and general version),
Mayer-Vietoris exact triangles (see \cite[2.2.16]{caro_surcoherent} or \cite[4.5.2]{caro-stab-sys-ind-surcoh}).
\end{proof}

\begin{rem}
\label{rem-div-cst}
Let $\mathfrak{C}$ be a data of coefficients stable under devissage which contains 
$\mathfrak{B} _\mathrm{div}$ (see Notation \ref{ex-cst-surcoh}). 
Then using the arguments of the proof of \ref{S(D,C)stability3-pre},
we check 
$\mathfrak{C}$ contains $\mathfrak{B} _\mathrm{cst}$ (see Notation \ref{ex-cst-surcoh}).
Similarly, we check that 
$S _0 (\mathfrak{B} _\mathrm{div},\mathfrak{C}) 
=
S _0 (\mathfrak{B} _\mathrm{cst},\mathfrak{C}) $
and
$S (\mathfrak{B} _\mathrm{div},\mathfrak{C}) 
=
S (\mathfrak{B} _\mathrm{cst},\mathfrak{C}) $.
\end{rem}

\begin{lem}
\label{dual-BK1-f+}
Let $\mathfrak{C}$ be a data of coefficients. 
If the data of coefficients $\mathfrak{C} $ 
satisfies $BK ^!$,
then so is $\mathfrak{C}  ^{\vee}$ (see Notation \ref{ntn-dual}).
\end{lem}

\begin{proof}
Let  $\W$ be an object of $\mathrm{DVR}  (\V)$, 
  $u \colon \ZZ \hookrightarrow \X$ be a closed immersion of  smooth formal schemes over $\W$, 
$\E ^{(\bullet)}$ be an object of $\mathfrak{C} ^\vee (\X)$ with support in $\ZZ$.
From Berthelot-Kashiwara theorem (see \cite[5.3.6]{caro-stab-sys-ind-surcoh}),
there exists $\G ^{(\bullet)} \in 
\smash{\underrightarrow{LD}} ^{\mathrm{b}} _{\Q,\mathrm{coh}} ( \smash{\widehat{\D}} _{\ZZ} ^{(\bullet)})$
such that 
$u _+ ( \G ^{(\bullet)} ) \riso \E ^{(\bullet)} $.
Since $\DD _{\X} (\E^{(\bullet) }) \in \mathfrak{C} (\X)$ has his support in $\ZZ$, 
since $BK ^!$ property holds, we get 
$u ^! \DD _{\X} (\E^{(\bullet) }) \in  \mathfrak{C} (\ZZ)$.
From the relative duality isomorphism (see \cite{Vir04}),
we get $\DD _{\X} (\E^{(\bullet) })
\riso
\DD _{\X} u _+ ( \G ^{(\bullet)} ) 
\riso u _+ (\DD _{\ZZ} ( \G ^{(\bullet)} ))$.
Hence, 
$u ^! u _+ (\DD _{\ZZ} ( \G ^{(\bullet)} )) \in  \mathfrak{C} (\ZZ)$.
From Berthelot-Kashiwara theorem (see \cite[5.3.6]{caro-stab-sys-ind-surcoh}),
we have $u ^!u _+ (\DD _{\ZZ} ( \G ^{(\bullet)} ))\riso \DD _{\ZZ} ( \G ^{(\bullet)} )$.
This yields
$\DD _{\ZZ} ( \G ^{(\bullet)} )
\in  \mathfrak{C} (\ZZ)$.
Since
$u ^! (\E ^{(\bullet)} ) \riso u ^!u _+ (\G ^{(\bullet)} )\riso \G ^{(\bullet)} $, 
this implies that $u ^! (\E ^{(\bullet)} ) \in  \mathfrak{C} ^{\vee}(\ZZ)$.

\end{proof}

\begin{lem}
\label{cor-rel-dual-isom}
Let $\mathfrak{C}$ be a data of coefficients which is included in $\smash{\underrightarrow{LD}} ^{\mathrm{b}} _{\Q,\mathrm{ovcoh}}$. 
If the data of coefficients $\mathfrak{C} $ is stable under realizable pushforwards,
then so is $\mathfrak{C}  ^{\vee}$.
\end{lem}

\begin{proof}
This is a straightforward consequence of 
the relative duality isomorphism of the form \ref{rel-dual-isom}.
\end{proof}

\begin{lem}
\label{S(D,C)stability}
Let $\mathfrak{C}$ and $\mathfrak{D}$ be two data of coefficients.
With the notation of \ref{dfnS(D,C)}, 
we have the following properties.

\begin{enumerate}
\item 
With Notation \ref{ex-Dcst}, 
if $\mathfrak{D}$ contains $\mathfrak{B} _\emptyset$
(resp. $\mathfrak{B} _\mathrm{div}$) 
then  $S(\mathfrak{D}, \mathfrak{C})$
 is contained in  $ \mathfrak{C}$
 (resp. and in $\smash{\underrightarrow{LD}} ^{\mathrm{b}} _{\Q,\mathrm{ovcoh}}$). 

\item Suppose that $\mathfrak{D} $ is stable under smooth extraordinary pullbacks, base change  and tensor products
 and that $\mathfrak{C} $ contains $\mathfrak{D} $.
Then 
$S(\mathfrak{D}, \mathfrak{C})$ contains $\mathfrak{D}$.
If $\mathfrak{C} $ is moreover stable under shifts then
$S \left (S(\mathfrak{D}, \mathfrak{C}), S(\mathfrak{D}, \mathfrak{C}) \right)$ 
contains $\mathfrak{D}$.

\item If the data of coefficients $\mathfrak{C}$ is local (resp. stable under devissages, resp. stable under direct factors), 
then so is $\mathfrak{C} ^{\vee}$ and 
$S(\mathfrak{D}, \mathfrak{C})$.

\item The data of coefficients $S(\mathfrak{D}, \mathfrak{C})$ is stable under smooth extraordinary pullbacks and under base change.

\item 
\label{S(D,C)stability3bis}
If 
$\mathfrak{D}$ is stable under local cohomological functors (resp. localizations outside a divisor) , then so is 
$S(\mathfrak{D}, \mathfrak{C})$.

\item Suppose that $\mathfrak{C}$ 
is stable under realizable pushforwards and shifts. Suppose that
 $\mathfrak{D}$ contains $\mathfrak{B} _\mathrm{div}$,
and is stable under extraordinary pullbacks.   
Then the data of coefficients $S(\mathfrak{D}, \mathfrak{C})$ is stable under realizable pushforwards.

\item Suppose that $\mathfrak{C}$ is 
stable under shifts, and satisfies $BK ^!$.
Moreover, suppose that $\mathfrak{D}$ 
satisfies $BK _+$. 
Then 
the data of coefficients $S(\mathfrak{D}, \mathfrak{C})$ satisfies $BK ^!$.

\end{enumerate}
\end{lem}

\begin{proof}
a) The respective case of 1) is a consequence of the equality
$ \smash{\underrightarrow{LD}} ^{\mathrm{b}} _{\Q,\mathrm{ovcoh}}= 
S _0 (\mathfrak{B} _\mathrm{div}, \smash{\underrightarrow{LD}} ^{\mathrm{b}} _{\Q,\mathrm{coh}})$). 
The non respective case of  1), 
the assertions 3) and 4) are obvious.

b) Let us prove 2). 
Using every hypotheses on $\mathfrak{C}$,  
we check easily that $S(\mathfrak{D}, \mathfrak{C})$ contains $\mathfrak{D}$.
Let us suppose moreover $\mathfrak{C}$ stable under shifts. 
Since $\mathfrak{D}$ is stable under base change, 
it remains to check 
$\mathfrak{D}$ is included in 
$S _0 \left (S(\mathfrak{D}, \mathfrak{C}), S(\mathfrak{D}, \mathfrak{C}) \right)$.
Let  $\W$ be an object of $\mathrm{DVR}  (\V)$, 
 $\X$ be a  smooth formal scheme over $\W$, 
  $\E ^{(\bullet)} \in \mathfrak{D} (\X)$. 
Let   $f\colon \Y \to \X$  be a smooth morphism of  smooth formal $\W$-schemes, 
$\FF ^{(\bullet)}
\in
S(\mathfrak{D}, \mathfrak{C})(\Y)$.
We have to check that 
$\FF ^{(\bullet)}
\smash{\overset{\L}{\otimes}}   ^{\dag}
_{\O  _{\Y}} f ^{!} (\E^{(\bullet) })
\in
S(\mathfrak{D}, \mathfrak{C})(\Y)$.
Since $\mathfrak{D}$ is stable under base change,
since tensor products and extraordinary inverse images commute with base change, 
we reduce to establish that 
$\FF ^{(\bullet)}
\smash{\overset{\L}{\otimes}}   ^{\dag}
_{\O  _{\Y}} f ^{!} (\E^{(\bullet) })
\in
S _0(\mathfrak{D}, \mathfrak{C})(\Y)$.
Let $g\colon \ZZ \to \Y$  be a smooth morphism of  smooth formal $\W$-schemes, 
let $\G ^{(\bullet)}
\in
\mathfrak{D} (\ZZ)$.
We have the isomorphisms
\begin{align}
\notag
\G ^{(\bullet)}
\smash{\overset{\L}{\otimes}}   ^{\dag}
_{\O  _{\ZZ}} 
 g ^! \left (
 \FF ^{(\bullet)}
\smash{\overset{\L}{\otimes}}   ^{\dag}
_{\O  _{\Y}} 
f ^{!} (\E^{(\bullet) })
\right )
&
\underset{\cite[2.1.9.1]{caro-stab-prod-tens}}{\riso} 
\left ( \G ^{(\bullet)}
\smash{\overset{\L}{\otimes}}   ^{\dag}
_{\O  _{\ZZ}} 
 g ^!  \FF ^{(\bullet)} \right )
\smash{\overset{\L}{\otimes}}   ^{\dag}
_{\O  _{\ZZ}} 
(f \circ g) ^{!} (\E^{(\bullet) })
[-d _{Z/Y}]
\\
\notag
&
\riso
\left ( \G ^{(\bullet)}
\smash{\overset{\L}{\otimes}}   ^{\dag}
_{\O  _{\ZZ}} 
(f \circ g) ^{!} (\E^{(\bullet) })
[-d _{Z/Y}] \right )
\smash{\overset{\L}{\otimes}}   ^{\dag}
_{\O  _{\ZZ}} 
 g ^!  \FF ^{(\bullet)}. 
 \hspace{2cm} 
 (\star)
\end{align}
Since $\mathfrak{D} $ is stable smooth extraordinary pullbacks, shift
and tensor products, then 
$ \G ^{(\bullet)}
\smash{\overset{\L}{\otimes}}   ^{\dag}
_{\O  _{\ZZ}} 
(f \circ g) ^{!} (\E^{(\bullet) })
[-d _{Z/Y}]
\in \mathfrak{D} (\ZZ)$.
Since 
$ \FF ^{(\bullet)} 
\in 
S(\mathfrak{D}, \mathfrak{C})(\Y)$,
then
$\left ( \G ^{(\bullet)}
\smash{\overset{\L}{\otimes}}   ^{\dag}
_{\O  _{\ZZ}} 
(f \circ g) ^{!} (\E^{(\bullet) })
[-d _{Z/Y}] \right )
\smash{\overset{\L}{\otimes}}   ^{\dag}
_{\O  _{\ZZ}} 
 g ^!  \FF ^{(\bullet)}
 \in 
 \mathfrak{C} (\ZZ)$.
Hence, using $(\star)$ we conclude.

c) Let us check 5). From the commutation of the base change with local cohomological functors, we reduce to check that 
$S _0(\mathfrak{D}, \mathfrak{C})$ is stable under local cohomological functors (resp. localisations outside a divisor).
Using \ref{localfunct-tensprod} and the commutation of local cohomological functors with extraordinary inverse images 
(see \cite[5.2.8]{caro-stab-sys-ind-surcoh}), we check the desired properties.

d) Let us check 6). Let $\W$ be an object of $\mathrm{DVR}  (\V)$. 
Let 
$g \colon \X '\to \X$ be a morphism of smooth formal $\W$-schemes.
Let 
$\E ^{\prime (\bullet)} \in S(\mathfrak{D}, \mathfrak{C}) (\X')$
with proper support over $X$. 
From the commutation of the base change with realizable pushforwards 
(see \ref{comm-chg-base}),
we reduce to check
$g _+ (\E^{\prime (\bullet) })
\in
S _0(\mathfrak{D}, \mathfrak{C})(\X)$.
Let $f \colon \Y \to \X$ be a smooth morphism of smooth formal $\W$-schemes.
Let  
$f '\colon \Y \times _{\X} \X '\to \X'$  and 
$g '\colon \Y \times _{\X} \X '\to \Y$ be the structural projections. 
Let   
$\FF ^{(\bullet)}
\in
\mathfrak{D}(\Y)$.
We have to check 
$\FF ^{(\bullet)}
\smash{\overset{\L}{\otimes}}   ^{\dag}
_{\O  _{\Y}} f ^{!} g _+ (\E^{\prime (\bullet) })
\in
 \mathfrak{C}(\Y)$.
Since  $\mathfrak{D}$ contains $\mathfrak{B} _\mathrm{div}$, then
$\E ^{\prime (\bullet)} \in 
\smash{\underrightarrow{LD}} ^{\mathrm{b}} _{\Q,\mathrm{ovcoh}} ( \smash{\widehat{\D}} _{\X '} ^{(\bullet)})$.
Hence from \ref{basechange} we get 
$f ^{!} g _+ (\E^{\prime (\bullet) })
\riso 
g '_{+}  f ^{\prime !} (\E ^{\prime (\bullet)} )$.
Using the hypotheses on 
$\mathfrak{C}$
and 
$\mathfrak{D}$,
via the isomorphisms
$$
\FF ^{(\bullet)}
\smash{\overset{\L}{\otimes}}   ^{\dag} _{\O  _{\Y}} 
f ^{!} g _+ (\E^{\prime (\bullet) })
\underset{\ref{basechange}}{\riso} 
\FF ^{(\bullet)}
\smash{\overset{\L}{\otimes}}   ^{\dag} _{\O  _{\Y}} 
g '_{+}  f ^{\prime !} (\E ^{\prime (\bullet)} )
\underset{\ref{2.1.4surcoh}}{\riso} 
g '_{+} \left ( 
  g ^{\prime !} (\FF ^{(\bullet)})
\smash{\overset{\L}{\otimes}}   ^{\dag} _{\O  _{\Y}} 
  f ^{\prime !} (\E ^{\prime (\bullet)} )
  \right )
  [-d _{X'/X}],$$
we check  that 
$  \FF ^{(\bullet)}
\smash{\overset{\L}{\otimes}}   ^{\dag} _{\O  _{\Y}} 
f ^{!} g _+ (\E^{\prime (\bullet) }) \in
 \mathfrak{C}(\Y)$.

e) Let us check 7) (we might remark the similarity with the proof of \cite[3.1.7]{caro_surcoherent}). 
Let $\W$ be an object of $\mathrm{DVR}  (\V)$, 
$u\colon \X \hookrightarrow \PP$ be a closed immersion of  smooth formal schemes  over $\W$. 
Let $\E ^{(\bullet)} \in S  (\mathfrak{D}, \mathfrak{C}) (\PP)$ with support in $\X$.
We have to check that $u ^! (\E ^{(\bullet)} ) \in S  (\mathfrak{D}, \mathfrak{C}) (\X)$.
We already know that
$u ^! (\E ^{(\bullet)})
\in 
\smash{\underrightarrow{LD}} ^{\mathrm{b}} _{\Q,\mathrm{coh}} ( \smash{\widehat{\D}} _{\X} ^{(\bullet)})$
(thanks to Berthelot-Kashiwara theorem
\cite[5.3.6]{caro-stab-sys-ind-surcoh}).
Since extraordinary pullbacks commute with base change, 
we reduce to check that $u ^! (\E ^{(\bullet)} ) \in S _0  (\mathfrak{D}, \mathfrak{C}) (\X)$.
Let $f\colon \Y \to \X$ be a smooth morphism of  smooth formal $\W$-schemes, 
let $\FF ^{(\bullet)}
\in
\mathfrak{D} (\Y)$.
We have to check 
$\FF ^{(\bullet)}
\smash{\overset{\L}{\otimes}}   ^{\dag}
_{\O  _{\Y}} f ^{!} (u ^! \E^{(\bullet) })
\in
\mathfrak{C} (\Y)$.
The morphism $f$ is the composition of its graph
$\Y \hookrightarrow \Y \times \X $
with the projection  
$\Y \times \X  \to \X$.
We denote by $v$ the composition of 
$\Y \hookrightarrow \Y \times \X $
with 
$id \times u \colon \Y \times \X \hookrightarrow \Y \times\PP$.
Let $g \colon \Y \times\PP \to \PP$ be the projection. 
Set $\U:=  \Y \times\PP$.
Since 
$\mathfrak{D} $ 
satisfies $BK _+$, 
then
$v _+ (\FF ^{(\bullet)}) \in \mathfrak{D} ( \U)$.
Since 
$\E ^{(\bullet)} \in S _0  (\mathfrak{D}, \mathfrak{C}) (\PP)$
and $g$ is smooth,
this yields 
$v _+ (\FF ^{(\bullet)})
\smash{\overset{\L}{\otimes}}   ^{\dag}
_{\O  _{\U}}  g ^{!} (\E^{(\bullet) })
\in \mathfrak{C} (\U)$.
Since $\mathfrak{C} $  satisfies $BK ^!$, 
this implies 
$v ^!\left (v _+ (\FF ^{(\bullet)})
\smash{\overset{\L}{\otimes}}   ^{\dag}
_{\O  _{\U}} g ^{!} (\E^{(\bullet) }) \right ) 
\in \mathfrak{C} (\Y)$.
Since 
$v ^!\left (v _+ (\FF ^{(\bullet)})
\smash{\overset{\L}{\otimes}}   ^{\dag}
_{\O  _{\U}}   g ^{!} (\E^{(\bullet) }) \right ) 
\riso 
v ^! v _+ (\FF ^{(\bullet)})
\smash{\overset{\L}{\otimes}}   ^{\dag}
_{\O  _{\Y}} 
v ^! g ^{!} (\E^{(\bullet) }) [r]$ with $r$ an integer (see \cite[2.1.9.1]{caro-stab-prod-tens}), 
since 
$v ^! v _+ (\FF ^{(\bullet)}) \riso 
\FF ^{(\bullet)}$
(see Berthelot-Kashiwara theorem
\cite[5.3.6]{caro-stab-sys-ind-surcoh}),
since $\mathfrak{C}$ is stable under shifts, 
since by transitivity
$v ^! g ^{!}  \riso f ^! u ^!$,
we get
$\FF ^{(\bullet)}
\smash{\overset{\L}{\otimes}}   ^{\dag}
_{\O  _{\Y}} 
f ^! u ^! (\E^{(\bullet) })
\in 
\mathfrak{C} (\Y)$.

\end{proof}

\begin{dfn}
Let $\mathfrak{D}$ be a data of coefficients over $\V$.
We say that $\mathfrak{D}$ is almost stable under dual functors if the following property holds:
for any data of coefficients $\mathfrak{C}$ over $\V$ which is stable under
devissages, 
direct factors and 
realizable pushforwards,
if $\mathfrak{D} \subset \mathfrak{C}$
then 
$\mathfrak{D} ^{\vee} \subset \mathfrak{C}$.
Remark from the biduality isomorphism that 
the inclusion 
$\mathfrak{D} ^{\vee} \subset \mathfrak{C}$
is equivalent to 
the following one 
$\mathfrak{D} \subset  \mathfrak{C} ^\vee$.

\end{dfn}

\begin{ntn}
\label{dfnqupre}
Let $\mathfrak{C}, \mathfrak{D}$ be two data of coefficients.
We put 
$T _0 (\mathfrak{D} ,\mathfrak{C}) := 
S(\mathfrak{D}  ,\mathfrak{C})$.
By induction on $i \in \N$, we set 
$U _i (\mathfrak{D}  ,\mathfrak{C}):= T _i (\mathfrak{D}  ,\mathfrak{C}) \cap T _i (\mathfrak{D}  ,\mathfrak{C} ) ^{\vee}$, 
$\widetilde{T} _{i} (\mathfrak{D}  ,\mathfrak{C}) := 
S(\mathfrak{D} , U _i (\mathfrak{D}  ,\mathfrak{C}))$
and
$T _{i+1} (\mathfrak{D}  ,\mathfrak{C}):= 
S(\widetilde{T} _{i} (\mathfrak{D}  ,\mathfrak{C}), \widetilde{T} _{i} (\mathfrak{D}  ,\mathfrak{C}) )$.
We put $T (\mathfrak{D}  ,\mathfrak{C}) := \cap _{i\in \N} T _{i} (\mathfrak{D}  ,\mathfrak{C}) $.
\end{ntn}

\begin{thm}
\label{dfnquprop}
Let $\mathfrak{D}$ be a data of coefficients
which contains $\mathfrak{B} _\mathrm{div}$,
which satisfies $BK _+$, 
which is stable under extraordinary pullbacks, base change, 
tensor products 
and which is almost stable under dual functors. 
Let $\mathfrak{C}$ be a data of coefficients 
containing $\mathfrak{D}$, 
which satisfies $BK ^!$, 
is 
stable under 
devissages, direct factors and 
realizable pushforwards. 
Then, 
the data of coefficients $T(\mathfrak{D}, \mathfrak{C})$ 
(see Definition \ref{dfnqupre})
 is included in 
$\mathfrak{C} $,
contains 
$\mathfrak{D} $,
is 
stable by devissages, direct factors, 
local cohomological functors, 
realizable pushforwards, extraordinary pullbacks, base change, tensor products
and duals. 

\end{thm}

\begin{proof}
I) First, we check by induction on $i\in \N $ that the data of coefficients 
$T _i (\mathfrak{D},\mathfrak{C})$ 
contains $\mathfrak{D}$, 
is contained in $\mathfrak{C}$, 
is stable under devissages, direct factors, 
local cohomological functors,
realizable pushforwards, extraordinary pullbacks, base change
(which implies such stability properties for $T(\mathfrak{D},\mathfrak{C})$). 

a)
Let us verify that $T _0 (\mathfrak{D},\mathfrak{C})$ satisfies these properties.
From \ref{S(D,C)stability}.1 (resp. \ref{S(D,C)stability}.2),
$T _0 (\mathfrak{D},\mathfrak{C})$ is included in $\mathfrak{C}$
(resp. contains $\mathfrak{D}$).
From \ref{S(D,C)stability}.3, 
$T _0 (\mathfrak{D},\mathfrak{C})$
is stable under devissages, and under direct factors.
From \ref{S(D,C)stability}.4, 
$T _0 (\mathfrak{D},\mathfrak{C})$
is stable under smooth extraordinary pullbacks and under base change. 
From \ref{S(D,C)stability3} and \ref{S(D,C)stability}.5,
$T _0 (\mathfrak{D},\mathfrak{C})$ is stable under 
localizations outside a divisor. 
Since $T _0 (\mathfrak{D},\mathfrak{C})$ is stable under devissage, 
then from \ref{S(D,C)stability3-pre} 
 $T _0 (\mathfrak{D},\mathfrak{C})$ is stable
 under local cohomological functors.
From  \ref{S(D,C)stability}.6
(resp. \ref{S(D,C)stability}.7), 
$T _0 (\mathfrak{D},\mathfrak{C})$
is stable realizable pushforwards(resp. satisfies $BK ^!$).
Hence, from \ref{lem-stabextpullback}, this implies that  
$T _0 (\mathfrak{D},\mathfrak{C})$
is stable under  extraordinary pullbacks.

b) Suppose that this is true for $T _i (\mathfrak{D},\mathfrak{C})$ for some $i \in \N$.

i) Since 
$\mathfrak{D}$ is almost stable under duals, 
then $U _i (\mathfrak{D}  ,\mathfrak{C})$
contains 
$\mathfrak{D}$.
  Since $\mathfrak{D}$ is stable by tensor products, extraordinary pullbacks, and base change 
then, using \ref{S(D,C)stability}.2
(where $\mathfrak{C}$ is replaced by
$U _i (\mathfrak{D}  ,\mathfrak{C})$),
this implies that $\mathfrak{D}$ is contained in
$\widetilde{T} _{i} (\mathfrak{D}  ,\mathfrak{C})$
and
$T _{i+1} (\mathfrak{D},\mathfrak{C})$.
Using \ref{S(D,C)stability}.1, we get that 
$\widetilde{T} _{i} (\mathfrak{D}  ,\mathfrak{C})$
and
$T _{i+1} (\mathfrak{D},\mathfrak{C})$
are included in $\mathfrak{C}$.

ii) 
From Lemmas \ref{dual-BK1-f+} 
(resp. \ref{cor-rel-dual-isom}, resp. \ref{S(D,C)stability}.3), 
$U _i (\mathfrak{D}  ,\mathfrak{C})$
satisfies $BK ^!$ (resp. is stable under realizable pushforwards,
resp. is stable under devissages, and direct factors).
Hence, using the step I)a) in the case where 
$\mathfrak{C}$ is replaced by 
$U _i (\mathfrak{D}  ,\mathfrak{C})$,
we get 
that $\widetilde{T} _{i} (\mathfrak{D}  ,\mathfrak{C})$
is stable under devissages, direct factors, 
local cohomological functors,
realizable pushforwards, extraordinary pullbacks, base change.
From Lemma \ref{S(D,C)stability}.3
(resp. Lemma \ref{S(D,C)stability}.4, 
resp. Lemma \ref{S(D,C)stability}.5, 
resp. Lemma \ref{S(D,C)stability}.6,
resp. Lemma \ref{S(D,C)stability}.7), 
this yields that 
$T _{i+1} (\mathfrak{D},\mathfrak{C})$
is stable under devissages, and direct factors
(resp. smooth extraordinary pullbacks and base change,
resp. local cohomological functors,
resp. realizable pushforwards,
resp. satisfies $BK ^!$). 
Using \ref{lem-stabextpullback}, this implies that  
$T _{{i+1}} (\mathfrak{D},\mathfrak{C})$
is stable under  extraordinary pullbacks.

II) From \ref{S(D,C)stability}.1, 
$T _{i+1} (\mathfrak{D},\mathfrak{C})$ is contained in 
$\widetilde{T} _{i} (\mathfrak{D},\mathfrak{C})$
and 
$\widetilde{T} _{i} (\mathfrak{D},\mathfrak{C})$
is contained in 
$T _i (\mathfrak{D},\mathfrak{C}) \cap T _i (\mathfrak{D},\mathfrak{C}) ^{\vee}$.
Hence, by construction, the tensor product of two objects of 
$T _{i+1} (\mathfrak{D},\mathfrak{C})$ is an object of 
$T _{i} (\mathfrak{D},\mathfrak{C})$ and
the dual of an object of 
$T _{i+1} (\mathfrak{D},\mathfrak{C})$ is an object of 
$T _{i} (\mathfrak{D},\mathfrak{C})$.
\end{proof}

\begin{rem}
\label{rem-Grothendieck-h}
We keep the notation and hypothesis of \ref{dfnquprop}.
\begin{enumerate}
\item \label{rem-Grothendieck-h1}
From \ref{6operations}, the proposition \ref{dfnquprop} gives a formalism of Grothendieck's six operations and base change 
on couples.
From Remark \ref{rem-overhol}.1--2, we get that the data of coefficients $T(\mathfrak{D} ,\mathfrak{C})$
is contained in $\smash{\underrightarrow{LD}} ^{\mathrm{b}} _{\Q,\mathrm{h}}$. 

\item If the data of coefficients $\mathfrak{C}$ is local, 
then so is $T(\mathfrak{D} ,\mathfrak{C})$.

\end{enumerate}
\end{rem}

\begin{ntn}
\label{TminTmax}
Let $\mathfrak{D}$ and $\mathfrak{C}$ be two data of coefficients
satisfying the hypotheses of the proposition \ref{dfnquprop}.

We define the data of coefficients 
$T _{\mathrm{max}}(\mathfrak{D}, \mathfrak{C})$
(resp. $T _{\mathrm{min}}(\mathfrak{D}, \mathfrak{C})$)
as follows: 
for any object $\W$ of $\mathrm{DVR}  (\V)$, 
for any  smooth formal scheme $\X$ over $\W$,  
the category $T _{\mathrm{max}}(\mathfrak{D}, \mathfrak{C}) (\X)$ 
(resp. $T _{\mathrm{min}}(\mathfrak{D}, \mathfrak{C}) (\X)$) is the full subcategory of 
$\smash{\underrightarrow{LD}} ^{\mathrm{b}} _{\Q,\mathrm{coh}} ( \smash{\widehat{\D}} _{\X} ^{(\bullet)})$
of objects  $\E ^{(\bullet)}$
satisfying the following $\mathrm{max}$ property (resp. $\mathrm{min}$ property): 
\begin{enumerate}
\item [($\mathrm{max}$)]
there exists a data of coefficients $\mathfrak{B}$
which contains
$\mathfrak{D} $, 
which is 
included in 
$\mathfrak{C}$,
stable by devissages, direct factors, local cohomological functors, 
realizable pushforwards, extraordinary pullbacks, base change, tensor products, duals
and such that 
$\E ^{(\bullet)} \in \mathfrak{B} (\X)$.

\item [($\mathrm{min}$)]
for any 
data of coefficients $\mathfrak{B}$
which contains
$\mathfrak{D} $, 
which is 
included in 
$\mathfrak{C}$,
stable by devissages, direct factors, local cohomological functors, 
realizable pushforwards, extraordinary pullbacks, base change, tensor products, and duals
we have 
$\E ^{(\bullet)} \in \mathfrak{B} (\X)$.

\end{enumerate}

\end{ntn}

\begin{thm}

Let $\mathfrak{D}$ and $\mathfrak{C}$ be two data of coefficients
satisfying the hypotheses of the proposition \ref{dfnquprop}.

The data of coefficients 
$T _{\mathrm{max}}(\mathfrak{D}, \mathfrak{C})$
and
$T _{\mathrm{min}}(\mathfrak{D}, \mathfrak{C})$ (see the definition in \ref{TminTmax})
contains
$\mathfrak{D} $, 
are 
included in 
$\mathfrak{C}$,
stable under devissages, direct factors, 
local cohomological functors, 
realizable pushforwards, extraordinary pullbacks, base change, tensor products, duals.
Moreover, they satisfy the following universal property: 
for any data of coefficients $\mathfrak{B}$
which contains
$\mathfrak{D} $, 
is 
included in 
$\mathfrak{C}$,
stable under devissages, direct factors, 
local cohomological functors, 
realizable pushforwards, extraordinary pullbacks, base change, tensor products, duals, 
the data of coefficients 
$T _{\mathrm{max}}(\mathfrak{D}, \mathfrak{C})$
contains $\mathfrak{B}$
and
the data of coefficients 
$T _{\mathrm{min}}(\mathfrak{D}, \mathfrak{C})$
is included in $\mathfrak{B}$.
\end{thm}

\section{Formalism of Grothendieck six operations for arithmetic $\D$-modules over couples}

Let $\W$ be an object of $\mathrm{DVR}  (\V)$ and $l$ be its residue field.

\subsection{Data of coefficients over frames}

\begin{dfn}
\begin{enumerate}
\item We define the category of frames over $\W$ as follows. 
A {\it frame} $(Y,X,\PP)$ 
over $\W$ means that 
$\PP$ is a realizable smooth formal scheme over $\W$, 
$X$ is a closed subscheme of the special fiber $P$ of $\PP$ and $Y$ is an open subscheme of $X$. 
Let 
$(Y', X', \PP')$ 
and $(Y,X,\PP)$ be two frames over $\W$. 
A morphism $\theta= (b,a,f) \colon (Y', X', \PP')\to (Y,X,\PP)$ of frames over $\W$ 
is the data of a  
morphism $f\colon \PP' \to \PP$ of realizable smooth formal schemes over $\W$,
a morphism $a\colon X' \to X$ of $l$-schemes, 
and a morphism $b \colon Y' \to Y$ of $l$-schemes 
inducing the commutative diagram
$$\xymatrix{
{Y'} 
\ar[d] ^-{b}
\ar@{^{(}->}[r] ^-{}
& 
{X'}
\ar[d] ^-{a}
\ar@{^{(}->}[r] ^-{}
& 
{\PP'} 
\ar[d] ^-{f}
\\
{Y} 
\ar@{^{(}->}[r] ^-{}
& 
{X}
\ar@{^{(}->}[r] ^-{}
& 
{\PP.} 
}$$
If there is no ambiguity with $\W$, we simply say frame or morphism of frames. 

\item A morphism 
$\theta= (b,a,f) \colon (Y', X', \PP')\to (Y,X,\PP)$ of frames over $\W$ 
is said to be {\it complete} 
(resp. {\it strictly complete})
if $a$ is proper (resp. 
$f$ and $a$ are proper). 

\end{enumerate}

\end{dfn}

\begin{dfn}
\begin{enumerate}
\item We define the category of {\it couples} over $\W$ as follow. 
A couple $(Y, X)$ over $\W$ means the two first data of a frame 
over $\W$
of the form $(Y,X, \PP)$.
A frame of the form $(Y,X, \PP)$ is said to be enclosing $(Y,X)$. 
A morphism of couples $u=(b,a)\colon (Y', X') \to (Y, X)$ over $\W$
is the data of a morphism of $l$-schemes of the form 
$a\colon X' \to X$ such that $a (Y' ) \subset Y$ and $b\colon Y' \to Y$ is the induced morphism.

\item A morphism of couples $u=(b,a)\colon (Y', X') \to (Y, X)$ over $\W$
is said to be {\it complete} if $a$ is proper.

\end{enumerate}
 
\end{dfn}

\begin{rem}
\label{rem-complete-frame-coup}
\begin{enumerate}

\item Let $u=(b,a)\colon (Y', X') \to (Y, X)$ be a complete morphism of couples  over $\W$.
Then there exists a strictly complete morphism  of frames over $\W$ of the form
$\theta= (b,a,f) \colon (Y', X', \PP')\to (Y,X,\PP)$.
Indeed, by definition, there exist some frames over $\W$ of the form
$(Y', X', \PP'')$ and 
$(Y,X,\PP)$. 
There exists an immersion 
$\PP'' \hookrightarrow \QQ ''$ with $\QQ''$ a proper and smooth formal $\W$-scheme. 
Hence, put $\PP':= \QQ '' \times \PP$ and let $f\colon \PP' \to \PP$ be the projection.
Since $a$ is proper, $X \hookrightarrow \PP$ is proper, 
and $f$ is proper, then the immersion $X' \hookrightarrow \PP'$ is also proper. 

\item Let $u=(b,a)\colon (Y', X') \to (Y, X)$ be a morphism of couples  over $\W$.
Similarly, we check that there exists a morphism  of frames over $\W$ of the form
$\theta= (b,a,f) \colon (Y', X', \PP')\to (Y,X,\PP)$. 

\end{enumerate}

\end{rem}

\begin{ntn}
\label{ntn-6operations}
Let $\mathfrak{C} $ be a data of coefficients over $\V$.
Let  $(Y, X,\PP)$ be a frame over $\W$. 
We denote by $\mathfrak{C} (Y, \PP/\W)$
the full subcategory of 
$\mathfrak{C}  (\PP)$ 
of objects $\E$ such that there exists an isomorphism of the form 
$\E \riso \R \underline{\Gamma} ^\dag _{Y} (\E)$.
\end{ntn}

\begin{ntn}
\label{ntn-Holpre}
Let  $(Y, X,\PP)$ be a frame over $\W$. 
The full subcategory of
$D ^{\mathrm{b}}  _{\mathrm{coh}} (\smash{\D} ^\dag _{\PP,\Q})$
which is the essential of $\smash{\underrightarrow{LD}} ^{\mathrm{b}} _{\Q,\mathrm{h}} (Y,\PP/\W)$ via the equivalence \ref{limeqcat}
will be denoted by $D^{\mathrm{b}}_{\mathrm{h}} (Y,\PP/\W)$.
Recall 
that $D ^{\mathrm{b}}  _{\mathrm{h}} (\smash{\D} ^\dag _{\PP,\Q})$
is endowed with a canonical t-structure induced by that 
of $D ^{\mathrm{b}}  _{\mathrm{coh}} (\smash{\D} ^\dag _{\PP,\Q})$ (see Remark \ref{rem-overhol}.3). 
Similarly to 
\ref{ntn-t-structureovcoh}, 
there is a canonical t-structure on 
$D^{\mathrm{b}}_{\mathrm{h}} (Y,\PP/\W)$
defined as follows: 
choose $\U$ an open set of $\PP$ such that 
$Y$ is closed in $\U$.
Then $ D^{\leq n} _{\textrm{h}}(Y,\PP/\W)$
and 
$D^{\geq n}_{\textrm{h}}(Y,\PP/\W)$
is the full subcategory of 
$D^{\mathrm{b}}_{\mathrm{h}} (Y,\PP/\W)$
of complexes 
$\E$ such that
$\E |\U 
\in 
D ^{\leq n}  _{\mathrm{h}} (\smash{\D} ^\dag _{\U,\Q})$
(resp. $\E |\U 
\in 
D ^{\geq n}  _{\mathrm{h}} (\smash{\D} ^\dag _{\U,\Q})$),
where the t-structure on 
$D ^\mathrm{b} _{\mathrm{h}} (\smash{\D} ^\dag _{\U,\Q})$
is the canonical one. 
The heart of this t-structure, the category of overholonomic modules on $(Y,\PP/\W)$ after any base change, 
will be denoted by
$\mathrm{H} (Y,\PP/\W)$.
Finally, we denote by 
$\mathcal{H} ^i _{\mathrm{t}}$
the $i$th space of cohomology with respect to this canonical t-structure. 

\end{ntn}

\begin{dfn}
\label{IsocSigmapre}
Let $(Y,X, \PP)$ be a frame over $\W$ with $Y$ smooth. 

\begin{enumerate}

\item 
\label{IsocSigma1pre}
Choose 
$\U$ an open set of $\PP$ such that 
$Y$ is closed in $\U$. 
Let $\E \in \mathrm{H} (Y,\PP/\W)$ (see the notation of \ref{ntn-Holpre}).
We say that $\E$ is an {\it overholonomic after any base change isocrystals} on $(Y,\PP/\W)$
if $\E | \U \in \mathrm{Isoc} ^{\dag \dag}(Y, \U/\W)$ (see the notation of \ref{Isodagdagdiv}, remark
we use the case where the divisor is empty).
We denote by 
$\textrm{H-Isoc} ^{\dag \dag}(Y,\PP/\W)$ 
the full subcategory of 
$\mathrm{H} (Y,\PP/\W)$
whose objects are overholonomic after any base change isocrystals on $(Y,\PP/\W)$.

\item 
Let
$D^{\mathrm{b}}_{\textrm{h-isoc}} (Y,\PP/\W)$
be the full subcategory of 
$D^{\mathrm{b}}_{\mathrm{h}} (Y,\PP/\W)$
of the objects $\E$ such that,
for any integer $i$, the module 
$\mathcal{H} ^i _{\mathrm{t}}( \E ) \in \textrm{H-Isoc} ^{\dag \dag}(Y,\PP/\W)$,
where $\mathcal{H} ^i _{\mathrm{t}}$ means the $i$th spaces of cohomology with respect to the canonical t-structure 
(see the notation of \ref{ntn-Holpre}). 
The canonical t-structure on 
$D^{\mathrm{b}}_{\textrm{h}} (Y,\PP/\W)$
induces canonically another one on 
$D^{\mathrm{b}}_{\textrm{h-isoc}} (Y,\PP/\W)$.
For any integer $n \in \Z$, 
we get the subcategories
$D^{\leq n}_{\textrm{h-isoc}}(Y,\PP/\W) := D^{\mathrm{b}}_{\textrm{h-isoc}}(Y,\PP/\W) \cap D^{\leq n}_{\textrm{h}}(Y,\PP/\W)$
and 
$D^{\geq n}_{\textrm{h-isoc}}(Y,\PP/\W) := D^{\mathrm{b}}_{\textrm{h-isoc}}(Y,\PP/\W) \cap D^{\geq n}_{\textrm{h}}(Y,\PP/\W)$.

\item We denote by 
$\underrightarrow{LD} ^{\mathrm{b}} _{\Q,  \textrm{h-isoc}}(Y,\PP/\W)$
the full subcategory of 
$\smash{\underrightarrow{LD}} ^{\mathrm{b}} _{\Q,\mathrm{coh}} ( \smash{\widehat{\D}} _{\PP} ^{(\bullet)})$
of objects
$\E ^{(\bullet)}$ such that
$\underrightarrow{\lim}~
\E ^{ (\bullet)}
\in 
D^{\mathrm{b}}_{\textrm{h-isoc}} (Y,\PP/\W)$.

\end{enumerate}
\end{dfn}

\begin{empt}
\label{dev-h2hisco}
Let $(Y,X, \PP)$ be a frame over $\W$. 
Let 
$\E ^{(\bullet)}\in \underrightarrow{LD} ^{\mathrm{b}} _{\Q,  \textrm{h}}(Y,\PP/\W)$.
Then there exists a smooth $d$-stratification 
$(Y _{i}) _{i=1, \dots , r}$ of $Y$ in $P$ (see Definition \cite[4.1.2.2]{caro-stab-prod-tens}) 
such that we have
$ \R \underline{\Gamma} ^\dag _{Y _i}  (\E ^{(\bullet)}) 
\in 
\underrightarrow{LD} ^{\mathrm{b}} _{\Q,  \textrm{h-isoc}} (Y _i,\PP/\W)$,
for any for any $i =1,\dots, r$ (see \cite[6.2.3]{caro-pleine-fidelite}).
\end{empt}

\subsection{Formalism of Grothendieck six operations over couples}

\begin{thm}
[Independence of the frame enclosing a couple]
\label{ind-CYW}
Let $\mathfrak{C}$ be a data of coefficients over $\V$
which contains $\mathfrak{B} _\mathrm{div}$, 
which is 
stable  under devissage, 
realizable pushforwards, extraordinary pullbacks,
and under local cohomological functors.
Let $\theta= (id,a,f) \colon (Y, X', \PP')\to (Y,X,\PP)$ be a complete morphism of frames over $\W$.

The functors
$\R \underline{\Gamma} ^{\dag} _{Y} f ^!$
and 
$f _+$ induce quasi-inverse equivalences of categories between 
$\mathfrak{C} (Y , \PP /\W)$
and 
$\mathfrak{C} (Y , \PP '/\W)$ (recall notation \ref{ntn-6operations}).
\end{thm}

\begin{proof}
Using the stability properties that $\mathfrak{C}$ satisfies, 
we check that the functors 
$f _+\colon 
\mathfrak{C} (Y , \PP' /\W)
\to 
\mathfrak{C} (Y , \PP /\W)$
and
$\R \underline{\Gamma} ^{\dag} _{Y} f ^!
\colon 
\mathfrak{C} (Y , \PP /\W)
\to 
\mathfrak{C} (Y , \PP '/\W)$
are well defined. 
Since $\mathfrak{C}$ is included in $\smash{\underrightarrow{LD}} ^{\mathrm{b}} _{\Q,\mathrm{ovcoh}} $, 
then this is a straightforward consequence of Theorem \ref{ind-ovcoh}.
\end{proof}

\begin{lem}
\label{ind-cat-overcouples}
Let $\mathfrak{C}$ be a data of coefficients over $\V$
which contains $\mathfrak{B} _\mathrm{div}$, 
which is stable  under devissage, 
realizable pushforwards, extraordinary pullbacks,
and local cohomological functors.
Let $\mathbb{Y}:= (Y,X)$ be a couple over $\W$. Choose a frame of the form $(Y, X,\PP)$. 
The category 
$\mathfrak{C} (Y, \PP/\W)$
 does not depend, up to a canonical equivalence of categories, 
on the choice of the frame $(Y,X, \PP)$ over $\W$ enclosing $(Y,X)$.
Hence, we can simply write 
$\mathfrak{C} (\mathbb{Y}/\W)$ 
instead of 
$\mathfrak{C} (Y, \PP/\W)$
without ambiguity (up to canonical equivalence of categories).
\end{lem}

\begin{proof}
Let $(Y,X, \PP _1)$ and $(Y,X, \PP _2)$ be
two frames over $\W$ enclosing $(Y,X)$.
The closed immersions 
$X \hookrightarrow \PP _1$
and
$X \hookrightarrow \PP _2$
induce
$X \hookrightarrow \PP _1 \times \PP _2$.
Denoting by 
$\pi _1\colon \PP _1 \times \PP _2 \to \PP _1$
and 
$\pi _2\colon \PP _1 \times \PP _2 \to \PP _1$
the structural projections, 
we get two morphisms of frames over $\W$ of the form
$(id,id, \pi _1) \colon 
(Y,X, \PP _1 \times \PP _2)
\to 
(Y,X, \PP _1)$
and 
$(id,id, \pi _2) \colon 
(Y,X, \PP _1 \times \PP _2)
\to 
(Y,X, \PP _2)$.
From \ref{ind-CYW}, 
the functors 
$\pi _{2 +} \R \underline{\Gamma} ^{\dag} _{Y} \pi _1 ^! $
and
$ \pi _{1+} \R \underline{\Gamma} ^{\dag} _{Y} \pi _2 ^!$
are quasi-inverse equivalences of categories 
between 
$\mathfrak{C} (Y, \PP_1/\W)$
and 
$\mathfrak{C} (Y, \PP _2/\W)$.
\end{proof}

\begin{lem}
\label{ind-dual}
Let $\mathfrak{C}$ be a data of coefficients over $\V$
which contains $\mathfrak{B} _\mathrm{div}$, 
which is stable  under devissage, 
realizable pushforwards, extraordinary pullbacks, 
local cohomological functors, and duals.
Let $\mathbb{Y}:= (Y,X)$ be a couple over $\W$. 
 Choose a frame of the form $(Y, X,\PP)$. 
The functor 
$\R \underline{\Gamma} ^\dag _{Y} \DD _{\PP} 
\colon 
\mathfrak{C} (Y,\PP/\W) \to \mathfrak{C} (Y,\PP/\W)$
does not depend,
up to canonical isomorphism of 
\ref{ind-cat-overcouples} (more precisely, we have the commutative diagram
\ref{ind-dual-diag1} up to canonical isomorphism), 
on the choice of the frame enclosing $(Y,X)$. 
Hence, 
we will denote by 
$\DD _{\mathbb{Y}}
\colon \mathfrak{C} (\mathbb{Y}/\W) \to \mathfrak{C} (\mathbb{Y}/\W)$
the functor 
$\R \underline{\Gamma} ^\dag _{Y} \DD _{\PP}$. 
\end{lem}

\begin{proof}
As in the beginning of the proof, 
\ref{ind-cat-overcouples},
let $(Y,X, \PP _1)$ and $(Y,X, \PP _2)$ be
two frames over $\W$ enclosing $(Y,X)$.
Let $\pi _1\colon \PP _1 \times \PP _2 \to \PP _1$
and 
$\pi _2\colon \PP _1 \times \PP _2 \to \PP _2$
be the structural projections.
We have to check that the diagram
\begin{equation}
\label{ind-dual-diag1}
\xymatrix{
{\mathfrak{C} (Y, \PP_1/\W)} 
\ar[r] _-{\R \underline{\Gamma} ^{\dag} _{Y} \pi _1 ^! } ^-{\cong}
\ar[d] _-{\R \underline{\Gamma} ^\dag _{Y} \DD _{\PP _1} }
&
{\mathfrak{C} (Y, \PP_1\times \PP _2/\W)} 
\ar[r] _-{\pi _{2 +} } ^-{\cong}
\ar[d] ^-{\R \underline{\Gamma} ^\dag _{Y} \DD _{\PP _1 \times \PP _2} }
& 
{\mathfrak{C} (Y, \PP_2/\W)} 
\ar[d] ^-{\R \underline{\Gamma} ^\dag _{Y} \DD _{\PP _2} }
\\ 
{\mathfrak{C} (Y, \PP_1/\W) } 
\ar[r] ^-{ \R \underline{\Gamma} ^{\dag} _{Y} \pi _1 ^! } _-{\cong}
&
{\mathfrak{C} (Y, \PP_1\times \PP _2/\W)} 
\ar[r] ^-{\pi _{2 +}} _-{\cong}
& 
{\mathfrak{C} (Y, \PP_2/\W) } 
}
\end{equation}
is commutative, up to canonical isomorphism.
Let 
$\E ^{(\bullet)}
\in 
\mathfrak{C} (Y, \PP_1\times \PP _2/\W) $.
From
\ref{rel-dual-isom}, we have the isomorphism
$\DD _{\PP _2} \pi _{2+} ( \E ^{(\bullet)})
\riso 
\pi _{2+} \DD _{\PP _1 \times \PP _2}  ( \E ^{(\bullet)}) $.
Hence, by applying the functor $\R \underline{\Gamma} ^{\dag} _{Y}$ to this isomorphism, 
we get the first one 
$\R \underline{\Gamma} ^{\dag} _{Y}\DD _{\PP _2} \pi _{2+} ( \E ^{(\bullet)})
\riso 
 \R \underline{\Gamma} ^{\dag} _{Y}\pi _{2+} \DD _{\PP _1 \times \PP _2}  ( \E ^{(\bullet)}) 
\underset{\cite[5.2.8]{caro-stab-sys-ind-surcoh}}{\riso}  
\pi _{2+}  \R \underline{\Gamma} ^{\dag} _{\pi _2 ^{-1}(Y)} \DD _{\PP _1 \times \PP _2}  ( \E ^{(\bullet)}) $.
Since $Y \hookrightarrow\pi _2 ^{-1}(Y)$ is a closed immersion (recall formal schemes are separated by convention), 
then $Y = \overline{Y}\cap \pi _2 ^{-1}(Y)$, where 
$\overline{Y}$ is the closure of $Y$ in $P _1 \times P _2$.
Since  $ \DD _{\PP _1 \times \PP _2}  ( \E ^{(\bullet)}) $
has in support in $\overline{Y}$,
then 
$\R \underline{\Gamma} ^{\dag} _{\pi _2 ^{-1}(Y)} \DD _{\PP _1 \times \PP _2}  ( \E ^{(\bullet)}) 
\riso 
\R \underline{\Gamma} ^{\dag} _{\pi _2 ^{-1}(Y)}
\R \underline{\Gamma} ^{\dag} _{\overline{Y}} \DD _{\PP _1 \times \PP _2}  ( \E ^{(\bullet)}) 
\riso 
\R \underline{\Gamma} ^{\dag} _{Y} \DD _{\PP _1 \times \PP _2}  ( \E ^{(\bullet)}) 
$.
Hence, we have checked the commutativity, up to commutative isomorphism,
of the right square of \ref{ind-dual-diag1}. 
From \ref{ind-CYW},  $\pi _{1+}$ is canonically a quasi-inverse 
of the equivalence of categories $\R \underline{\Gamma} ^{\dag} _{Y} \pi _1 ^! 
\colon 
\mathfrak{C} (Y, \PP_1\times \PP _2/\W)
\cong
\mathfrak{C} (Y, \PP_1/\W) $
(we means that we have canonical isomorphisms
$\pi _{1+} \R \underline{\Gamma} ^{\dag} _{Y} \pi _1 ^! 
\riso id$
and 
$id \riso 
\R \underline{\Gamma} ^{\dag} _{Y} \pi _1 ^! \pi _{1+}$).
Hence, 
we get the commutativity, up to canonical isomorphism,
of the left square of \ref{ind-dual-diag1}. 
\end{proof}

\begin{lem}
\label{ind-pushforward-extinv}
Let $\mathfrak{C}$ be a data of coefficients over $\V$
which contains $\mathfrak{B} _\mathrm{div}$, 
which is stable  under devissage, 
realizable pushforwards, extraordinary pullbacks,
and local cohomological functors.
Let  $u=(b,a)\colon (Y', X') \to (Y, X)$ be a morphism of couples over $\W$.
Put $\mathbb{Y}:= (Y,X)$ and $\mathbb{Y}':= (Y',X')$.
Let us choose a morphism of frames $\theta= (b,a,f) \colon (Y', X', \PP')\to (Y,X,\PP)$ over $\W$ enclosing $u$.

\begin{enumerate}
\item 
The functor $\theta ^! := \R \underline{\Gamma} ^\dag _{Y'} \circ f ^{!}
\colon
\mathfrak{C} (Y,\PP/\W) \to \mathfrak{C} (Y',\PP'/\W)$ 
does not depend on the choice of such $\theta$ enclosing $u$
(up to canonical equivalences
of categories).
Hence, it will be denoted by $u ^{!}
\colon \mathfrak{C} (\mathbb{Y}/\W) \to \mathfrak{C} (\mathbb{Y}'/\W)$. 
\item Suppose that $u$ is complete, i.e. that $a \colon X' \to X$ is proper. 
The functor $\theta _{+}:= f _+
\colon 
\mathfrak{C} (Y',\PP'/\W) \to \mathfrak{C} (Y,\PP/\W)$.
does not depend on the choice of such $\theta$ enclosing $u$
(up to canonical equivalences
of categories).
Hence, it will be denoted by 
$u _{+} \colon \mathfrak{C} (\mathbb{Y}'/\W) \to \mathfrak{C} (\mathbb{Y}/\W)$.
\end{enumerate}

\end{lem}

\begin{proof}
To check the first assertion, 
we proceed as in the proof of 
\ref{ind-dual}
(use also the commutation of local cohomological functors with extraordinary inverse images 
given in \cite[5.2.8]{caro-stab-sys-ind-surcoh}).
Let us check that the functor 
$f _+
\colon 
\mathfrak{C} (Y',\PP'/\W) \to \mathfrak{C} (Y,\PP/\W)$
is well defined.
Let $\E ^{(\bullet)} \in \mathfrak{C} (Y',\PP'/\W) $.
Since $a$ is proper, then 
$f _+ (\E ^{(\bullet)} ) \in \mathfrak{C} (\PP') $.
We compute 
$\R \underline{\Gamma} ^\dag _{Y}f _+ (\E ^{(\bullet)} ) 
\riso 
f _+ \R \underline{\Gamma} ^\dag _{ f ^{-1}Y} (\E ^{(\bullet)} ) $.
Since $Y'$ is included in $f ^{-1}Y$ and 
$\E ^{(\bullet)} \in \mathfrak{C} (Y',\PP'/\W) $,
then 
$\R \underline{\Gamma} ^\dag _{ f ^{-1}Y} (\E ^{(\bullet)} ) 
\riso 
\R \underline{\Gamma} ^\dag _{Y} (\E ^{(\bullet)} ) $.
Hence, 
$\R \underline{\Gamma} ^\dag _{Y}f _+ (\E ^{(\bullet)} ) 
\riso f _+ (\E ^{(\bullet)} )$, which implies that
$f _+ (\E ^{(\bullet)} ) \in  \mathfrak{C} (Y,\PP/\W)$.
To check that the functor does not depend on the choice of $\theta$ enclosing $u$, 
we proceed as in the proof of 
\ref{ind-dual}.
\end{proof}

\begin{lem}
\label{ind-prod-tensor}
Let $\mathfrak{C}$ be a data of coefficients over $\V$
which contains $\mathfrak{B} _\mathrm{div}$, 
which is stable under devissage, 
realizable pushforwards, extraordinary pullbacks, 
and tensor products.
Let $\mathbb{Y}:= (Y,X)$ be a couple over $\W$. 
 Choose a frame of the form $(Y, X,\PP)$. 
The bifunctor $-\smash{\overset{\L}{\otimes}}   ^{\dag}
_{\O  _{\PP}} - [-\dim P]
\colon
\mathfrak{C} (Y,\PP/\W) \times \mathfrak{C} (Y,\PP/\W) \to \mathfrak{C} (Y,\PP/\W)$
does not depend, up to the canonical equivalences of categories of \ref{ind-cat-overcouples}, 
on the choice of the frame enclosing $(Y,X)$.
It will be denoted by 
$\widetilde{\otimes} _{\mathbb{Y}}
\colon 
\mathfrak{C} (\mathbb{Y}/\W) \times \mathfrak{C} (\mathbb{Y}/\W) \to \mathfrak{C} (\mathbb{Y}/\W)$. 
\end{lem}

\begin{proof}
From Lemmas \ref{S(D,C)stability3} and \ref{S(D,C)stability3-pre}, 
the data of coefficients 
$\mathfrak{C}$  is also stable under local cohomological functors.
From \cite[2.1.9.1]{caro-stab-prod-tens} 
(resp. \ref{localfunct-tensprod}), 
extraordinary inverse images 
(resp. local cohomological functors) 
commute with tensor products (up to a shift).
Proceeding as in the proof of \ref{ind-dual} with its notation,
$\R \underline{\Gamma} ^{\dag} _{Y} \pi _1 ^! $ 
and 
$\R \underline{\Gamma} ^{\dag} _{Y} \pi _2 ^! $ 
commute with tensor products
and then so are 
$\pi _{1+}$ and $\pi _{2+}$.

\end{proof}

\begin{empt}
[Formalism of Grothendieck six operations]
\label{6operations}
Let $\mathfrak{C}$ be a data of coefficients over $\V$
which contains $\mathfrak{B} _\mathrm{div}$, 
which is 
stable  under devissage, 
realizable pushforwards, extraordinary pullbacks, 
duals,
and tensor products.
To sum-up the above Lemmas
we can define a formalism of Grothendieck operations on couples 
as follows.
Let  $u=(b,a)\colon (Y', X') \to (Y, X)$ be a morphism of couples over $\W$.
Put $\mathbb{Y}:= (Y,X)$ and $\mathbb{Y}':= (Y',X')$.
\begin{enumerate}

\item We have the dual functor $\DD _{\mathbb{Y}}
\colon \mathfrak{C} (\mathbb{Y}/\W) \to \mathfrak{C} (\mathbb{Y}/\W)$ (see \ref{ind-dual}).

\item 
We have the extraordinary pullback $u ^{!}
\colon
\mathfrak{C} (\mathbb{Y}/\W) \to \mathfrak{C} (\mathbb{Y}'/\W)$ (see \ref{ind-pushforward-extinv}).
We get the pullbacks $u ^{+}:= \DD _{\mathbb{Y}'} \circ u ^{!} \circ \DD _{\mathbb{Y}}$.

\item Suppose that $u$ is complete.
Then, we have the functor $u _+
\colon \mathfrak{C} (\mathbb{Y}'/\W) \to \mathfrak{C} (\mathbb{Y}/\W)$ (see \ref{ind-pushforward-extinv}) 
We denote by  $u _{!}:= \DD _{\mathbb{Y}} \circ u _{+} \circ \DD _{\mathbb{Y}'}$, 
the extraordinary pushforward by $u$.

\item 
We have the tensor product 
$-\widetilde{\otimes} _{\mathbb{Y}}-
\colon 
\mathfrak{C} (\mathbb{Y}/\W) \times \mathfrak{C} (\mathbb{Y}/\W) \to \mathfrak{C} (\mathbb{Y}/\W)$
(see \ref{ind-prod-tensor})
\end{enumerate}

\end{empt}

\begin{ex}
\label{nota-h-ovhol}
 We recall
the data of coefficients $\smash{\underrightarrow{LD}} ^{\mathrm{b}} _{\Q,\mathrm{ovhol}}$
and $\smash{\underrightarrow{LD}} ^{\mathrm{b}} _{\Q,\mathrm{h}}$.
are defined respectively in \ref{ex-cst-surcoh}.\ref{hstab} and \ref{ex-cst-surcoh}.\ref{ovholstab}.
Using Lemmas \ref{lem-stabextpullback} and \ref{S(D,C)stability} (and Remark \ref{rem-div-cst}) are stable 
under local cohomological functors, 
realizable pushforwards, extraordinary pullbacks, and duals.
Hence, 
with the notation 
\ref{ind-cat-overcouples},
using Lemmas \ref{ind-pushforward-extinv}, \ref{ind-prod-tensor}, and \ref{ind-dual},
for any frame 
$(Y,X, \PP)$ over $\W$,
we get the categories of the form 
$\smash{\underrightarrow{LD}} ^{\mathrm{b}} _{\Q,\mathrm{h}} (Y, \PP /\W)$,
$\smash{\underrightarrow{LD}} ^{\mathrm{b}} _{\Q,\mathrm{h}}(\mathbb{Y}/\W)$,
$\smash{\underrightarrow{LD}} ^{\mathrm{b}} _{\Q,\mathrm{ovhol}} (Y, \PP /\W)$
or
$\smash{\underrightarrow{LD}} ^{\mathrm{b}} _{\Q,\mathrm{ovhol}}(\mathbb{Y}/\W)$
endowed with five of Grothendieck cohomological operations (the tensor product is a priori missing).
We keep in this context the notation \ref{6operations}.1--3 concerning
these five functors

\end{ex}

\begin{ntn}
\label{ntn-Hol}
Let $(Y,X, \PP)$ be a frame over $\W$. 
\begin{enumerate}
\item With the notation of  \ref{ntn-Holpre},
the category 
$D^{\mathrm{b}}_{\mathrm{h}} (Y,\PP/\W)$
does not depend on the choice of the frame $(Y,X, \PP)$ enclosing the couple $\mathbb{Y}:=(Y,X)$
(up to canonical equivalences of categories).
Hence, it will be denoted by
$D^{\mathrm{b}}_{\mathrm{h}} (\mathbb{Y}/\W)$ without any ambiguity.

\item 
\label{ntn-Hol3} 
From \ref{ntn-Holpre},
there is a canonical t-structure on 
$D^{\mathrm{b}}_{\mathrm{h}} (Y,\PP/\W)$.
Using \ref{ind-ovcoh}, this t-structure is independent on the choice of the frame
$(Y,X, \PP)$
enclosing 
$\mathbb{Y}:=(Y,X)$. 
Hence, we get a canonical t-structure on 
$D^{\mathrm{b}}_{\mathrm{h}} (\mathbb{Y}/\W)$,
whose heart, the category of overholonomic modules on $\mathbb{Y}/\W$ after any base change, is denoted by
$\mathrm{H} (\mathbb{Y}/\W)$.
Finally, we denote by 
$\mathcal{H} ^i _{\mathrm{t}}$
the $i$th space of cohomology with respect to this canonical t-structure. 
With this canonical t-structure, for any integer $n \in \Z$, 
we get the subcategories
$ D^{\leq n}_{\textrm{h}}(\mathbb{Y}/\W)$
and 
$D^{\geq n}_{\textrm{h}}(\mathbb{Y}/\W)$.

\item From \ref{IsocSigmapre}, we have 
a canonical t-structure on 
$D^{\mathrm{b}}_{\textrm{h-isoc}} (Y,\PP/\W)$
such that the inclusion 
$D^{\mathrm{b}}_{\textrm{h-isoc}} (Y,\PP/\W) \subset D^{\mathrm{b}}_{\mathrm{h}} (Y,\PP/\W)$
preserves  t-structures. 
Using Lemma \cite[5.4.1.1]{caro-pleine-fidelite}, 
this t-structure is independent (up the canonical equivalence of categories of the type of Theorem \ref{ind-ovcoh}) on the choice of the frame
$(Y,X, \PP)$
enclosing 
$\mathbb{Y}:=(Y,X)$. 
Hence, we get a canonical t-structure on 
$D^{\mathrm{b}}_{\textrm{h-isoc}} (\mathbb{Y}/\W)$,
whose heart, the category of overholonomic  after any base change isocrystals on $\mathbb{Y}/\W$, is denoted by
$\textrm{H-Isoc} ^{\dag \dag}(\mathbb{Y}/\W)$.
With the notation of \ref{ntn-Hol}.\ref{ntn-Hol3}, 
for any integer $n \in \Z$, 
we get the subcategories
$D^{\leq n}_{\textrm{h-isoc}}(\mathbb{Y}/\W) := D^{\mathrm{b}}_{\textrm{h-isoc}}(\mathbb{Y}/\W) \cap D^{\leq n}_{\textrm{h}}(\mathbb{Y}/\W)$
and 
$D^{\geq n}_{\textrm{h-isoc}}(\mathbb{Y}/\W) := D^{\mathrm{b}}_{\textrm{h-isoc}}(\mathbb{Y}/\W) \cap D^{\geq n}_{\textrm{h}}(\mathbb{Y}/\W)$.
Finally, we denote by
$\underrightarrow{LD} ^{\mathrm{b}} _{\Q,  \textrm{h-isoc}}(\mathbb{Y}/\W)$
the full subcategory of 
$\smash{\underrightarrow{LD}} ^{\mathrm{b}} _{\Q,\mathrm{coh}} ( \smash{\widehat{\D}} _{\PP} ^{(\bullet)})$
of objects
$\E ^{(\bullet)}$ such that
$\underrightarrow{\lim}~
\E ^{ (\bullet)}
\in 
D^{\mathrm{b}}_{\textrm{h-isoc}} (\mathbb{Y}/\W)$.

\end{enumerate}

\end{ntn}

\subsection{Formalism of Grothendieck six operations over realizable varieties}

\begin{dfn}
[Proper compactification]
\begin{enumerate}
\item A frame $(Y,X,\PP)$ over $\W$ is said to be {\it proper} if $\PP$ is proper. 
The category of proper frames over $\W$ is the subcategory of the category
of frames over $\W$ whose objects are proper frames over $\W$.

\item The category of {\it proper couples} over $\W$ is the full subcategory of 
the category of couples over $\W$ whose objects $(Y, X)$ are such that 
$X$ is proper. 
We remark that if $(Y, X)$ is a proper couple over $\W$  then 
there exists a proper frame over $\W$ of the form $(Y,X, \PP)$.

\item A {\it realizable variety} over $\W$ is a $l$-scheme $Y$ such that there exists a proper frame of the form 
$(Y,X,\PP)$. For such frame $(Y,X,\PP)$, we say that 
the proper frame $(Y,X,\PP)$ encloses $Y$ or that the proper couple $(Y,X)$ encloses $Y$. 
\end{enumerate}

\end{dfn}

\begin{empt}
[Formalism of Grothendieck six operations]
\label{6operations-variety}
Let $\mathfrak{C}$ be a data of coefficients over $\V$
which contains $\mathfrak{B} _\mathrm{div}$, 
which is 
stable  under devissage, 
realizable pushforwards, extraordinary pullbacks, 
duals,
and tensor products.
Similarly to Lemma \ref{ind-cat-overcouples}, we check using Theorem \ref{ind-CYW}
that the category 
$\mathfrak{C} (Y, \PP/\W)$
(resp. $\mathfrak{C} (Y,X/\W)$)
 does not depend, up to a canonical equivalence of categories, 
on the choice of the proper frame $(Y,X, \PP)$ (resp. the proper couple $(Y,X)$) over $\W$ enclosing $Y$.
As for \ref{6operations}, we can define a formalism of Grothendieck six operations on realizable varieties 
as follows.
Let  $u\colon Y'\to Y$ be a morphism of realizable varieties over $\W$.
\begin{enumerate}

\item We have the dual functor $\DD _{Y}
\colon \mathfrak{C} (Y/\W) \to \mathfrak{C} (Y/\W)$ (see \ref{ind-dual}).

\item 
We have the extraordinary pullback $u ^{!}
\colon
\mathfrak{C} (Y/\W) \to \mathfrak{C} (Y'/\W)$ (see \ref{ind-pushforward-extinv}).
We get the pullbacks $u ^{+}:= \DD _{Y'} \circ u ^{!} \circ \DD _{Y}$.

\item 
We have the functor $u _+
\colon \mathfrak{C} (Y'/\W) \to \mathfrak{C} (Y/\W)$ (see \ref{ind-pushforward-extinv}) 
We denote by  $u _{!}:= \DD _{Y} \circ u _{+} \circ \DD _{Y'}$, 
the extraordinary pushforward by $u$.

\item 
We have the tensor product 
$-\widetilde{\otimes} _{Y}-
\colon 
\mathfrak{C} (Y/\W) \times \mathfrak{C} (Y/\W) \to \mathfrak{C} (Y/\W)$
(see \ref{ind-prod-tensor})
\end{enumerate}

\end{empt}

\subsection{Constructible t-structure for overholonomic complexes after any base change}

For completeness (this will not be useful in this paper), 
we extend Tomoyuki Abe's definition of constructibility in the context of
overholonomic complexes after any base change by introducing a new way of defining it (i.e. by devissage).

\begin{empt}
[Constructible t-structure]
\label{t-structure}

Let $\mathbb{Y}:= (Y,X)$ be a couple. Choose a frame $(Y, X,\PP)$. 
If $Y' \rightarrow Y$ is an immersion, then 
we denote by 
$i _{Y'} \colon (Y', X', \PP) \to  (Y, X, \PP)$ the induced morphism
where $X'$ is the closure of $Y'$ in $X$. 
We define on $D^{\mathrm{b}}_{\mathrm{h}}(\mathbb{Y}/\W)$
the constructible t-structure as follows. 
\begin{enumerate}
\item An object $\E\in D^{\mathrm{b}}_{\mathrm{h}}(\mathbb{Y}/\W) $
belongs to 
$D^{c, \geq 0}_{\mathrm{h}}(\mathbb{Y}/\W)$ if there exists a smooth stratification 
(see Definition \cite[2.2.1]{Abe-Caro-weights})
$(Y _{i}) _{i=1, \dots , r}$ of $Y$ such that 
for any $i $, the complex $i _{Y _i} ^{+} (\E ) [ d _{Y _i}]$ (see notation \ref{nota-h-ovhol})
belongs to
$D^{\geq 0}_{\textrm{h-isoc}}(Y _i,\PP/\W)$.

\item An object $\E\in D^{\mathrm{b}}_{\mathrm{h}}(\mathbb{Y}/\W) $
belongs to 
$D^{c, \leq 0}_{\mathrm{h}}(\mathbb{Y}/\W)$ if there exists a smooth stratification 
$(Y _{i}) _{i=1, \dots , r}$ of $Y$ such that 
for any $i $, the complex $i _{Y _i} ^{+} (\E )  [ d _{Y _i}]$ belongs to
$D^{\leq 0}_{\textrm{h-isoc}}(Y _i,\PP/\W)$.

\end{enumerate}

\end{empt}

\begin{prop}
Let $\mathbb{Y}:= (Y,X)$ be a couple.
\begin{enumerate}

\item Let $\E '\to \E \to \E'' \to \E' [1]$ be an exact triangle in $D^{\mathrm{b}}_{\mathrm{h}}(\mathbb{Y}/\W) $. 
If $\E'$ and $\E''$are in $D^{c, \geq 0}_{\mathrm{h}}(\mathbb{Y}/\W)$
(resp. $D^{c, \leq 0}_{\mathrm{h}}(\mathbb{Y}/\W)$)
then so is $\E$.

\item 
Suppose that $Y$ is smooth. 
Let $\E\in D^{\mathrm{b}}_{\textrm{h-isoc}}(\mathbb{Y}/\W) $.
Then $\E \in D^{c, \geq 0}_{\mathrm{h}}(\mathbb{Y}/\W)$
(resp. $\E \in D^{c, \leq 0}_{\mathrm{h}}(\mathbb{Y}/\W)$)
if and only if 
$\E \in D^{\geq d _X}_{\textrm{h-isoc}}(\mathbb{Y}/\W)$
(resp. $\E \in D^{\leq d _X}_{\textrm{h-isoc}}(\mathbb{Y}/\W)$). 

\end{enumerate}
\end{prop}

\begin{proof}
The proof of the first part is similarly to \ref{trisub}.
The second part is easy.
\end{proof}

\begin{rem}
Let $\mathbb{Y}:= (Y,X)$ be a proper couple.
Then, the categories 
$D^{c, \geq 0}_{\mathrm{h}}(\mathbb{Y}/\W)$
and 
$D^{c, \leq 0}_{\mathrm{h}}(\mathbb{Y}/\W)$
only depend on $Y$ and can be simply denoted by 
$D^{c, \geq 0}_{\mathrm{h}}(Y/\W)$
and 
$D^{c, \leq 0}_{\mathrm{h}}(Y/\W)$.
This constructible t-structure 
is compatible with that defined by Tomoyuki Abe in 
\cite[1.3.1]{Abe-LanglandsIsoc}
(more precisely, one can check that, 
if we restrict to the categories denoted there by 
$D^{\mathrm{b}}_{\mathrm{hol}}(Y/\W)$, 
we get Tomoyuki Abe's definition of constructibility).
Indeed, let $\E \in D^{\geq 0}_{\textrm{h-isoc}}(Y/\W)$.
For any immersion $i$ of realizable varieties,  
Tomoyuki Abe's definition of $D^{c, \geq 0}_{\mathrm{hol}}$
and of $D^{c, \leq 0}_{\mathrm{hol}}$ are stable by $i ^{+} $ and under $i _{!}$. 
Since this property is obvious with the definition of \ref{t-structure}, by devissage in overconvergent isocrystals,
we reduce to the case where there exist a smooth subvariety $Z$ of $Y$ and an object 
$\G \in D^{\geq 0}_{\textrm{h-isoc}}(Z/\W)$ 
such that 
$\E=i _{Z!} (\G)$.
In that case, this is clear that both definitions of $D^{\geq 0}_{\textrm{h-isoc}}(Y/\W)$ are the same.
We proceed in the same way for $D^{\leq 0}_{\textrm{h-isoc}}(Y/\W)$.
\end{rem}

\section{Around unipotence}

\subsection{$\Sigma$-unipotent monodromy}
Let $\W$ be an object of $\mathrm{DVR}  (\V)$ and $l$ be its residue field.

\begin{empt}
\label{IsocdivSigma}
Let $(Y,X, \PP)$ be a frame over $\W$. 
We suppose that $X$ is $l$-smooth, $Z:=X - Y$ is a simple normal crossing divisor of $X$ 
and that there exists a divisor  $T$ of $P$ such that $Z=X \cap T$.
Let 
$Z = \cup _{i=1} ^{r} Z _i$ be the decomposition of $Z$ into irreducible components. 
We denote by 
$\mathrm{Isoc} ^{\dag} _\Sigma (\PP, T, X/\W)$, 
the full subcategory of Berthelot's category of overconvergent isocrystals on $(\PP, T, X/\W)$ (see Notation \ref{Isodagdagdiv}) of 
isocrystals on $(Y,X/\W)$ having  $\Sigma ^{r}$-unipotent monodromy according to Shiho's definition \cite[3.9]{Shiho-logextension} (and its remark).
We denote by 
$\textrm{Isoc} _ \Sigma ^{\dag \dag}(\PP, T, X/\W)$ the full subcategory of 
$\textrm{Isoc}  ^{\dag \dag}(\PP, T, X/\W)$
such that the equivalence of categories 
\ref{eqcat-isocpre}
induces 
the following one 
\begin{equation}
\label{eqcat-isoc}
\sp _{+}
\colon
\mathrm{Isoc} _\Sigma ^{\dag}(\PP, T, X/\W)
\cong 
\mathrm{Isoc} ^{\dag \dag} _\Sigma (\PP, T, X/\W).
\end{equation}
We denote by $\mathrm{Isoc} _\Sigma ^{(\bullet)} (\PP, T, X/\W)$
the full subcategory of 
$\mathrm{Isoc} ^{(\bullet)} (\PP, T, X/\W)
\subset \smash{\underrightarrow{LM}}  _{\Q,\mathrm{coh}} ( \smash{\widehat{\D}} _{\PP} ^{(\bullet)} (T))$
(see Notation \cite[2.2.4]{caro-stab-sys-ind-surcoh})
of objects
$\E ^{(\bullet)}$ such that
$\underrightarrow{\lim}~
\E ^{ (\bullet)}
\in 
\mathrm{Isoc} _\Sigma ^{\dag \dag}(\PP, T, X/\W)$.
Since the equivalence of categories
\ref{limeqcat} is still valid by adding overconvergent singularities along a divisor (i.e. see \cite[2.2.4.2]{caro-stab-sys-ind-surcoh}),
then 
we get the equivalence of categories
\begin{equation}
\label{lim-Isoc-Sigma}
\underrightarrow{\lim}
\colon 
\mathrm{Isoc} _\Sigma ^{(\bullet)} (\PP, T, X/\W)
\cong
\mathrm{Isoc} _\Sigma ^{\dag \dag}(\PP, T, X/\W).
\end{equation}

\end{empt}

\begin{lem}
\label{ovhol-Sigma}
We keep the notation and hypotheses of \ref{IsocdivSigma}. 
The category 
$\mathrm{Isoc} ^{\dag \dag} _\Sigma (\PP, T, X/\W)$
is an abelian subcategory of 
$\textrm{H-Isoc} ^{\dag \dag}(Y,\PP/\W)$ (see Notation \ref{IsocSigmapre}.1)
stable under extension.
\end{lem}

\begin{proof}
Let $\E \in \mathrm{Isoc} ^{\dag \dag} _\Sigma (\PP, T, X/\W)$.
Since the property that $\E$ belongs to a category of the form 
$\mathrm{Isoc} ^{\dag \dag} _\Sigma (\PP, T, X/\W)$
is stable under base change, 
then we reduce to  check that $\E$ is an overholonomic $\D ^\dag _{\PP,\Q}$-module. 
Let $E \in \mathrm{Isoc} _\Sigma ^{\dag}(\PP, T, X/\W)$ such that
 $\sp _{+} (E) \riso \E$ (see \ref{eqcat-isoc}).
 Since this local in $\PP$, 
 we can suppose $\PP$ affine, that there exists a closed immersion of smooth formal schemes over $\W$ of the form 
$\X \hookrightarrow \PP$ which is a lifting of $X \hookrightarrow P$, and that  
 there exists a strict normal crossing divisor
 $\ZZ$ of $\X$ which lifts $Z$. 
By using Berthelot-Kashiwara theorem (see \cite[5.3.6]{caro-stab-sys-ind-surcoh}), 
 we reduce to the case where $X=P$. 
 Let $\sp \colon \X _K \to \X$ be the specialisation morphism
 from the rigid analytic space associated to $\X$ (also called Raynaud generic fiber of $\X$) to $\X$.
From Theorem \cite[3.16]{Shiho-logextension} (or better Remark \cite[3.17]{Shiho-logextension}), 
there exists a convergent isocrystal $G$	
on the log scheme $(X , M _{Z})$ over $\W$, where $M_Z$ is the log-structure corresponding to the  strict normal crossing divisor $Z$ of $X$, 
with exponents in $\tau (\Sigma)$
such that $j ^\dag (G) \riso E$, where $j\colon Y \to X$ is the open immersion.
From \cite[2.3.13]{caro-Tsuzuki}, since by hypothesis the elements of the group $\Sigma$
are $p$-adically non Liouville numbers, then 
$u _+ \sp _* (G)$ is overholonomic, where 
$u \colon (\X , M _{\ZZ}) \to \X$ is the canonical morphism.  
Since $\E \riso (\hdag Z) u _+ \sp _* (G)$, and since overholonomicity is stable under
$(\hdag Z)$ then $\E$ is also overholonomic.

The stability under extension is clear by definition (see Definition \cite[1.3]{Shiho-logextension}) and the fact that 
$\mathrm{Isoc} ^{\dag \dag} _\Sigma (\PP, T, X/\W)$
is an abelian subcategory of 
$\textrm{H-Isoc} ^{\dag \dag}(Y,\PP/\W)$ follows from  \cite[1.17]{Shiho-logextension}.
\end{proof}

\begin{dfn}
\label{dfnSunip}
Let $(Y,X, \PP)$ be a frame over $\W$. 
We suppose that $X$ is $l$-smooth, $Z:=X - Y$ is a simple normal crossing divisor of $X$. 
We put $\mathbb{Y}:= (Y,X)$.
\begin{enumerate}

\item Let $\E \in \textrm{H-Isoc} ^{\dag \dag}(Y,\PP/\W)$ (see Notation \ref{IsocSigmapre}.\ref{IsocSigma1pre}).
We will say that $\E$ has ``$\Sigma$-unipotent monodromy'' if 
for any open set $\PP'$ of $\PP$ such that there exists a divisor $T'$ of $P'$ satisfying
$Z \cap P' = X \cap T'$, we have 
$\E |\PP' \in \mathrm{Isoc} ^{\dag \dag} _\Sigma (\PP', T', X \cap P'/\W)$ (see Notation \ref{IsocdivSigma}).

\item We denote by 
$\mathrm{Isoc} ^{\dag \dag} _{\Sigma} (Y,\PP/\W)$
the full subcategory of 
$ \textrm{H-Isoc} ^{\dag \dag}(Y,\PP/\W)$
whose objects ``have $\Sigma$-unipotent monodromy''. 
We remark that Lemma \ref{ovhol-Sigma} justifies the fact that
we remove ``H'' in the notation. 
Since 
the category $\mathrm{Isoc} ^{\dag \dag} _{\Sigma} (Y,\PP/\W)$ is independent (up to canonical equivalences of categories 
appearing in \ref{ntn-Hol} to define 
$\textrm{H-Isoc} ^{\dag \dag} (\mathbb{Y}/\W)$) 
on the choice of the frame 
$(Y,X, \PP)$ enclosing $\mathbb{Y}$, 
we will denote it by $\mathrm{Isoc} ^{\dag \dag} _{\Sigma}(\mathbb{Y}/\W)$. 

\item We denote by 
$\mathrm{Isoc} ^{(\bullet)} _{\Sigma}(Y,\PP/\W)$
or simply by
$\mathrm{Isoc} ^{(\bullet)} _{\Sigma}(\mathbb{Y}/\W)$
the full subcategory of 
$\smash{\underrightarrow{LM}}  _{\Q,\mathrm{coh}} ( \smash{\widehat{\D}} _{\PP} ^{(\bullet)})$
(see Notation \cite[2.2.4]{caro-stab-sys-ind-surcoh})
of objects
$\E ^{(\bullet)}$ such that
$\underrightarrow{\lim}~
\E ^{ (\bullet)}
\in 
\mathrm{Isoc} ^{\dag \dag} _{\Sigma}(\mathbb{Y}/\W)$.
The functor 
$\underrightarrow{\lim}$
of \ref{limeqcat} induces the equivalence of categories
$\underrightarrow{\lim}
\colon 
\mathrm{Isoc} ^{(\bullet)} _{\Sigma}(\mathbb{Y}/\W)
\cong 
\mathrm{Isoc} ^{\dag \dag} _{\Sigma}(\mathbb{Y}/\W)$.
\end{enumerate}

\end{dfn}

\begin{prop}
\label{unip-hol}
 We keep the notation and hypotheses of 
\ref{dfnSunip}.
\begin{enumerate}

\item The property that 
an object
$\E^{(\bullet)}$
of 
$\smash{\underrightarrow{LM}}  _{\Q} ( \smash{\widehat{\D}} _{\PP} ^{(\bullet)})$
is in $ \mathrm{Isoc} ^{(\bullet)} _{\Sigma}(\mathbb{Y}/\W)$ is local in $\PP$.

\item The category 
$\mathrm{Isoc} ^{\dag \dag} _{\Sigma}(\mathbb{Y}/\W)$
(resp. $\mathrm{Isoc} ^{(\bullet)} _{\Sigma}(\mathbb{Y}/\W)$)
is an abelian subcategory of 
$\textrm{H-Isoc} ^{\dag \dag}(\mathbb{Y}/\W)$
(resp. $\smash{\underrightarrow{LM}}  _{\Q,\mathrm{coh}} ( \smash{\widehat{\D}} _{\PP} ^{(\bullet)})$)
stable under extension.

\item The category $\mathrm{Isoc} ^{(\bullet)} _{\Sigma}(\mathbb{Y}/\W)$ is stable under base change in the following sense: 
for any morphism $\W\to \W '$ of $\mathrm{DVR}  (\V)$,
for any $\E^{(\bullet)}  \in \mathrm{Isoc} ^{(\bullet)} _{\Sigma}(\mathbb{Y}/\W)$,
putting $(Y', X' , \PP')$ the frame over $\W'$ induced by base change from $(Y,X,\PP)$ by $\W\to \W '$,
we get 
$ \W '  \smash{\overset{\L}{\otimes}}   ^{\dag}
_{\W}  \E^{(\bullet) }
 \in  
 \mathrm{Isoc} ^{(\bullet)} _{\Sigma} (Y',X'/\W')$. 

\item Let 
$\E ^{(\bullet)} ,\, \FF ^{(\bullet)} \in \mathrm{Isoc} ^{(\bullet)} _{\Sigma} (\mathbb{Y}/\W)$.
We have
$\E ^{ (\bullet)}\smash{\overset{\L}{\otimes}}   ^{\dag} _{\O _{\PP}} \FF ^{ (\bullet)} [d _{Y/P}]
\in 
\mathrm{Isoc} ^{(\bullet)} _{\Sigma}(\mathbb{Y}/\W)$ 
(which means in particular
that the complex is in fact isomorphic to a module).

\end{enumerate}
\end{prop}

\begin{proof}
Let $\E^{(\bullet)}$
of 
$\smash{\underrightarrow{LM}}  _{\Q} ( \smash{\widehat{\D}} _{\PP} ^{(\bullet)})$.
The fact that 
$\E^{(\bullet)} \in \smash{\underrightarrow{LM}}  _{\Q,\mathrm{coh}} ( \smash{\widehat{\D}} _{\PP} ^{(\bullet)})$
is local in $\PP$ (recall Definition \cite[2.2.1]{caro-stab-sys-ind-surcoh}).
Hence, we get the first assertion. 
The second one is a consequence of \ref{ovhol-Sigma}.
The assertion 3) is straightforward. 
Let us  check 4).
Let 
$\E ^{(\bullet)} ,\, \FF ^{(\bullet)} \in \mathrm{Isoc} ^{(\bullet)} _{\Sigma} (\mathbb{Y}/\W)$.
By using the assertion 1), 
we can suppose that there exists a divisor $T$ of $P$ such that $Y= X \setminus T$.
Then, 
$ \mathrm{Isoc} ^{(\bullet)} _{\Sigma}(\mathbb{Y}/\W)
=
\mathrm{Isoc} ^{(\bullet)} _{\Sigma}(\PP, T,X/\W)$ 
is a full subcategory of 
$ \mathrm{Isoc} ^{(\bullet)} (\PP, T,X/\W)$.
From Lemma \cite[3.2.2.1]{caro-stab-prod-tens} (in fact, \cite[3.1.5.1]{caro-stab-prod-tens} is sufficient), we get 
$\E ^{ (\bullet)}\smash{\overset{\L}{\otimes}}   ^{\dag} _{\O _{\PP} (\hdag T) _\Q} \FF ^{ (\bullet)} [d _{Y/P}]
\in 
 \mathrm{Isoc} ^{(\bullet)} (\PP, T,X/\W)$.
 Since 
 $\E^{(\bullet)} ,
 \,
 \FF^{(\bullet)}  \in \smash{\underrightarrow{LM}}  _{\Q,\mathrm{coh}} ( \smash{\widehat{\D}} _{\PP} ^{(\bullet)}) 
 \cap  \mathrm{Isoc} ^{(\bullet)} (\PP, T,X/\W)$,
 then 
 $\E^{(\bullet)}  (\hdag T) \riso \E^{(\bullet)} $
 and 
  $\FF^{(\bullet)}  (\hdag T) \riso \FF^{(\bullet)} $.
  Hence, from \cite[2.1.5]{caro-stab-prod-tens}, we get 
 $\E ^{ (\bullet)}\smash{\overset{\L}{\otimes}}   ^{\dag} _{\O _{\PP} } \FF ^{ (\bullet)}
 \riso 
 \E ^{ (\bullet)}\smash{\overset{\L}{\otimes}}   ^{\dag} _{\O _{\PP} (\hdag T) _\Q} \FF ^{ (\bullet)} $.

Let 
$\E := \underrightarrow{\lim}~
\E ^{ (\bullet)}$, 
$\FF := \underrightarrow{\lim}~
\FF ^{ (\bullet)}$.
From \ref{eqcat-isoc}, there exist  $E,F \in 
\mathrm{Isoc} _\Sigma ^{\dag}(\PP, T, X/\W) $
such that
 $\sp _{+} (E) \riso \E$ 
 and 
 $\sp _{+} (F) \riso \FF$.
 From Theorem \cite[3.2.6.2]{caro-stab-prod-tens} 
(in fact, Proposition \cite[3.1.8]{caro-stab-prod-tens} is sufficient), 
$\sp _+ ( E \otimes F) \riso \sp _{+} (\E) \smash{\overset{\L}{\otimes}}   ^{\dag} _{\O _{\PP} (\hdag T) _\Q} \sp _+ (F) [d _{Y/P}]$
where $\otimes $ is the usual tensor product 
of $\mathrm{Isoc} ^{\dag}(\PP, T, X/\W)$.
Hence, we reduce to check that
$E \otimes F$ has $\Sigma ^{r}$-unipotent monodromy according to Shiho's definition \cite[3.9]{Shiho-logextension} (and its remark).
Using \cite[3.16]{Shiho-logextension} (or better \cite[3.17]{Shiho-logextension}), 
we get that $E $ (resp. $F$) comes from a log convergent isocrystal 
$G _1$ (resp. $G _2$) with exponents in $\tau (\Sigma)$.  
If $\mathrm{Exp} (G _1)$ and $\mathrm{Exp}(G _2)$
are the exponents of respectively $G _1$ and $G _2$ then 
the exponents of $G _1 \otimes G _2$ 
are $\mathrm{Exp} (G _1)+\mathrm{Exp}(G _2)$.
Hence, since $\Sigma$ is a group, 
since $E \otimes F$ comes from $G _1 \otimes G _2$,
then using 
\cite[3.16]{Shiho-logextension}, 
we get that 
$E \otimes F$ 
 has $\Sigma ^{r}$-unipotent monodromy.

\end{proof}

\begin{dfn}
\label{dfnSunipD}
We keep the notation and hypotheses of 
\ref{dfnSunip}.
\begin{enumerate}
\item Let $\E \in D^{\mathrm{b}}_{\mathrm{h}}(\mathbb{Y}/\W)$ (see Notation \ref{ntn-Hol}.1).
We say that 
$\E$ ``has $\Sigma$-unipotent monodromy'' if, 
for any integer $i$, the module 
$\mathcal{H} ^i 
_{\mathrm{t}}
( \E )\in \mathrm{Isoc} ^{\dag \dag} _{\Sigma}(\mathbb{Y}/\W)$.
We will denote by 
$D^{\mathrm{b}}_{\mathrm{isoc}, \Sigma} (\mathbb{Y}/\W)$
the full subcategory of 
$D^{\mathrm{b}}_{\mathrm{h}}(\mathbb{Y}/\W)$
whose objects have
$\Sigma$-unipotent monodromy.

\item We denote by 
$\underrightarrow{LD} ^{\mathrm{b}} _{\Q,  \mathrm{isoc}, \Sigma}(\mathbb{Y}/\W)$
the full subcategory of 
$\smash{\underrightarrow{LD}} ^{\mathrm{b}} _{\Q,\mathrm{coh}} ( \smash{\widehat{\D}} _{\PP} ^{(\bullet)})$
of the objects
$\E ^{(\bullet)}$ such that
$\underrightarrow{\lim}~
\E ^{ (\bullet)}
\in 
D^{\mathrm{b}}_{\mathrm{isoc}, \Sigma} (\mathbb{Y}/\W)$.

\end{enumerate}

\end{dfn}

\begin{rem}
\label{rem-local}
We keep the notation and hypotheses of 
\ref{dfnSunip}.
Let $\E \in D^{\mathrm{b}}_{\mathrm{h}}(\mathbb{Y}/\W)$.
The fact that 
$\E \in D^{\mathrm{b}}_{\mathrm{isoc}, \Sigma} (\mathbb{Y}/\W)$
is local in $\PP$. 
\end{rem}

\begin{prop}
\label{QUserreHol}
We keep the notation and hypotheses of 
\ref{dfnSunipD}.

\begin{enumerate}

\item The category $\underrightarrow{LD} ^{\mathrm{b}} _{\Q,  \mathrm{isoc}, \Sigma}(\mathbb{Y}/\W)$
(resp. $D^{\mathrm{b}}_{\mathrm{isoc}, \Sigma} (\mathbb{Y}/\W)$)
is a triangle subcategory of
$\smash{\underrightarrow{LD}} ^{\mathrm{b}} _{\Q,\mathrm{h}} ( \smash{\widehat{\D}} _{\PP} ^{(\bullet)})$
(resp. $D^{\mathrm{b}}_{\mathrm{h}} (Y,\PP/\W)$).

\item A direct factor in $\smash{\underrightarrow{LD}} ^{\mathrm{b}} _{\Q,\mathrm{coh}} ( \smash{\widehat{\D}} _{\PP} ^{(\bullet)})$
of an object of $\underrightarrow{LD} ^{\mathrm{b}} _{\Q,  \mathrm{isoc}, \Sigma}(\mathbb{Y}/\W)$ is an object of
$\underrightarrow{LD} ^{\mathrm{b}} _{\Q,  \mathrm{isoc}, \Sigma}(\mathbb{Y}/\W)$.

\item \label{stab-tensorSigma1}
The category $\underrightarrow{LD} ^{\mathrm{b}} _{\Q,  \mathrm{isoc}, \Sigma}(\mathbb{Y}/\W)$ is stable under base change, i.e. 
for any morphism $\W\to \W '$ of $\mathrm{DVR}  (\V)$,
for any $\E^{(\bullet)}  \in  \underrightarrow{LD} ^{\mathrm{b}} _{\Q,  \mathrm{isoc}, \Sigma}(\mathbb{Y}/\W)$,
putting $(Y', X' , \PP')$ the frame over $\W'$ induced by base change from $(Y,X,\PP)$ by $\W \to \W'$,
we get 
$ \W '  \smash{\overset{\L}{\otimes}}   ^{\dag}
_{\W}  \E^{(\bullet) }
 \in  \underrightarrow{LD} ^{\mathrm{b}} _{\Q,  \mathrm{isoc}, \Sigma}(Y',X'/\W')$.

\item Let 
$\E ^{(\bullet)} ,\, \FF ^{(\bullet)} \in \underrightarrow{LD} ^{\mathrm{b}} _{\Q,  \mathrm{isoc}, \Sigma}(\mathbb{Y}/\W)$.
We have
$\E ^{ (\bullet)}\smash{\overset{\L}{\otimes}}   ^{\dag} _{\O _{\PP}} \FF ^{ (\bullet)} 
\in \underrightarrow{LD} ^{\mathrm{b}} _{\Q,  \mathrm{isoc}, \Sigma}(\mathbb{Y}/\W).$

\end{enumerate}
\end{prop}

\begin{proof}
The second assertion is straightforward. 
The other ones are a consequence of \ref{unip-hol}. 
\end{proof}

\begin{prop}
\label{regu-pullback}
Let $\theta = (b,a,f) \colon (Y',X', \PP')\to (Y,X, \PP)$ be a morphism of frames over $\W$.
We suppose that $X$ and $X'$ are $l$-smooth, and $Z:=X - Y$ (resp. $Z':=X' - Y'$) is a simple normal crossing divisor of $X$ (resp. $X'$).
We put $\mathbb{Y}:= (Y,X)$ and 
$\mathbb{Y}':= (Y',X')$.

\begin{enumerate}

\item We have the exact functor
\begin{equation}
\label{f+isocSigmapre}
\R \underline{\Gamma} ^{\dag} _{Y'} f ^{!} [d _Y -d_{Y'}]
\colon
\mathrm{Isoc} ^{(\bullet)} _{\Sigma}(\mathbb{Y}/\W)
\to
\mathrm{Isoc} ^{(\bullet)} _{\Sigma}(\mathbb{Y}'/\W).
\end{equation}

\item We have the t-exact functor
\begin{equation}
\label{f+isocSigma}
\R \underline{\Gamma} ^{\dag} _{Y'} f ^{!} [d _Y -d_{Y'}]
\colon
\underrightarrow{LD} ^{\mathrm{b}} _{\Q,  \mathrm{isoc}, \Sigma} (\mathbb{Y}/\W)
\to
\underrightarrow{LD} ^{\mathrm{b}} _{\Q,  \mathrm{isoc}, \Sigma} (\mathbb{Y}'/\W),
\end{equation}
and a similar one by replacing 
``$\underrightarrow{LD} _{\Q} $''
by 
``$D$''.
\end{enumerate}
\end{prop}

\begin{proof}
Let us check \ref{f+isocSigmapre}.
From \ref{unip-hol}.1, we can suppose that there exist a divisor $T$ of $P$ such that $Y= X \setminus T$
and a divisor $T'$ of $P'$ such that $Y'= X '\setminus T'$.
Is this case,
$ \mathrm{Isoc} ^{(\bullet)} _{\Sigma}(\mathbb{Y}/\W)
=
\mathrm{Isoc} ^{(\bullet)} _{\Sigma} (\PP, T,X/\W)$ 
is a full subcategory
of 
$ \mathrm{Isoc} ^{(\bullet)} (\PP, T,X/\W)$.
Let $\E ^{(\bullet)}\in \mathrm{Isoc} ^{(\bullet)} _{\Sigma} (\PP, T,X/\W)$.
 From \cite[1.4.5.3]{caro-stab-prod-tens},
we have
$ \R \underline{\Gamma} ^{\dag} _{Y'} f ^{!} [d _Y -d_{Y'}]
(\E ^{(\bullet)})\in 
\mathrm{Isoc} ^{(\bullet)} (\PP', T',X'/\W)$.
Let 
$\E := \underrightarrow{\lim}~
\E ^{ (\bullet)}$.
From \ref{eqcat-isoc},
there exists  
$E\in \mathrm{Isoc} ^{\dag} _\Sigma (\PP, T,X/\W)$ such that
 $\sp _{+} (E) \riso \E$.
Using  \cite[3.17]{Shiho-logextension}, 
the overconvergent isocrystal $E $ comes from a log convergent isocrystal 
$G $ with exponents in $\tau (\Sigma)$.  
Using the Remark \cite[1.1.3.1]{caro-Tsuzuki}, 
we get that 
$a _\sharp ^* (G)$ 
is a log convergent isocrystal with exponents in $\tau (\Sigma)$,
where $a _\sharp \colon (X ' , M _{Z'}) \to (X  , M _{Z})$
is the morphism of log-schemes induced by $a$. 
Hence, $a  ^* (E)$ has $\Sigma ^{r}$-unipotent monodromy. 
Using  \cite[1.4.5.4]{caro-stab-prod-tens}, we get that 
$\underrightarrow{\lim} (\R \underline{\Gamma} ^{\dag} _{Y'} f ^{!} [d _Y -d_{Y'}]
(\E ^{(\bullet)})) 
\riso 
\R \underline{\Gamma} ^{\dag} _{Y'} f ^{!} [d _Y -d_{Y'}]
(\E )
\in 
\mathrm{Isoc} _\Sigma ^{\dag \dag}  (\PP', T',X'/\W)$.
Hence, 
$(\R \underline{\Gamma} ^{\dag} _{Y'} f ^{!} [d _Y -d_{Y'}]
(\E ^{(\bullet)})) 
\in 
\mathrm{Isoc} _\Sigma ^{(\bullet)}  (\PP', T',X'/\W)$
has also $\Sigma ^{r}$-unipotent monodromy.
\end{proof}

\begin{dfn}
Let $\PP$ be a  smooth formal scheme over $\W$.
We denote by 
$\underrightarrow{LD} ^{\mathrm{b}} _{\Q, \Sigma} (\smash{\widehat{\D}} _{\PP} ^{(\bullet)})$
the smallest subcategory of 
$\underrightarrow{LD} ^{\mathrm{b}} _{\Q, \mathrm{h}} (\smash{\widehat{\D}} _{\PP} ^{(\bullet)})$
stable by devissages and containing the categories of the form
$ \underrightarrow{LD} ^{\mathrm{b}} _{\Q,  \mathrm{isoc}, \Sigma}(Y, \overline{Y} /\W)$
where $\overline{Y}$ is a closed $l$-smooth subvariety of $P$,
$Y$ is an open subscheme of $\overline{Y}$ such that 
$\overline{Y} \setminus Y$ is a strict normal crossing divisor in $Y$ (thanks to the Proposition \ref{QUserreHol}.1,
this is concretely defined as in \cite[3.2.21]{caro-2006-surcoh-surcv}). 
We call the objects of 
$\underrightarrow{LD} ^{\mathrm{b}} _{\Q, \Sigma} (\smash{\widehat{\D}} _{\PP} ^{(\bullet)})$
as those of 
$\underrightarrow{LD} ^{\mathrm{b}} _{\Q, \mathrm{h}} (\smash{\widehat{\D}} _{\PP} ^{(\bullet)})$
``having $\Sigma$-unipotent monodromy''.
Finally, 
we denote by
$D ^{\mathrm{b}}  _{\Sigma}  (\smash{\D} ^\dag _{\PP\Q})$
the essential image of the functor
$\underrightarrow{LD} ^{\mathrm{b}} _{\Q,\Sigma} (\smash{\widehat{\D}} _{\PP} ^{(\bullet)})
\to 
D ^{\mathrm{b}}  _{\mathrm{h}} (\smash{\D} ^\dag _{\PP\Q})$
induced by \ref{limeqcat}.

\end{dfn}

\begin{thm}
 \label{dual}
Let $\PP$ be a  smooth formal scheme over $\W$.
The dual functor $\DD_\PP$ induces an autoequivalence of 
$\underrightarrow{LD} ^{\mathrm{b}} _{\Q, \Sigma} (\smash{\widehat{\D}} _{\PP} ^{(\bullet)})$
(resp. of $D ^{\mathrm{b}} _{\Sigma}  ( \smash{\D} ^\dag _{\PP\Q})$).
\end{thm}

\begin{proof}
Let $\FF \in D ^{\mathrm{b}} _{\Sigma}  ( \smash{\D} ^\dag _{\PP\Q})$.
By devissage,  
we can suppose that there exists
a frame $(Y,X, \PP)$ 
where 
$X$ is $l$-smooth, $Z := X \setminus Y$ is a strict normal crossing divisor of 
$X$, 
and there exists
$\E \in D^{\mathrm{b}}_{\mathrm{isoc},\Sigma}(Y,X/\W)$ 
such that 
$j _{+}\E \riso \FF$,
where 
$j \colon Y \hookrightarrow P$ 
is the immersion (here $j _+$ means simply the inclusion of
$D^{\mathrm{b}}_{\mathrm{h}} (Y,\PP/\W)$
in $D^{\mathrm{b}}_{\mathrm{h}} (\smash{\D} ^\dag _{\PP\Q})$ but we keep it in the notation to be precise).
By devissage,
we can suppose that 
$\E \in \mathrm{Isoc} ^{\dag \dag} _{\Sigma}(Y,X/\W)$. 
From the remark \cite[3.17]{Shiho-logextension}, 
there exists a convergent log isocrystals $\G$ on the log scheme
$(X, M _Z)$, where $M _Z$ is the log-structure induced by $Z$,
with exponents in $\tau (\Sigma)$
such that $(\hdag Z) (\G) \riso \E$. 

First, suppose there exists a morphism of smooth formal $\W$-schemes
$\X \hookrightarrow \PP$, and a strict normal crossing divisor $\ZZ$ of $\X$ 
which lifts $Z$.
Using Berthelot-Kashiwara Theorem, we reduce to the case where $\X=\PP$. 
Since $\tau (0)=0$, 
from \cite[2.2.9]{caro-Tsuzuki} (or \cite[3.5.6.2]{caro-stab-u!R-Gamma}),
we have $j _{+} (\E)= \alpha _{+} (\G)$, where $\alpha \colon 
(\X, M _\ZZ) \to \X$ is the canonical morphism of log formal schemes. 
Then, and with \cite[5.24.(ii)]{caro_log-iso-hol} for the last isomorphism, we get 
$\DD _{\PP} \circ j _{+} (\E) \riso 
 \alpha _{!} (\G ^{\vee})
 \riso 
  \alpha _{+} (\G ^{\vee} (-\ZZ))$. 
From \cite[3.5.6]{caro-stab-u!R-Gamma},
we get that
$  \alpha _{+} (\G ^{\vee} (-\ZZ)) \in D ^{\mathrm{b}} _{\Sigma}  ( \smash{\D} ^\dag _{\PP\Q})$ 
(with the remark that the exponents stay in $\Sigma$). 
Hence, 
$\DD _{\PP} (\FF )\in D ^{\mathrm{b}} _{\Sigma}  ( \smash{\D} ^\dag _{\PP\Q})$.

In general, we remark that the devissage in overconvergent isocrystals 
having $\Sigma$-unipotent monodromy in the local situation (i.e. the paragraph above) from 
\cite[3.5.6]{caro-stab-u!R-Gamma} is given by a smooth stratification 
which is constructed from $X$, $D$. Hence, using Remark \ref{rem-local}, 
we get by localness that the restriction of $\DD _{\PP} (\FF )$ to any strata of this smooth stratification have
$\Sigma$-unipotent monodromy.

\end{proof}

\subsection{Potentially $\Sigma$-unipotent monodromy}
Let $\W$ be an object of $\mathrm{DVR}  (\V)$ and $l$ be its residue field. 

\begin{lem}
\label{remIsocSigma1}
Let $\PP$ be a smooth formal scheme over $\W$, 
$X$ be a closed subscheme of $P$ and $T$ be a divisor of $P$ such that $Y:=X \setminus T$ is $l$-smooth (over the residue field of $\W$).
Let $\E \in \textrm{Isoc} ^{\dag \dag}(\PP, T, X/\W)$ (see Notation \ref{Isodagdagdiv}).
If there exists a complete morphism of frames over $\W$ of the form 
$\theta = (b,a,f) \colon (Y',X', \PP')\to (Y ,X, \PP)$
such that 
$a \colon X' \to X$
is a projective surjective generically finite and etale morphism, 
$X'$ is $l$-smooth, $Z':=X' \setminus Y'$ is a simple normal crossing divisor of $X'$ 
and $\theta ^{!}(\E) :=
\R \underline{\Gamma} ^{\dag} _{Y'} f ^{!}(\E)
\in \mathrm{Isoc} ^{\dag \dag} _{\Sigma} (Y',\PP'/\W)$ (see Notation \ref{dfnSunip}.2), 
then $\E \in 
\textrm{H-Isoc} ^{\dag \dag}(Y,\PP/\W))$ (see Notation \ref{IsocSigmapre}.1).
\end{lem}

\begin{proof}
From Lemma \ref{ovhol-Sigma}, 
we get 
$\R \underline{\Gamma} ^{\dag} _{Y'} f ^{!}(\E)
\in 
\textrm{H-Isoc} ^{\dag \dag}(Y',\PP'/\W))$.
In particular we have $\theta ^{!}(\E) \in \textrm{H}(Y',\PP'/\W))$.
Since the overholonomicity after any base change is stable by realizable pushforwards, 
we get 
$f _{+}\R \underline{\Gamma} ^{\dag} _{Y'} f ^{!}(\E) 
\in \textrm{H} (Y,\PP/\W))$.
Moreover, since $\E$ is a direct factor of $f _{+}\R \underline{\Gamma} ^{\dag} _{Y'} f ^{!}(\E)$, 
then this yields that $\E \in  \textrm{H} (Y,\PP/\W)$.
\end{proof}

The above lemma \ref{remIsocSigma1} justifies why we restrict to overholonomic after any base change isocrystals 
in the definition \ref{IsocpotSigma} below :
\begin{dfn}
\label{IsocpotSigma}
Let $\mathbb{Y}:= (Y,X)$ be a couple over $\W$ such that $Y$ is smooth over $l$. 
Choose a frame over $\W$ of the form $(Y,X,\PP)$.
\begin{enumerate}

\item 
\label{IsocpotSigma1}
Let 
$\E \in \textrm{H-Isoc} ^{\dag\dag} (Y, \PP/\W)$ (see Notation \ref{IsocSigmapre}.1).
We say that $\E$ is an isocrystal on $(Y, \PP/\W)$ (or simply on $\mathbb{Y}$) having ``potentially $\Sigma$-unipotent monodromy'' 
if, for any irreducible component $Y _{1}$ of $Y$, denoting by $X _1$ the closure of $Y _1$ in $X$, there exists
a morphism of frames over $\W$ of the form 
$\theta = (b,a,f) \colon (Y',X', \PP')\to (Y _1,X _1, \PP)$
such that 
$a \colon X' \to X _1$ is  a projective surjective generically finite and etale morphism, 
$Y'= a ^{-1} (Y)$,
$X'$ is $l$-smooth, 
$X' \setminus Y'$ is a simple normal crossing divisor of $X'$ 
and such that $\theta ^{!}(\E) 
:=
\R \underline{\Gamma} ^{\dag} _{Y'} f ^{!}(\E)
\in \mathrm{Isoc} ^{\dag \dag} _{\Sigma} (Y',\PP'/\W)$
(see Notation \ref{dfnSunip}.2). 

We denote by 
$\mathrm{Isoc} ^{\dag \dag} _{\text{pot-}\Sigma}(Y, \PP/\W)$ 
the full subcategory of 
$\mathrm{H} (Y,\PP/\W)$
whose objects are isocrystals having potentially $\Sigma$-unipotent monodromy. 
Using \ref{dfnSunip}.2, we check that the category
$\mathrm{Isoc} ^{\dag \dag} _{\text{pot-}\Sigma}(Y,\PP/\W)$
does not depend on the choice of the frame $(Y,X, \PP)$ enclosing $\mathbb{Y}$.
Hence, we will also 
write $\mathrm{Isoc} ^{\dag \dag} _{\text{pot-}\Sigma}(\mathbb{Y}/\W)$ instead of 
$\mathrm{Isoc} ^{\dag \dag} _{\text{pot-}\Sigma}(Y, \PP/\W)$.

\item 
Let
$D^{\mathrm{b}}_{\mathrm{isoc}, \text{pot-}\Sigma} (Y,\PP/\W)$
be the full subcategory of 
$D^{\mathrm{b}}_{\mathrm{h}} (Y,\PP/\W)$
of the objects $\E$ such that,
for any integer $i$, we have 
$\mathcal{H} ^i _{\mathrm{t}}( \E ) \in \mathrm{Isoc} ^{\dag \dag} _{\text{pot-}\Sigma}(Y,\PP/\W)$.
Since the category
$D^{\mathrm{b}}_{\mathrm{isoc}, \text{pot-}\Sigma} (Y,\PP/\W)$
does not depend on the choice of the frame $(Y,X, \PP)$ enclosing $\mathbb{Y}$,
we will also write  
$D^{\mathrm{b}}_{\mathrm{isoc}, \text{pot-}\Sigma} (\mathbb{Y}/\W)$ instead of 
$D^{\mathrm{b}}_{\mathrm{isoc}, \text{pot-}\Sigma} (Y,\PP/\W)$.

\item We denote by 
$\underrightarrow{LD} ^{\mathrm{b}} _{\Q,  \mathrm{isoc}, \text{pot-}\Sigma}(Y,\PP/\W)$
the full subcategory of 
$\underrightarrow{LD} ^{\mathrm{b}} _{\Q,  \textrm{h-isoc}}(Y,\PP/\W)$
of objects
$\E ^{(\bullet)}$ such that
$\underrightarrow{\lim}~
\E ^{ (\bullet)}
\in 
D^{\mathrm{b}}_{\mathrm{isoc}, \text{pot-}\Sigma} (Y,\PP/\W)$.
Since this does not depend on the choice of the frame $(Y,X, \PP)$ enclosing $\mathbb{Y}$,
we can simply write 
$\underrightarrow{LD} ^{\mathrm{b}} _{\Q,  \mathrm{isoc}, \text{pot-}\Sigma}(\mathbb{Y}/\W)$.
We denote by 
$\underrightarrow{LM}  _{\Q,  \mathrm{isoc}, \text{pot-}\Sigma}(Y,\PP/\W)$
the full subcategory of $\underrightarrow{LD} ^{\mathrm{b}} _{\Q,  \mathrm{isoc}, \text{pot-}\Sigma}(Y,\PP/\W)$
of complexes 
$\E ^{(\bullet)}$ such that
$\mathcal{H} ^j _t (\E ^{(\bullet)})=0$ for any $j \not =0$.

\end{enumerate}

\end{dfn}

\begin{rem}
\label{remIsocSigma2}
The following remark should justify our notation above. 
Let $\mathbb{Y}= (Y,X)$ be a couple over $\W$ where $Y$ is $l$-smooth.
Suppose in this remark that the absolute Frobenius homomorphism
$l \riso l$ sending $x $ to $x ^{p}$ lifts to an automorphism of the form $\W \riso \W$. 
In that case, in \cite[1.2.13]{Abe-Caro-weights}, 
we have defined the category 
$F\text{-}\mathrm{Isoc} ^{\dag \dag}(\mathbb{Y}/\W)$,
whose objects belong to 
$\textrm{H-Isoc} ^{\dag\dag} (\mathbb{Y}/\W)$.
We can translate Kedlaya's semistable reduction theorem of 
\cite{kedlaya-semistableIV} as follows:  
if $\E $ is an object of $F\text{-}\mathrm{Isoc} ^{\dag \dag}(\mathbb{Y}/\W)$ then
$\E\in \mathrm{Isoc} ^{\dag \dag} _{\text{pot-}\Sigma}(\mathbb{Y}/\W)$ (for 
$\Sigma = {0}$ and then for any $\Sigma$ satisfying the convention of the paper).

\end{rem}

\begin{prop}
\label{prepot-regu-pullback}
Let $(Y,X,\PP)$ be a frame over $\W$ with $Y$ smooth over $l$. 
\begin{enumerate}

\item Let $\theta = (b,a,f) \colon (Y',X',\PP')\to (Y,X,\PP)$ be a morphism of frames over $\W$.
We suppose that $Y$ and $Y'$ are $l$-smooth.
We have the t-exact functor
\begin{equation}
\label{f!isocpotSigma}
\R \underline{\Gamma} ^{\dag} _{Y'} f ^{!} [d _Y -d_{Y'}]
\colon
 \underrightarrow{LD} ^{\mathrm{b}} _{\Q,  \mathrm{isoc}, \mathrm{pot}\textrm{-}\Sigma}(Y ,\PP/\W) 
\to
 \underrightarrow{LD} ^{\mathrm{b}} _{\Q,  \mathrm{isoc}, \mathrm{pot}\textrm{-}\Sigma}(Y ',\PP '/\W).
\end{equation}

\item  
The category $ \underrightarrow{LD} ^{\mathrm{b}} _{\Q,  \mathrm{isoc}, \mathrm{pot}\textrm{-}\Sigma}  (Y,\PP/\W)$
is a triangle subcategory of
$\smash{\underrightarrow{LD}} ^{\mathrm{b}} _{\Q,\mathrm{h}} ( Y,\PP/\W)$, 
stable under direct factors and base change.

\item 
For any
$\E ^{(\bullet)} ,\, \FF ^{(\bullet)} \in \underrightarrow{LD} ^{\mathrm{b}} _{\Q,  \mathrm{isoc}, \mathrm{pot}\textrm{-}\Sigma}(Y,\PP/\W)$, 
we have 
$\E ^{ (\bullet)}\smash{\overset{\L}{\otimes}}   ^{\dag} _{\O _{\PP}} \FF ^{ (\bullet)} 
\in 
\underrightarrow{LD} ^{\mathrm{b}} _{\Q,  \mathrm{isoc}, \mathrm{pot}\textrm{-}\Sigma}(Y,\PP/\W)$.

\end{enumerate}
We have similar properties
by replacing 
``$\underrightarrow{LD} _{\Q} $''
by ``$D$''.
\end{prop}

\begin{proof}
Let 
$\E ^{(\bullet)}
\in \underrightarrow{LM}  _{\Q,  \mathrm{isoc}, \text{pot-}\Sigma}(Y,\PP/\W)$.
We have to check 
$\R \underline{\Gamma} ^{\dag} _{Y'} f ^{!} [d _Y -d_{Y'}] 
( \E ^{(\bullet)})
\in 
\underrightarrow{LM}  _{\Q,  \mathrm{isoc}, \text{pot-}\Sigma}(Y',\PP'/\W)$.
We can suppose $Y$ and $Y'$ integral. 
By definition,  there exists
a morphism of frames over $\W$ of the form 
$(d,c,g) \colon (Y'',X'', \PP'')\to (Y ,X , \PP)$
such that 
$c \colon X'' \to X$ is  a projective surjective generically finite and etale morphism, 
$Y''= c ^{-1} (Y)$,
$X''$ is $l$-smooth, 
$X'' \setminus Y''$ is a simple normal crossing divisor of $X''$ 
and such that $\R \underline{\Gamma} ^{\dag} _{Y''} g ^{!}(\E)
\in 
\underrightarrow{LM}  _{\Q,  \mathrm{isoc}, \Sigma} (Y'',\PP''/\W)$.
Replacing $\PP''$ by $\PP''\times \PP$ if necessary, we can suppose 
$g$ smooth. 
Let  $Y _{1}$ be an irreducible component of $Y'' \times _{Y} Y'$, let  $X _1$ be the closure of $Y _1$ in $X ''\times _{X} X'$.
Using de Jong desingularization theorem (see \cite{dejong}), 
we get a morphism of frames over $\W$ of the form 
$(d',c',g') \colon (Y''',X''', \PP''')\to (Y _1,X _1 , \PP'' \times _{\PP} \PP ')$
such that 
$c' \colon X''' \to X _1$ is  a projective surjective generically finite and etale morphism, 
$Y'''= (c ') ^{-1} (Y_1)$,
$X'''$ is $l$-smooth, 
$X''' \setminus Y'''$ is a simple normal crossing divisor of $X'''$.
Let $\pi _1 \colon \PP'' \times _{\PP} \PP ' \to \PP''$,
$\pi _2 \colon \PP'' \times _{\PP} \PP ' \to \PP'$
be  the canonical projections. 
By using \ref{regu-pullback}, we get that 
$$\R \underline{\Gamma} ^{\dag} _{Y'''} (\pi _2 \circ g ') ^{ !}
\left (
\R \underline{\Gamma} ^{\dag} _{Y'} f ^{!} [d _Y -d_{Y'}] 
( \E ^{(\bullet)})
\right )
\riso 
\R \underline{\Gamma} ^{\dag} _{Y'''} ( \pi _1 \circ g ') ^{ !}[d _Y -d_{Y'}] 
(\R \underline{\Gamma} ^{\dag} _{Y''} g ^{!}
( \E ^{(\bullet)})
) 
\in 
\underrightarrow{LM}  _{\Q,  \mathrm{isoc}, \Sigma} (Y''',\PP'''/\W),$$
which yields that 
$\R \underline{\Gamma} ^{\dag} _{Y'} f ^{!} [d _Y -d_{Y'}] 
( \E ^{(\bullet)})
\in 
\underrightarrow{LM}  _{\Q,  \mathrm{isoc}, \text{pot-}\Sigma}(Y',\PP'/\W)$.

Since the functor
$\R \underline{\Gamma} ^{\dag} _{Y'} f ^{!} [d _Y -d_{Y'}] 
\colon
\underrightarrow{LD} ^{\mathrm{b}} _{\Q,  \textrm{h-isoc}}(Y ,\PP/\W) 
\to
\underrightarrow{LD} ^{\mathrm{b}} _{\Q,  \textrm{h-isoc}}(Y ',\PP '/\W)$
is t-exact, then so is 
\ref{f!isocpotSigma},  which completes the proof of 1).

Using \ref{regu-pullback}.2, \ref{QUserreHol}.1, 
(resp. \ref{regu-pullback}.2, \ref{QUserreHol}.4)
by proceeding similarly to the proof of  the part 1), we check that 
the category $ \underrightarrow{LD} ^{\mathrm{b}} _{\Q,  \mathrm{isoc}, \mathrm{pot}\textrm{-}\Sigma}  (Y,\PP/\W)$
is a triangle subcategory of
$\smash{\underrightarrow{LD}} ^{\mathrm{b}} _{\Q,\mathrm{h}} ( Y,\PP/\W)$, 
and that part 3) is valid.

The stability under direct factors and under base change
are respectively a consequence of \ref{QUserreHol}.2
and \ref{QUserreHol}.3.

\end{proof}

\begin{dfn}
\label{dfn-potSigma}
Let $\mathbb{Y}:= (Y,X)$ be a couple over $\W$. Choose a frame $(Y, X,\PP)$ over $\W$. 

\begin{enumerate}
\item Let 
$\E ^{(\bullet)} 
\in 
\underrightarrow{LD} ^{\mathrm{b}} _{\Q, \mathrm{h}} (\smash{\widehat{\D}} _{\PP} ^{(\bullet)})$.
We say that $\E ^{(\bullet)}$ has ``potentially $\Sigma$-unipotent monodromy'' if
there exist a smooth stratification (see Definition \cite[2.2.1]{Abe-Caro-weights})
$(P _0, \dots, P _r)$ of $P$
such that
$\R \underline{\Gamma} ^{\dag} _{P _i} (\E ^{(\bullet)}) \in  
 \underrightarrow{LD} ^{\mathrm{b}} _{\Q,  \mathrm{isoc}, \text{pot-}\Sigma}(P _i,\PP/\W)$.
We denote by
$\underrightarrow{LD} ^{\mathrm{b}} _{\Q, \text{pot-}\Sigma} (\smash{\widehat{\D}} _{\PP} ^{(\bullet)})$
the full subcategory of 
$\underrightarrow{LD} ^{\mathrm{b}} _{\Q, \mathrm{h}} (\smash{\widehat{\D}} _{\PP} ^{(\bullet)})$
whose objects have potentially $\Sigma$-unipotent monodromy.

\item We denote by 
$\underrightarrow{LD} ^{\mathrm{b}} _{\Q, \textrm{pot-}\Sigma} (Y,\PP/\W)$
or by
$\underrightarrow{LD} ^{\mathrm{b}} _{\Q, \textrm{pot-}\Sigma} (\mathbb{Y}/\W)$
the full category of 
$\underrightarrow{LD} ^{\mathrm{b}} _{\Q, \text{pot-}\Sigma} (\smash{\widehat{\D}} _{\PP} ^{(\bullet)})$
of objects $\E$ such that there exists an isomorphism of the form 
$\E \riso \R \underline{\Gamma} ^\dag _{Y} (\E)$.
From \ref{IsocpotSigma}.3, 
we check that $\underrightarrow{LD} ^{\mathrm{b}} _{\Q, \textrm{pot-}\Sigma} (\mathbb{Y}/\W)$
does not depend on the choice of the frame enclosing $\mathbb{Y}$, which justifies the notation.

\item We denote 
$D ^{\mathrm{b}}  _{\textrm{pot-}\Sigma}  (\mathbb{Y}/\W)$ by
the essential image of the functor
$\underrightarrow{LD} ^{\mathrm{b}} _{\Q,\textrm{pot-} \Sigma} (\mathbb{Y}/\W)
\to 
D ^{\mathrm{b}}  _{\mathrm{coh}} (\smash{\D} ^\dag _{\PP\Q})$
induced by \ref{limeqcat}.
We say that
$D ^{\mathrm{b}}  _{\text{pot-}\Sigma}  (\mathbb{Y}/\W)$ is the full subcategory of 
$D^{\mathrm{b}}_{\mathrm{h}}(\mathbb{Y}/\W)$
of objects having ``potentially $\Sigma$-unipotent monodromy''.
\item When $\Sigma = 0$, 
we denote respectively  
$\underrightarrow{LD} ^{\mathrm{b}} _{\Q, \textrm{pot-}\Sigma} (\mathbb{Y}/\W)$ 
and
$D ^{\mathrm{b}}  _{\textrm{pot-}\Sigma}  (\mathbb{Y}/\W)$
by
$\smash{\underrightarrow{LD}} ^{\mathrm{b}} _{\Q,\text{u}} (\mathbb{Y}/\W)$ and
$D^{\mathrm{b}}_{u}(\mathbb{Y}/\W)$.

\item We get some data of coefficients 
$\smash{\underrightarrow{LD}} ^{\mathrm{b}} _{\Q,\textrm{pot-}\Sigma}$,
$\smash{\underrightarrow{LD}} ^{\mathrm{b}} _{\Q,\mathrm{u}}$
defined by posing 
$\smash{\underrightarrow{LD}} ^{\mathrm{b}} _{\Q,\textrm{pot-}\Sigma} (\PP):= 
\smash{\underrightarrow{LD}} ^{\mathrm{b}} _{\Q,\textrm{pot-}\Sigma} ( \smash{\widehat{\D}} _{\PP} ^{(\bullet)})$,
$\smash{\underrightarrow{LD}} ^{\mathrm{b}} _{\Q,\mathrm{u}} (\PP):= 
\smash{\underrightarrow{LD}} ^{\mathrm{b}} _{\Q,\mathrm{u}} ( \smash{\widehat{\D}} _{\PP} ^{(\bullet)})$.

\end{enumerate}
\end{dfn}

\begin{lem}
\label{quj+}
Let $(Y, X,\PP)$ be a frame over $\W$. 
For any
$\E ^{(\bullet)}
\in \underrightarrow{LD} ^{\mathrm{b}} _{\Q, \mathrm{pot}\textrm{-}\Sigma} (\smash{\widehat{\D}} _{\PP} ^{(\bullet)})$,
we have 
$\R \underline{\Gamma} ^{\dag} _{Y} (\E ^{(\bullet)}) 
\in  
\underrightarrow{LD} ^{\mathrm{b}} _{\Q, \textrm{pot-}\Sigma} (Y,\PP/\W)$.
For any
$\G ^{(\bullet)} \in \underrightarrow{LD} ^{\mathrm{b}} _{\Q, \mathrm{h}} (Y,\PP/\W)$, 
the property 
$\G ^{(\bullet)} 
\in \underrightarrow{LD} ^{\mathrm{b}} _{\Q, \mathrm{pot}\textrm{-}\Sigma} (\smash{\widehat{\D}} _{\PP} ^{(\bullet)})$
is equivalent to the one that 
there exists a smooth stratification 
$(Y _0, \dots, Y _r)$ of $Y$
such that
$\R \underline{\Gamma} ^{\dag} _{Y _i} (\E ^{(\bullet)}) \in  
 \underrightarrow{LD} ^{\mathrm{b}} _{\Q,  \mathrm{isoc}, \mathrm{pot}\textrm{-}\Sigma}(P _i,\PP/\W)$.
\end{lem}

\begin{proof}
Let $\E ^{(\bullet)}
\in \underrightarrow{LD} ^{\mathrm{b}} _{\Q, \text{pot-}\Sigma} (\smash{\widehat{\D}} _{\PP} ^{(\bullet)})$.
Let 
$(P _1, \dots, P _r)$ 
be a smooth stratification
of 
the special fiber of $\PP$ 
such that
$\R \underline{\Gamma} ^{\dag} _{P _i} (\E ^{(\bullet)}) 
\in  
 \underrightarrow{LD} ^{\mathrm{b}} _{\Q,  \mathrm{isoc}, \text{pot-}\Sigma}(P _i,\PP/\W)$.
For each $i= 1,\dots, r$, choose
a smooth stratification
$(Y _{i,1}, \dots, Y _{i, r_i})$ of $Y \cap P _i$.
Set  $J = \{ (i, j _i) \; ;\; 1\leq i\leq r ,\ 1 \leq j _i \leq r _i\}$.
By ordering the set $J$ with the lexicographic order, 
we get a smooth stratification 
$( Y _{i, j _i}) _{(i, j _i)\in J}$ 
of $Y$. 
Let $\overline{Y}$ be the closure of $Y$ in $P$. 
Set $U:= P \setminus \overline{Y}$.
Then $U$ is an open subscheme of $P$ such that $Y$ is a open subscheme of $P \setminus U$. 
Hence, we get a smooth stratification $(U, ( Y _{i, j _i}) _{(i, j _i)\in J}, \overline{Y}\setminus Y)$ of $P$. 
We have $\R \underline{\Gamma} ^{\dag} _{U} (\E ^{(\bullet)}) =0$
and
$\R \underline{\Gamma} ^{\dag} _{ \overline{Y}\setminus Y} (\E ^{(\bullet)}) =0$. 
Moreover, using \ref{f!isocpotSigma}, 
since 
$\R \underline{\Gamma} ^{\dag} _{P _i } (\E ^{(\bullet)}) 
\in  
 \underrightarrow{LD} ^{\mathrm{b}} _{\Q,  \mathrm{isoc}, \text{pot-}\Sigma}(P _i,\PP/\W)$, 
since 
$Y _{i, j _i} \subset  P _i$, 
then 
we get
$\R \underline{\Gamma} ^{\dag} _{Y _{i, j _i} } (\E ^{(\bullet)}) 
\in  
 \underrightarrow{LD} ^{\mathrm{b}} _{\Q,  \mathrm{isoc}, \text{pot-}\Sigma}(Y _{i, j _i},\PP/\W)$.
We check the second part of the Lemma in the same way. 
\end{proof}

\begin{prop}
\label{trisub}
The data of coefficients 
$\smash{\underrightarrow{LD}} ^{\mathrm{b}} _{\Q,\mathrm{pot}\textrm{-}\Sigma}$
is stable by devissages, direct factors and base change. 
\end{prop}

\begin{proof}
Let $\PP$ be a  smooth formal scheme over $\W$. 
Let $\E ' \to \E \to \E '' \to \E ' [1]$ be an exact triangle 
of $D^{\mathrm{b}}_{\mathrm{h}}(\PP/\W)$
with 
$\E' ,\E'' \in D^{\mathrm{b}}_{\text{pot-}\Sigma}(\PP/\W)$.
Let 
$(P _1, \dots, P _r)$ 
be a smooth stratification
of 
the special fiber of $\PP$ 
such that
$\R \underline{\Gamma} ^{\dag} _{P _i} (\E ^{\prime (\bullet)}) 
\in  
 \underrightarrow{LD} ^{\mathrm{b}} _{\Q,  \mathrm{isoc}, \text{pot-}\Sigma}(P _i,\PP/\W)$.
For each $i= 1,\dots, r$, 
following \ref{quj+}, 
we have $\R \underline{\Gamma} ^{\dag} _{P _i} (\E ^{\prime \prime (\bullet)}) 
\in  
\underrightarrow{LD} ^{\mathrm{b}} _{\Q, \text{pot-}\Sigma}(P _i,\PP/\W)$.
Hence, there exists 
a smooth stratification
$(Y _{i,1}, \dots, Y _{i, r_i})$ of $P _i$
such that 
$\R \underline{\Gamma} ^{\dag} _{Y _{i, j _i}} (\E ^{\prime \prime (\bullet)}) 
\in  
\underrightarrow{LD} ^{\mathrm{b}} _{\Q, \mathrm{isoc}, \text{pot-}\Sigma}(Y _{i, j _i},\PP/\W)$.
Set  $J = \{ (i, j _i) \; ;\; 1\leq i\leq r ,\ 1 \leq j _i \leq r _i\}$.
By ordering the set $J$ with the lexicographic order, 
we get a smooth stratification 
$( Y _{i, j _i}) _{(i, j _i)\in J}$ 
of $P$. 
On the other hand, using \ref{f!isocpotSigma}, 
we get 
$\R \underline{\Gamma} ^{\dag} _{Y _{i, j _i}} (\E ^{\prime  (\bullet)}) 
\in  
\underrightarrow{LD} ^{\mathrm{b}} _{\Q, \mathrm{isoc}, \text{pot-}\Sigma}(Y _{i, j _i},\PP/\W)$.
From \ref{prepot-regu-pullback}.2, 
$\underrightarrow{LD} ^{\mathrm{b}} _{\Q, \mathrm{isoc}, \text{pot-}\Sigma}(Y _{i, j _i},\PP/\W)$ is a triangled subcategory of 
$\underrightarrow{LD} ^{\mathrm{b}} _{\Q, \text{h}}(Y _{i, j _i},\PP/\W)$. 
Hence, 
we get
$\R \underline{\Gamma} ^{\dag} _{Y _{i, j _i}} (\E ^{ (\bullet)}) 
\in  
\underrightarrow{LD} ^{\mathrm{b}} _{\Q, \mathrm{isoc}, \text{pot-}\Sigma}(Y _{i, j _i},\PP/\W)$,
which gives the stability of 
$\smash{\underrightarrow{LD}} ^{\mathrm{b}} _{\Q,\textrm{pot-}\Sigma}$
by devissages.
The rest of the proposition is a consequence of 
\ref{prepot-regu-pullback}.2.
\end{proof}

\begin{prop}
\label{pot-regu-pullback}
Let $\theta = (b,a,f) \colon (Y',X',\PP')\to (Y,X,\PP)$ be a morphism of frames over $\W$.
We have the factorization
\begin{equation}
\label{f!isocpotSigmaGen}
\theta ^{!}=\R \underline{\Gamma} ^{\dag} _{Y'} f ^{!} \colon
 \underrightarrow{LD} ^{\mathrm{b}} _{\Q, \mathrm{pot}\textrm{-}\Sigma}(Y,\PP /\W) 
\to
\underrightarrow{LD} ^{\mathrm{b}} _{\Q, \mathrm{pot}\textrm{-}\Sigma}(Y',\PP'/\W),
\end{equation}
and a similar one by replacing 
``$\underrightarrow{LD} _{\Q} $''
by ``$D$''.
\end{prop}

\begin{proof}
Let 
$\E ^{(\bullet)}
\in 
 \underrightarrow{LD} ^{\mathrm{b}} _{\Q, \mathrm{pot}\textrm{-}\Sigma}(Y,\PP /\W) $.
Using  \ref{trisub},
we reduce by devissage 
to the case where $Y$ is smooth and 
$\E ^{(\bullet)}
\in 
\underrightarrow{LD} ^{\mathrm{b}} _{\Q, \mathrm{isoc}, \text{pot-}\Sigma}(Y,\PP /\W) $.
Let 
$(Y ' _1, \dots, Y' _r)$ 
be a smooth stratification
of $Y'$. 
Then,  from 
\ref{f!isocpotSigma}, 
we get 
$\R \underline{\Gamma} ^{\dag} _{Y' _i} f ^{!} 
(\E ^{(\bullet)})
\in 
\underrightarrow{LD} ^{\mathrm{b}} _{\Q, \mathrm{isoc}, \text{pot-}\Sigma}(Y ' _i,\PP '/\W)$
for any integer $i$.
Hence, 
$\R \underline{\Gamma} ^{\dag} _{Y'} f ^{!} 
(\E ^{(\bullet)})
\in 
\underrightarrow{LD} ^{\mathrm{b}} _{\Q, \mathrm{pot}\textrm{-}\Sigma}(Y',\PP'/\W)$.
\end{proof}

\begin{prop}
\label{tens-potSigma}
The data of coefficients
$\smash{\underrightarrow{LD}} ^{\mathrm{b}} _{\Q,\mathrm{pot}\text{-}\Sigma}$
is stable under tensor products.
\end{prop}

\begin{proof}
Using \ref{trisub}, we get by devissage the stability from
\ref{prepot-regu-pullback}.1 and \ref{prepot-regu-pullback}.3.
\end{proof}

\begin{prop}
\label{finite-et-pushforwar}
Let 
$(b,a,f) \colon (Y',X',\PP')\to (Y,X,\PP)$ be a complete morphism of frames
such that $b  \colon Y' \to Y$ is finite, étale, and surjective.
Let $\E ^{(\bullet)}\in  \underrightarrow{LD} ^{\mathrm{b}} _{\Q, \mathrm{h}}(Y,\PP /\W)$, 
$\E ^{\prime(\bullet)}\in  \underrightarrow{LD} ^{\mathrm{b}} _{\Q, \mathrm{h}}(Y',\PP '/\W)$, 
\begin{enumerate}

\item Suppose $Y$ is smooth. 
We have 
$\E ^{(\bullet)}\in  \underrightarrow{LD} ^{\mathrm{b}} _{\Q, \mathrm{isoc},\mathrm{pot}\textrm{-}\Sigma}(Y,\PP /\W)$
if and only if 
$\R \underline{\Gamma} ^{\dag} _{Y'}f ^! (\E ^{ (\bullet)})
\in  
\underrightarrow{LD} ^{\mathrm{b}} _{\Q,\mathrm{isoc}, \mathrm{pot}\textrm{-}\Sigma}(Y',\PP' /\W)$.
We have 
$\E ^{\prime (\bullet)}\in  \underrightarrow{LD} ^{\mathrm{b}} _{\Q,\mathrm{isoc}, \mathrm{pot}\textrm{-}\Sigma}(Y',\PP '/\W)$
if and only if 
$f _{+} (\E ^{\prime(\bullet)})
\in  \underrightarrow{LD} ^{\mathrm{b}} _{\Q, \mathrm{isoc},\mathrm{pot}\textrm{-}\Sigma}(Y,\PP /\W)$.

\item  We have 
$\E ^{(\bullet)}\in  \underrightarrow{LD} ^{\mathrm{b}} _{\Q, \mathrm{pot}\textrm{-}\Sigma}(Y,\PP /\W)$
if and only if 
$\R \underline{\Gamma} ^{\dag} _{Y'}f ^! (\E ^{(\bullet)})
\in  \underrightarrow{LD} ^{\mathrm{b}} _{\Q, \mathrm{pot}\textrm{-}\Sigma}(Y',\PP' /\W)$.
Moreover, we have 
$\E ^{\prime(\bullet)}\in  \underrightarrow{LD} ^{\mathrm{b}} _{\Q, \mathrm{pot}\textrm{-}\Sigma}(Y',\PP' /\W)$
if and only if 
$f _{+} (\E ^{\prime(\bullet)})
\in  \underrightarrow{LD} ^{\mathrm{b}} _{\Q, \mathrm{pot}\textrm{-}\Sigma}(Y,\PP /\W)$.

\end{enumerate}
\end{prop}

\begin{proof}
Let us check the part 1) of the Proposition. 
From \ref{prepot-regu-pullback}.1, if 
$\E ^{(\bullet)}\in  \underrightarrow{LD} ^{\mathrm{b}} _{\Q, \mathrm{isoc},\mathrm{pot}\textrm{-}\Sigma}(Y,\PP /\W)$
then
$\R \underline{\Gamma} ^{\dag} _{Y'} f ^! (\E ^{ (\bullet)})
\in  \underrightarrow{LD} ^{\mathrm{b}} _{\Q,\mathrm{isoc}, \mathrm{pot}\textrm{-}\Sigma}(Y',\PP' /\W)$.
Conversely, suppose $\R \underline{\Gamma} ^{\dag} _{Y'} f ^! (\E ^{ (\bullet)})
\in  \underrightarrow{LD} ^{\mathrm{b}} _{\Q,\mathrm{isoc}, \mathrm{pot}\textrm{-}\Sigma}(Y',\PP' /\W)$.
Since $b$ is finite and étale, we get the functor
$f _+
\colon 
\underrightarrow{LD} ^{\mathrm{b}} _{\Q,  \textrm{h-isoc}}(Y',\PP'/\W)
\to 
\underrightarrow{LD} ^{\mathrm{b}} _{\Q,  \textrm{h-isoc}}(Y,\PP/\W)$
(see Notation \ref{IsocSigmapre}.3).
Since 
$\E ^{(\bullet)}$ is a direct factor of 
$f _+ \R \underline{\Gamma} ^{\dag} _{Y'} f ^! \E ^{(\bullet)}$,
then
$\E ^{(\bullet)}
\in 
\underrightarrow{LD} ^{\mathrm{b}} _{\Q,  \textrm{h-isoc}}(Y,\PP/\W)$.
Recalling $\R \underline{\Gamma} ^{\dag} _{Y'} f ^! (\E ^{ (\bullet)})
\in  \underrightarrow{LD} ^{\mathrm{b}} _{\Q,\mathrm{isoc}, \mathrm{pot}\textrm{-}\Sigma}(Y',\PP' /\W)$, 
this yields almost by definition
$\E ^{(\bullet)}\in  \underrightarrow{LD} ^{\mathrm{b}} _{\Q, \mathrm{isoc},\mathrm{pot}\textrm{-}\Sigma}(Y,\PP /\W)$.

Suppose $f _{+} (\E ^{\prime(\bullet)})
\in  \underrightarrow{LD} ^{\mathrm{b}} _{\Q, \mathrm{isoc},\mathrm{pot}\textrm{-}\Sigma}(Y,\PP /\W)$.
Then from the first part 
$\R \underline{\Gamma} ^{\dag} _{Y'} f ^{!}f _{+} (\E ^{\prime (\bullet)})
\in 
\underrightarrow{LD} ^{\mathrm{b}} _{\Q,\mathrm{isoc}, \mathrm{pot}\textrm{-}\Sigma}(Y',\PP' /\W)$.
Since $Y' \to Y$ is finite étale, then 
$Y'$ is an open component connected component of $Y' \times _Y Y'$ (see  \cite[I.3.12]{milne}).
Hence, using the base change isomorphism \ref{basechange},
we check that 
$\E ^{\prime (\bullet)}$ is a direct factor of 
$\R \underline{\Gamma} ^{\dag} _{Y'} f ^{!}f _{+} (\E ^{\prime (\bullet)})$. 
By using \ref{QUserreHol}.2,
this yields 
$\E ^{\prime (\bullet)}
\in 
\underrightarrow{LD} ^{\mathrm{b}} _{\Q,\mathrm{isoc}, \mathrm{pot}\textrm{-}\Sigma}(Y',\PP' /\W)$.
Conversely, suppose
$\E ^{\prime (\bullet)}
\in 
\underrightarrow{LD} ^{\mathrm{b}} _{\Q,\mathrm{isoc}, \mathrm{pot}\textrm{-}\Sigma}(Y',\PP' /\W)$.
There exists a Galois finite, étale morphism
$b'\colon Y'' \to Y$ that factors through $b \colon Y' \to Y$ (see the beginning of \cite[I.5]{milne}
or \cite[V.4.g) and V.4.1 and V.7]{sga1}). 
Denoting by $b'' \colon Y'' \to Y'$ this factorization we get $b'= b \circ b''$. 
Since $b''$ is in particular projective, 
we get a morphism of frames over $\W$ of the form
$ (b'',a'',f'') \colon (Y'',X'',\PP'')\to (Y',X',\PP')$ 
where $a''$ is projective, $f''$ is projective and smooth,
$Y''$ is dense in $X''$.
Since 
$\E ^{\prime (\bullet)}$
 is a direct factor of 
 $f'' _{+} \R \underline{\Gamma} ^{\dag} _{Y''} f ^{\prime \prime !}  (\E ^{\prime (\bullet)})$, 
 then 
$f _{+} (\E ^{\prime (\bullet)})$ is a direct factor of 
$f _+   f'' _{+} \R \underline{\Gamma} ^{\dag} _{Y''} f ^{\prime \prime !} (\E ^{\prime (\bullet)})$. 
Hence, by using \ref{QUserreHol}.2 and \ref{regu-pullback}.2,
we can suppose that $f$ is a Galois morphism.
In that case, $Y' \times _Y Y'= \sqcup _{\sigma \in G} Y ' _\sigma$, where $G = \mathrm{Aut} _Y (Y')$
and $Y ' _\sigma$ is a copy of $Y '$.
Hence, using the base change isomorphism \ref{basechange},
we check that 
$\R \underline{\Gamma} ^{\dag} _{Y'} f ^{!}f _{+} (\E ^{\prime (\bullet)})
= 
\oplus _{\sigma \in G} \E ^{\prime (\bullet)} _\sigma$,
where
$\E ^{\prime (\bullet)} _\sigma$ means a copy of 
$\E ^{\prime (\bullet)}$. 
Hence, 
$\R \underline{\Gamma} ^{\dag} _{Y'} f ^{!}f _{+} (\E ^{\prime (\bullet)})
\in 
\underrightarrow{LD} ^{\mathrm{b}} _{\Q,\mathrm{isoc}, \mathrm{pot}\textrm{-}\Sigma}(Y',\PP' /\W)$.
Then from the first part 
$f _{+} (\E ^{\prime (\bullet)})
\in 
\underrightarrow{LD} ^{\mathrm{b}} _{\Q,\mathrm{isoc}, \mathrm{pot}\textrm{-}\Sigma}(Y,\PP /\W)$.

Let us check the part 2) of the Proposition. 
Moreover, since 
$\E ^{ (\bullet)}$
 is a direct factor of 
 $f _{+} \R \underline{\Gamma} ^{\dag} _{Y} f ^{\prime  !}  (\E ^{ (\bullet)})$, 
this yields that if 
$\R \underline{\Gamma} ^{\dag} _{Y'}f ^! (\E ^{ (\bullet)})
\in  
\underrightarrow{LD} ^{\mathrm{b}} _{\Q,\mathrm{pot}\textrm{-}\Sigma}(Y',\PP' /\W)$
then 
$\E ^{(\bullet)}\in  \underrightarrow{LD} ^{\mathrm{b}} _{\Q, \mathrm{pot}\textrm{-}\Sigma}(Y,\PP /\W)$.
From \ref{pot-regu-pullback}, the converse is known. 
Finally, we proceed as in the part 1) of the proof 
to check that 
$\E ^{\prime(\bullet)}\in  \underrightarrow{LD} ^{\mathrm{b}} _{\Q, \mathrm{pot}\textrm{-}\Sigma}(Y',\PP' /\W)$
if and only if 
$f _{+} (\E ^{\prime(\bullet)})
\in  \underrightarrow{LD} ^{\mathrm{b}} _{\Q, \mathrm{pot}\textrm{-}\Sigma}(Y,\PP /\W)$.
\end{proof}

\begin{lem}
\label{potSigmavsduals}
The data of coefficients 
$\smash{\underrightarrow{LD}} ^{\mathrm{b}} _{\Q, \text{pot-}\Sigma} $
is almost stable by dual functors. 
\end{lem}

\begin{proof}
Let $\mathfrak{C}$ be a data of coefficients containing 
$\smash{\underrightarrow{LD}} ^{\mathrm{b}} _{\Q, \text{pot-}\Sigma} $ and 
stable under  devissages, under direct factors and under realizable pushforwards. 
Let $\W$ be an object of $\mathrm{DVR}  (\V)$, 
$\PP$ be a  smooth formal scheme over $\W$.
Let $\E ^{(\bullet)} \in 
\smash{\underrightarrow{LD}} ^{\mathrm{b}} _{\Q, \text{pot-}\Sigma} (\PP)$.
We have to check that $\DD _{\PP} (\E ^{(\bullet)}) \in  \mathfrak{C} (\PP)$.
Since $\mathfrak{C}$ is stable by devissages, we can suppose
there exists an irreducible smooth subvariety $Y$ of $P$ 
such that
$\E ^{(\bullet)} \in   \underrightarrow{LD} ^{\mathrm{b}} _{\Q,  \mathrm{isoc}, \text{pot-}\Sigma}(Y,\PP/\W)$.
Again by devissage, we can suppose that 
$\E := \underrightarrow{\lim}~ \E ^{(\bullet)} \in \mathrm{Isoc} ^{\dag \dag} _{\text{pot-}\Sigma}(Y/\W)$,
where
$\underrightarrow{\lim}$
is the functor of \ref{limeqcat}. 
We denote by $\overline{Y}$ the closure of $Y$ in $P$.
There exists 
a morphism of frames
$\theta = (b,a,f)\colon 
(Y', \overline{Y'}, \PP')
\to 
(Y, \overline{Y}, \PP)$
such that $f$ is proper, $a$ is generically finite and etale,
$\overline{Y'}$ is smooth, 
$Y '= a ^{-1} (Y)$,
$\overline{Y'} \setminus Y '$ is a strict normal crossing divisor in $\overline{Y'}$
and 
$\theta ^{!}(\E )= 
\R \underline{\Gamma} ^\dag _{Y'} f ^{!} (\E )  
\in \mathrm{Isoc} ^{\dag \dag} _{\Sigma}(Y', \overline{Y'}/\W)$. 
From \ref{dual}, we get 
$\DD _{\PP '} \theta ^{!}(\E) \in 
\smash{\underrightarrow{LD}} ^{\mathrm{b}} _{\Q, \Sigma} (\PP')
\subset
\smash{\underrightarrow{LD}} ^{\mathrm{b}} _{\Q, \text{pot-}\Sigma} (\PP')
\subset \mathfrak{C} (\PP')$.
Since $\E $ is a direct factor of $\theta _+ \circ \theta ^{!} (\E)$ (where we set 
$\theta _+ := f _+$), 
by using the relative duality isomorphism (see \ref{rel-dual-isom}),
we check that 
$\DD _{\PP } (\E )$
is a direct factor of 
$\theta _+ \circ \DD _{\PP '}   \theta ^{!} (\E )$.
Since $\mathfrak{C}$ is stable by direct factors and realizable (in fact proper would have been sufficient here) pushforwards, 
this yields
$\DD _{\PP } (\E )
\in \mathfrak{C} (\PP)$.

\end{proof}

\subsection{Quasi-unipotence, coefficients satisfying semi-stable reduction property}

\begin{dfn}
\label{dfnqu}
From  
\ref{dfn-potSigma}.2,
\ref{trisub},
\ref{pot-regu-pullback},
\ref{tens-potSigma},
\ref{potSigmavsduals}, 
the data of coefficients $\smash{\underrightarrow{LD}} ^{\mathrm{b}} _{\Q,\text{pot-}\Sigma}$ 
satisfies the hypotheses of \ref{dfnquprop} concerning the data $\mathfrak{D}$.
Since $\underrightarrow{LD} ^{\mathrm{b}} _{\Q, \mathrm{ovcoh}}$ is 
stable by direct factors, devissages, extraordinary pullbacks,
realizable pushforwards (see \ref{ovcoh-invim} and \ref{stab-propersupp}), then we can define
$\smash{\underrightarrow{LD}} ^{\mathrm{b}} _{\Q,\text{q-}\Sigma}   := 
T _{\mathrm{min}}(\smash{\underrightarrow{LD}} ^{\mathrm{b}} _{\Q,\text{pot-}\Sigma},
\underrightarrow{LD} ^{\mathrm{b}} _{\Q, \mathrm{ovcoh}})$.
When $\Sigma = 0$, we put 
$\smash{\underrightarrow{LD}} ^{\mathrm{b}} _{\Q,\text{qu}}  :=\smash{\underrightarrow{LD}} ^{\mathrm{b}} _{\Q,\text{q-}\Sigma}  $.
The objects of the data of coefficients $\smash{\underrightarrow{LD}} ^{\mathrm{b}} _{\Q,\text{q-}\Sigma}$ (resp. 
$\smash{\underrightarrow{LD}} ^{\mathrm{b}} _{\Q,\text{qu}}$) are called 
``quasi-$\Sigma$-unipotent'' (resp. ``quasi-unipotent'').
We also say that they have ``quasi-$\Sigma$-unipotent monodromy''.

\end{dfn}

\begin{ntn}
\label{dfnst}
We define the data of coefficients with potentially Frobenius structure $\smash{\underrightarrow{LD}} ^{\mathrm{b}} _{\Q,F}$ as follows. 
Let $\W$ be an object of $\mathrm{DVR}  (\V,\sigma)$, 
$\X$ be a  smooth formal scheme over $\W$. 
The category $\smash{\underrightarrow{LD}} ^{\mathrm{b}} _{\Q,F}( \X)$ is by definition 
the full subcategory of $\underrightarrow{LD} ^{\mathrm{b}} _{\Q, \mathrm{ovcoh}} (\smash{\widehat{\D}} _{\X} ^{(\bullet)})$ whose 
objects are equal to 
the essential image of the canonical functor
which forgets Frobenius structures :
$F\text{-}\underrightarrow{LD} ^{\mathrm{b}} _{\Q, \mathrm{ovcoh}} (\smash{\widehat{\D}} _{\X} ^{(\bullet)})
\to 
\underrightarrow{LD} ^{\mathrm{b}} _{\Q, \mathrm{ovcoh}} (\smash{\widehat{\D}} _{\X} ^{(\bullet)})$. 

\end{ntn}

\begin{rem}
Let 
$\W$ 
be an object of $\mathrm{DVR}  (\V,\sigma)$, 
$\X$ be a  smooth formal scheme over $\W$. 
From Kedlaya's semistable reduction theorem (see \ref{remIsocSigma2}), we remark that 
$\smash{\underrightarrow{LD}} ^{\mathrm{b}} _{\Q,F}(\X)$ is contained in
$\smash{\underrightarrow{LD}} ^{\mathrm{b}} _{\Q,\text{u} } (\X)$ (recall the notation of \ref{dfn-potSigma}),
which is at its turn contained in 
$\smash{\underrightarrow{LD}} ^{\mathrm{b}} _{\Q,\text{h} } (\X)$.

\end{rem}

\begin{ntn}
\begin{enumerate}

\item We define by induction on $n\in \N$ the data of coefficients with potentially Frobenius structure
$\smash{\underrightarrow{LD}} ^{\mathrm{b}} _{\Q,\text{u} ^{+},n} $
as follows. 
For $n=0$, we put 
$\smash{\underrightarrow{LD}} ^{\mathrm{b}} _{\Q,\text{u} ^{+},0} 
=
\smash{\underrightarrow{LD}} ^{\mathrm{b}} _{\Q,\text{u}} $.
Suppose 
$\smash{\underrightarrow{LD}} ^{\mathrm{b}} _{\Q,\text{u} ^{+},n} $ constructed.
Let 
$\W$ be an object of $\mathrm{DVR}  (\V,\sigma)$, 
$\PP$ be a  smooth formal scheme over $\W$.
The category $\smash{\underrightarrow{LD}} ^{\mathrm{b}} _{\Q,\text{u} ^{+},n+1} (\PP)$
 is by definition the full subcategory of 
 $\smash{\underrightarrow{LD}} ^{\mathrm{b}} _{\Q,\text{u} ,n} (\PP)$
 of complexes $\E^{(\bullet) }$ 
satisfying  the following property: 
for any morphism $f \colon \X  \to \PP$ of smooth formal $\W$-schemes,
for any realizable morphism $g \colon \X  \to \Y '$ of smooth formal $\W$-schemes,
for any subscheme $Z$ of $X$ which is proper over $Y'$ (via $g$),
we have 
$g _{+} \R \underline{\Gamma} ^\dag _{Z} f ^!
(\E^{(\bullet) })
\in 
\smash{\underrightarrow{LD}} ^{\mathrm{b}} _{\Q,\text{u} ,n} (\Y')$.

\item We set $\smash{\underrightarrow{LD}} ^{\mathrm{b}} _{\Q,\text{u} ^{+}} 
:= \cap _{n\in\N}\smash{\underrightarrow{LD}} ^{\mathrm{b}} _{\Q,\text{u} ^{+},n} $.

\item The constructions of $T _{\mathrm{max}}$ and $T _{\mathrm{min}}$ in \ref{TminTmax}
are still valid if we restrict to data of coefficients with potentially Frobenius structure
instead of data of coefficients.
From \cite{caro-Tsuzuki}, the data 
$\smash{\underrightarrow{LD}} ^{\mathrm{b}} _{\Q,F}$
satisfies the hypotheses of \ref{dfnquprop} concerning the data $\mathfrak{D}$.
Moreover, the data 
$ \smash{\underrightarrow{LD}} ^{\mathrm{b}} _{\Q,\text{u} ^{+}} $
is 
stable by devissages, direct factors, extraordinary pullbacks  and realizable pushforwards. 
Hence, we get the following data of coefficients with potentially Frobenius structure by putting 
$\smash{\underrightarrow{LD}} ^{\mathrm{b}} _{\Q,\text{st}}:= 
T _{\mathrm{max}} (\smash{\underrightarrow{LD}} ^{\mathrm{b}} _{\Q,F}, \smash{\underrightarrow{LD}} ^{\mathrm{b}} _{\Q,\text{u} ^{+}} )$. 
\end{enumerate}

\end{ntn}

\begin{empt}
Let $\W$ be an object of $\mathrm{DVR}  (\V,\sigma)$ and $\PP$ be a  smooth formal scheme over $\W$.
We denote by 
$D ^{\mathrm{b}}  _{\text{q-}\Sigma} (\smash{\D} ^\dag _{\PP\Q})$
(resp. $D ^{\mathrm{b}}  _\text{qu} (\smash{\D} ^\dag _{\PP\Q})$
resp. $D ^{\mathrm{b}}  _\text{st} (\smash{\D} ^\dag _{\PP\Q})$) 
the essential image of 
$\smash{\underrightarrow{LD}} ^{\mathrm{b}} _{\Q,\text{q-}\Sigma} (\PP)$ 
(resp. 
$\smash{\underrightarrow{LD}} ^{\mathrm{b}} _{\Q,\text{qu}} (\PP)$, resp. 
$\smash{\underrightarrow{LD}} ^{\mathrm{b}} _{\Q,\text{st}} (\PP)$) 
by  the functor $\underrightarrow{\lim}$ of 
\ref{limeqcat}
\end{empt}

\begin{rem}
\begin{enumerate}

\item The data of coefficients $\smash{\underrightarrow{LD}} ^{\mathrm{b}} _{\Q,\text{st}}$
contains 
$\smash{\underrightarrow{LD}} ^{\mathrm{b}} _{\Q,F}$
and is contained in 
$\smash{\underrightarrow{LD}} ^{\mathrm{b}} _{\Q,\text{u}}$.
In particular, 
the isocrystals in 
$\smash{\underrightarrow{LD}} ^{\mathrm{b}} _{\Q,\text{st}}$
satisfy Kedlaya's semistable reduction theorem.

\item 
From  
\ref{trisub},
\ref{pot-regu-pullback},
\ref{tens-potSigma},
\ref{potSigmavsduals}, 
we remark that if the data of coefficients 
$\smash{\underrightarrow{LD}} ^{\mathrm{b}} _{\Q,\text{u}}$
is stable under pushforwards (which is unlikely but this remains an open question), 
then we get 
$\smash{\underrightarrow{LD}} ^{\mathrm{b}} _{\Q,\text{u}}
=
\smash{\underrightarrow{LD}} ^{\mathrm{b}} _{\Q,\text{u} ^{+}}
=
\smash{\underrightarrow{LD}} ^{\mathrm{b}} _{\Q,\text{st}}$.

\end{enumerate}

\end{rem}

\bibliographystyle{alpha}

\def\cprime{$'$}

\end{document}